\documentclass[11pt]{article}
\usepackage[normalem]{ulem} %%%\sout
\usepackage{enumerate}
\usepackage{enumitem}

\usepackage{amsfonts}
\usepackage{amssymb}
\usepackage{amsmath}
\usepackage{amsfonts}
\usepackage{multicol}
\usepackage{graphicx}
\usepackage{amsfonts}
\usepackage{multicol}
\usepackage{amssymb}
\usepackage{graphicx}
\usepackage{amsmath,amssymb}
\usepackage{xcolor}
\usepackage{epstopdf}
\usepackage[authoryear,round]{natbib}
\usepackage{multirow}
\usepackage{subfig}
\usepackage{float}
\usepackage{dsfont}

\usepackage{mathptmx}

\usepackage{pxfonts}
\usepackage[onehalfspacing]{setspace}
\newenvironment{proof}[1][Proof]{\noindent \textbf{#1.} }{\  \rule{0.5em}{0.5em}}
\RequirePackage[colorlinks,citecolor=blue,linkcolor=blue,urlcolor=blue,pagebackref]{hyperref}
\allowdisplaybreaks
\RequirePackage[left=1.8cm, right=1.8cm, top=2.0cm, bottom=2.0cm,headheight=12pt]{geometry}

\newtheorem{theorem}{Theorem}[section]

\newtheorem{corollary}[theorem]{Corollary}

\newtheorem{lemma}{Lemma}[section]
\newtheorem{remark}{Remark}[section]

\newtheorem{proposition}{Proposition}[section]

\newtheorem{assumption}{Assumption}

\graphicspath{{fig/}}

\newcommand{\N}{\mathbb{N}}
\newcommand{\Z}{\mathbb{Z}}
\newcommand{\R}{\mathbb{R}}

\newcommand{\E}{\mathbb{E}}
\renewcommand{\P}{\mathbb{P}}

\newcommand{\dd}{\,\mathrm{d}}
\newcommand{\Indi}{\mbox{\rm{1}}\hspace{-0.25em}\mbox{\rm{l}}}
\newcommand{\Cov}{\mathrm{Cov}}
\newcommand{\Var}{\mathrm{Var}}
\newcommand{\diag}{\mathrm{diag}}

\newcommand{\argmax}{\mathop{\rm arg~max}\limits}
\newcommand{\argmin}{\mathop{\rm arg~min}\limits}

\author{Tetsuya Takabatake
\thanks{Graduate School of Engineering Science, The University of Osaka, 1-3 Machikaneyama, Toyonaka, Osaka, Japan.
Email:~t.takabatake.es@osaka-u.ac.jp.
}
\and 
Jun Yu
\thanks{Department of Finance and Business Economics, Faculty of Business Administration, University of Macau, Avenida da Universidade, Taipa, Macao, China. Email: junyu@um.edu.mo.}
\and 
Chen Zhang
\thanks{Department of Finance and Business Economics, Faculty of Business Administration, University of Macau, Avenida da Universidade, Taipa, Macao, China. Email: chenzhang@um.edu.mo.}
}
\usepackage[onehalfspacing]{setspace}
\begin{document}

\title{\LARGE Optimal Estimation for General Gaussian Processes}

\maketitle

\begin{abstract}
This paper proposes a novel exact maximum likelihood (ML) estimation method for general Gaussian processes, where all parameters are estimated jointly. The exact ML estimator (MLE) is consistent and asymptotically normally distributed. We prove the local asymptotic normality (LAN) property of the sequence of statistical experiments for general Gaussian processes in the sense of Le Cam, thereby enabling optimal estimation and facilitating statistical inference. The results rely solely on the asymptotic behavior of the spectral density near zero, allowing them to be widely applied. The established optimality not only addresses the gap left by \cite{Adenstedt-1974}, who proposed an efficient but infeasible estimator for the long-run mean $\mu$, but also enables us to evaluate the finite-sample performance of the existing method— the commonly used plug-in MLE, in which the sample mean is substituted into the likelihood. Our simulation results show that the plug-in MLE performs nearly as well as the exact MLE, alleviating concerns that inefficient estimation of $\mu$ would compromise the efficiency of the remaining parameter estimates.
\end{abstract}

\section{Introduction}
Gaussian processes have been widely applied across a broad range of scientific and applied disciplines, including economics, finance, physics, hydrology, and telecommunications. One of their most extensively studied features is the long-memory property, which captures long-range dependence. The autoregressive fractionally integrated moving average (ARFIMA) model was independently introduced by \cite{Granger-1980} and \cite{Hosking-1981} to model this feature. In economics and finance,  long memory has been examined in a wide array of time series, including real output \citep{Diebold-Rudebusch-1989}, income \citep{Diebold-Rudebusch-1991}, stock returns \citep{Lo-1991, Liu-Jing-2018}, interest rates \citep{Shea-1991}, exchange rates \citep{Diebold-Husted-Rush-1991, Cheung-1993}, volatility \citep{Ding-Granger-Engle-1993, Andersen-Bollerslev-1997-b, Andersen-Bollerslev-Diebold-Labys-2003} and bubble detection \citep{Lui-Phillips-Yu-2024}. Various mechanisms have been proposed to explain the emergence of long memory, including cross-sectional aggregation \citep{Granger-1980, Abadir-2002}, regime switching \citep{Potter-1976, Diebold-Inoue-2001}, marginalization \citep{Chevillon-2018}, and network effects \citep{Schennach-2018}.

More recently, a rapidly growing strand of literature has focused on continuous-time Gaussian processes, which can characterize the local behavior and reproduce the rough sample paths observed in volatility and trading volume \citep{Gatheral-Jaisson-Rosenbaum-2018, Fukasawa-Takabatake-Westphal-2022, Wang-Xiao-Yu-2023-JoE, Bolko-Christensen-Pakkanen-Veliyev-2023,  Shi-Yu-Zhang-2024-b, Chong-Todorov-2025}. Two prominent models in this class are fractional Brownian motion (fBm)\citep{Mandelbrot-1965, mandelbrot1968, Gatheral-Jaisson-Rosenbaum-2018} and the fractional Ornstein–Uhlenbeck (fOU) process \citep{Cheridito-Kawaguchi-Maejima-2003, Wang-Xiao-Yu-2023-JoE}. When applied to data, fractional Gaussian noise (fGn, the first-order difference of fBm) and fOU—exhibit anti-persistence \citep{Gatheral-Jaisson-Rosenbaum-2018, Fukasawa-Takabatake-2019, Shi-Yu-Zhang-2024} and short memory \citep{Shi-Yu-Zhang-2024-b, Wang-Xiao-Yu-Zhang-2023}. Several studies have begun to investigate the micro-level origins of this roughness. For example, \cite{eleuch2018} show that in highly endogenous markets, rough volatility may arise from a large number of split orders, while \cite{Jusselin-Rosenbaum-2020} demonstrate that rough volatility emerges naturally under the no-arbitrage condition.

The ML estimation method of non-centered stationary discrete- and continuous-time Gaussian models with long memory, short memory, or anti-persistence, referred to as general Gaussian processes, is the focus of this paper. Since these memory properties are relevant across various applications, our goal is to develop an estimation method that does not impose prior restrictions on the memory type of the process. The choice of ML estimation is motivated by the practical need for accurate estimation of all model parameters, particularly when computing impulse response functions or performing forecasts. In such cases, suboptimal estimators---such as the semi-parametric methods of \cite{geweke1983, Robinson-1995-LWE, phillips2004, shimotsu2010}, the method of moments in \cite{Wang-Xiao-Yu-2023-JoE}, and the composite likelihood approach in \cite{Bennedsen-2024}---are not recommended. For example, \cite{corsi2009} criticized semi-parametric methods for producing significantly biased and inefficient estimates in forecasting applications with ARFIMA models. Moreover, although the ML and Whittle ML estimators are asymptotically equivalent, the ML estimation method generally demonstrates superior finite-sample performance \citep{Cheung-Diebold-1994}.\footnote{\cite{Rao-Yang-2021} propose new frequency-domain quasi-likelihoods that improve the finite-sample behavior of the Whittle ML estimator for short-memory Gaussian processes.}

Considerable progress has been made in developing ML estimation methods, extending from specific parametric models to general Gaussian processes based on discrete-time observations.\footnote{A parallel literature studies parameter estimation for continuous-time fractional models under continuous-record observations; see \cite{Kleptsyna2002}.} For example, \cite{Yajima-1985} established the consistency and asymptotic normality of the MLE for the ARFIMA$(0,d,0)$ model with $d \in (0, 0.5)$, representing the long-memory case. These results were subsequently extended to stationary Gaussian processes with long memory by \cite{Dahlhaus-1989, Dahlhaus-2006}, and further generalized to general Gaussian processes by \cite{Lieberman-2012}. However, these existing methods adopt a two-stage procedure in which $\mu$ is first estimated by the sample mean and then substituted into the likelihood, yielding a so-called plug-in MLE for the remaining parameters. Some implementations apply the MLE to demeaned data \citep{Tsai-Chan-2005, shi2022volatility}, which is essentially equivalent to the plug-in MLE procedure. Regarding the optimality of this procedure, the sample mean is clearly not the efficient estimator for $\mu$. While \cite{Dahlhaus-1989, Dahlhaus-2006} argued for the efficiency of the plug-in MLE by showing that its asymptotic covariance matrix equals the inverse of the Fisher information matrix, they did not explicitly establish the existence of a Crame\'r--Rao lower bound. \cite{Cohen-Gamboa-Lacaux-Loubes-2013} derived the LAN property for centered stationary Gaussian processes with  long memory, short memory, or anti-persistence, providing a minimax lower bound for general estimators; however, their results do not imply the asymptotic efficiency of plug-in MLE. Another concern is that the inefficiency in estimating $\mu$ may impair the finite-sample performance of the plug-in MLE. \cite{Cheung-Diebold-1994} showed that when the $\mu$ is unknown, the finite-sample performance of the MLE for the other parameters deteriorates, even though their asymptotic variances remain the same as in the known-mean case. To date, the problem of obtaining theoretically optimal estimators for all parameters for general Gaussian processes within the ML framework remains unresolved. Only one exception is \cite{Wang-Xiao-Yu-Zhang-2023}, who established the consistency and asymptotic normality of the MLE for all parameters—including $\mu$—in the fOU process. Nonetheless, their framework is model-specific and not readily applicable to other fractional models. Moreover, the efficiency of the MLE was not addressed.

We introduce a novel exact ML method, a term we adopt to distinguish it from the plug-in ML method, for general Gaussian processes, where all the parameters are estimated jointly.  We establish the consistency and asymptotic normality of the exact MLE. These results extend those in \cite{Wang-Xiao-Yu-Zhang-2023} from fOU to general Gaussian processes. Moreover, we establish the LAN property of the sequence of statistical experiments in the Le Cam sense. This result extends that in \cite{Cohen-Gamboa-Lacaux-Loubes-2013} from centered stationary Gaussian processes 
to  ``non-centered'' ones within the long span asymptotics, which is different from several extensions under the high-frequency asymptotics recently made by \cite{Brouste-Fukasawa-2018,Fukasawa-Takabatake-2019,Szymanski-2024,Szymanski-Takabatake-2023-Additive,Chong-Mies-2025}. The LAN property directly yields the efficiency of the exact MLE. The results rely solely on the asymptotic behavior of the spectral density near zero for a discrete record of observations, which allows for broad applicability. \footnote{To ensure our theoretical results are broadly applicable and consistent with discrete observations, we work with the spectral density corresponding to discrete-time data, regardless of whether the underlying process is continuous- or discrete-time. In TYZ (working paper), we demonstrate how to verify the necessary conditions for the discrete spectral density using the continuous-time spectral density for a wide range of continuous-time processes.}

To demonstrate the practical applicability of the proposed method, we conduct three Monte Carlo simulation studies in which our exact ML estimator is applied to three widely studied non-centered processes: the ARFIMA$(0,d,0)$ model, fractional Gaussian noise (fGn), and the fractional Ornstein–Uhlenbeck (fOU) process. Overall, our exact estimator for $\mu$ improves upon the performance of the sample mean. Regarding the performance of plug-in MLE, our simulation results show that the plug-in MLE performs nearly as well as the exact MLE, alleviating concerns that inefficient estimation of $\mu$ would compromise the efficiency of the remaining parameter estimates. In particular, for the ARFIMA$(0,d,0)$ model, the gain in efficiency for estimating $\mu$ aligns with the theoretical result of \cite{Adenstedt-1974}. We also conduct a forecasting horse race for realized volatility using the fOU process with
three alternative estimators: the exact MLE, the plug-in MLE, and the change-of-frequency (CoF) estimator by \cite{Wang-Xiao-Yu-2023-JoE}. As expected, the exact MLE delivers the best forecasting performance, followed by the plug-in MLE, which performs satisfactorily, and then the CoF estimator.

To sum up, we contribute to the literature in the following aspects. 
First, we propose a novel exact ML estimation method for all parameters in a general stationary Gaussian process, establishing its consistency and asymptotic normality.  Second, we prove the LAN property of the sequence of statistical experiments for general stationary Gaussian processes in the Le Cam sense, providing a theoretical foundation of optimal estimation. The LAN property we have established is also essential for building asymptotic optimality of statistical tests and selecting the order of models based on the likelihood function, see Remark~\ref{Remark:App-LAN} for further references. Third, our method serves as a benchmark for evaluating the finite-sample performance of the existing plug-in MLE.  Although the performance gap between the plug-in MLE and the MLE with known \( \mu \) can be substantial in finite samples \citep{Cheung-Diebold-1994}, our analysis shows that this difference is not driven by inefficiencies in estimating \( \mu \).

The remainder of this paper is organized as follows. Section \ref{Sec:main-results} presents the exact ML method and develops asymptotic Properties. In Section \ref{Sec:example}, we provide several examples to which our results can be applied. Section \ref{Sec:Simulation} presents a Monte Carlo study to assess the performance of the estimation method. Section \ref{Sec:con} concludes the paper.

%%%%%%%%%%%%%%%%%%%%%%%%%%%%%%
\section{Exact MLE and Asymptotic Properties} 
\label{Sec:main-results}
%%%Estimation Methods
\subsection{Notation and Exact MLE}

Let $\Theta_{\xi}$ be a convex domain of $\R^{p-1}$ with compact closure and set $\Theta:=\Theta_{\xi}\times(0,\infty)$. 
Write $\theta=(\xi,\sigma)^{\top}\in\Theta$ and $\vartheta=(\theta,\mu)^{\top}=(\xi,\sigma,\mu)^{\top}\in\Theta\times\R$. 
Denote by $\partial_{z}=\partial/\partial{z}$, $\partial_{\omega}=\partial/\partial{\omega}$ and $\partial_{j} := \partial/\partial{\theta_{j}}$ for $j\in\{1,\cdots,p+1\}$. 
For notational simplicity, $\partial_{0}$, $\partial_{z}^{0}$ and $\partial_{\omega}^{0}$ denote the identify operator. 
The derivative operators $\partial_{j_{1},\cdots,j_{k}}^{k}$ are recursively defined by $\partial_{j_{1},\cdots,j_{k}}^{k} := \partial_{j_{1}} \circ \partial_{j_{2},\cdots,j_{k}}^{k-1}$ for $j_{1},\cdots,j_{k}\in\{0,1,\cdots,p+1\}$ and $k\in\N$. 
Moreover, $\mathbf{1}_{n}$ denotes a $n$-dimensional vector whose all elements are equal to $1$ and, for an integrable function $f$ on $[-\pi,\pi]$, $\Sigma_{n}(f)$ denotes the symmetric Toeplitz matrix whose $(i,j)$th elements are equal to the $(i-j)$th Fourier coefficients of $f$.\\

Let us consider a stationary Gaussian time series $\{X^{\vartheta}_{j}\}_{j\in\Z}$ with mean $\mu$ and spectral density function $s^{X}(\omega,\theta)$. 
We may write $s_{\theta}^{X}(\omega):=s^{X}(\omega,\theta)$. 
Let us denote by $\vartheta_{0}=(\xi_{0},\sigma_{0},\mu_{0})^{\top}$ an interior point of $\Theta\times\R$, which may call a true value of the parameter $\vartheta$, and we assume that we observe a realization of $X_{1}^{\vartheta_{0}},\cdots,X_{n}^{\vartheta_{0}}$. \textbf{Let $s_{\xi}^{X}(\omega)= s_{\theta}^{X}(\omega)/\sigma^2$.}
Then, for each $\vartheta=(\theta,\mu)^{\top}\in\Theta\times\R$ and $n\in\N$, we denote by $\P_{\vartheta}^{n}$ a distribution on the Borel space $(\R^{n},\mathcal{B}(\R^{n}))$ under which a random vector $\mathbf{X}_{n} =(X_{1},\cdots,X_{n})^{\top}$ follows a $n$-dimensional Gaussian vector with mean vector $\mu\mathbf{1}_{n}$ and variance-covariance matrix $\Sigma_{n}(s_{\theta}^{X})$. We also denote by $\gamma_{\theta}^{X}(\cdot)$ the autocovariance function of $X^{\vartheta}$. \\

Denote by $\ell_{n}(\vartheta)\equiv\ell_n(\xi,\sigma,\mu)$ the Gaussian log-likelihood function of the observations $\mathbf{X}_{n}$ under the distribution $\P_{\vartheta}^{n}$, which is given by
\begin{align}\label{def:log-Gaussian-likelihood}
    \ell_{n}(\vartheta)%\propto
    = -\frac{n}{2}\log(2\pi)
    %-\frac{1}{2}
    -\frac{n}{2}\log{\sigma^{2}}
    -\frac{1}{2}\log\mathrm{det}\bigl[\Sigma_{n}(s_{\xi}^{X})\bigr]
    %-\frac{1}{2}
    -\frac{1}{2\sigma^{2}}\left(\mathbf{X}_{n}-\mu\mathbf{1}_{n}\right)^{\top}\Sigma_{n}(s_{\xi}^{X})^{-1}\left(\mathbf{X}_{n}-\mu\mathbf{1}_{n}\right),
\end{align}
and then the maximum likelihood estimator~(MLE)\footnote{We have derived alternative expression of the likelihood function using the conditional likelihood based on the Bayes formula, which improve the computational efficiency by avoiding direct computations of the inverse and determinant of large-scale covariance matrix. See \ref{Sec:mle_implement} for details.} is defined by
\begin{equation*}
	\widehat{\vartheta}_{n}^{\mathrm{MLE}}:=(\widehat{\xi}_{n}^{\mathrm{MLE}},\widehat{\sigma}_{n}^{\mathrm{MLE}},\widehat{\mu}_{n}^{\mathrm{MLE}})^{\top}
    \in \argmax_{(\xi,\sigma,\mu)^{\top}\in\Theta_{\xi}\times(0,\infty)\times\R}\ell_{n}(\vartheta).
\end{equation*}
Notice that, from the definition of the MLE, the MLE satisfies the estimating equations
\begin{align*}
    \partial_{\mu}\ell_{n}(\vartheta) = 0\ \ \mbox{and}\ \ \partial_{\sigma}\ell_{n}(\vartheta) = 0,  
\end{align*}
which conclude the equations
\begin{align*}
    \mu = \frac{ \mathbf{1}_{n}^{\top}\Sigma_{n}(s_{\xi}^{X})^{-1}\mathbf{X}_{n} }{ \mathbf{1}_{n}^{\top}\Sigma_{n}(s_{\xi}^{X})^{-1}\mathbf{1}_{n} } =: \mu_{n}(\xi)
    \ \ \mbox{and}\ \ 
    \sigma^{2} = \frac{1}{n}\left(\mathbf{X}_{n}-\mu\mathbf{1}_{n}\right)^{\top}\Sigma_{n}(s_{\xi}^{X})^{-1}\left(\mathbf{X}_{n}-\mu\mathbf{1}_{n}\right) =: \sigma_{n}^{2}(\xi,\mu)
\end{align*}
hold for any $(\xi,\sigma,\mu)^{\top}\in\Theta_{\xi}\times(0,\infty)\times\R$. 
Then the MLE 
$\widehat{\xi}_{n}^{\mathrm{MLE}}$ would be a maximizer of the function $\bar{\ell}_{n}(\xi):=\ell_{n}(\xi,\bar{\sigma}_{n}(\xi),\mu_{n}(\xi))$, where $\bar{\sigma}^{2}_{n}(\xi):=\sigma_{n}^{2}(\xi,\mu_{n}(\xi))$ and $\bar{\sigma}_{n}(\xi):=\sqrt{\bar{\sigma}^{2}_{n}(\xi)}$, over the parameter $\xi\in\Theta_{\xi}$ and the MLEs $\widehat{\mu}_{n}^{\mathrm{MLE}}$ and $\widehat{\sigma}_{n}^{\mathrm{MLE}}$ satisfy the equations
\begin{align*}
    \widehat{\mu}_{n}^{\mathrm{MLE}}=\mu_{n}(\widehat{\xi}_{n}^{\mathrm{MLE}})
    \ \ \mbox{and}\ \ 
    \widehat{\sigma}_{n}^{\mathrm{MLE}}=\bar{\sigma}_{n}(\widehat{\xi}_{n}^{\mathrm{MLE}})=\sqrt{\sigma_{n}^{2}(\widehat{\xi}_{n}^{\mathrm{MLE}},\widehat{\mu}_{n}^{\mathrm{MLE}})}.
\end{align*}
Therefore, we define our proposed estimator $\widehat{\vartheta}_{n} :=(\widehat{\xi}_{n},\widehat{\sigma}_{n},\widehat{\mu}_{n})^{\top}$ by
\begin{align}\label{Def:EMLE}
    \widehat{\xi}_{n}\in \argmax_{\xi\in\Theta_{\xi}}\bar{\ell}_{n}(\xi),\ \ \widehat{\sigma}_{n}:=\bar{\sigma}_{n}(\widehat{\xi}_{n}),
    \ \ \mbox{and}\ \ \widehat{\mu}_{n}:=\mu_{n}(\widehat{\xi}_{n})
\end{align}
and we call $\widehat{\vartheta}_{n} =(\widehat{\xi}_{n},\widehat{\sigma}_{n},\widehat{\mu}_{n})^{\top}$ the {\it exact MLE}  throughout this paper. %\\
In subsequent sections, we investigate the asymptotic properties of the exact MLE.

%%%%%%%%%%%%%%%%%%%%%%%%%%%%%%
\subsection{Consistency and Asymptotic Normality of Exact MLE}\label{Sec:MLE}

We firstly introduce several conditions on the spectral density function $s_{\theta}^{X}(\omega)$ summarized in the following assumption that is used to prove asymptotic properties of the exact MLE and the likelihood ratio process; see Sections~\ref{Sec:MLE} and \ref{Sec:LAN} for details.
\begin{assumption} \label{Assump:DSPD1}
	\begin{enumerate}[label=$(\arabic*)$]
		\item For each $\theta\in\Theta$, $s_{\theta}^{X}(\omega)$ is a non-negative integrable even function in $\omega$ on $[-\pi,\pi]$ with $2\pi$-periodicity. Moreover, it satisfies
        \begin{itemize}
            \item for each $\omega\in[-\pi,\pi]\backslash\{0\}$, $s_{\theta}^{X}(\omega)$ is three times continuously differentiable in $\theta$ on the interior of $\Theta$,
            \item for each $\theta\in\Theta$ and $j\in\{1,\cdots,p\}$, $s_{\theta}^{X}(\omega)$ and $\partial_{j}s_{\theta}^{X}(\omega)$ are continuously differentiable in $\omega$ on $[-\pi,\pi]\setminus\{0\}$. 
        \end{itemize}
		%%%%%%%%%%
        \item \label{SPD.Identifiability.Cond} 
        If $\theta_{1}$ and $\theta_{2}$ are distinct elements of $\Theta$, the set $\{\omega\in [-\pi,\pi]: s_{\theta_{1}}^{X}(\omega)\neq s_{\theta_{2}}^{X}(\omega)\}$ has a positive Lebesgue measure.
		%%%%%%%%%%
        \item\label{dspd-boundedness} 
        There exists a continuous function $\alpha_X:\Theta_\xi\to(-\infty,1)$ such that for any $\iota>0$ and some constants $c_{1,\iota},c_{2,\iota},c_{3,\iota}>0$, which only depends on $\iota$, the following conditions hold for every $(\theta,\omega)\in\Theta\times[-\pi, \pi]\backslash\{0\}$:
		\begin{enumerate}[label=$(\alph*)$]
			\item\label{dspd-boundedness-1}
            $c_{1,\iota}|\omega|^{-\alpha_{X}(\xi)+\iota}\leq s_{\theta}^{X}(\omega)\leq c_{2,\iota}|\omega|^{-\alpha_{X}(\xi)-\iota}$.
			\item For any $j_{1},j_{2},j_{3}\in\{0,1,\cdots,p\}$,
			\begin{equation*}
			\left| \partial_{j_{1},j_{2},j_{3}}^{3} s_{\theta}^{X}(\omega)\right|\leq c_{3, \iota}|\omega|^{-\alpha_{X}(\xi)-\iota}
            \ \ \mbox{and}\ \ 
			\left| \partial_{\omega} \partial_{j_{1}} s_{\theta}^{X}(\omega)\right|\leq c_{3, \iota}|\omega|^{-\alpha_{X}(\xi)-1-\iota}.
			\end{equation*}
		\end{enumerate} 
	\end{enumerate}
\end{assumption}

Assumption~\ref{Assump:DSPD1} is the usual conditions on the ``discrete-time'' spectral density function for stationary Gaussian time series with long/short/anti-persistent memory used in the literature; see the assumptions in \cite{Fox-Taqqu-1986}, \cite{Dahlhaus-1989, Dahlhaus-2006}, \cite{Lieberman-2012}, \cite{Cohen-Gamboa-Lacaux-Loubes-2013} and \cite{Fukasawa-Takabatake-2019} as references. $X_j^{\vartheta}$ is said to have long memory (or long-range dependence) if $0 < \alpha_{X}(\xi) < 1$, 
short memory if $\alpha_{X}(\xi) = 0$, and anti-persistence if $\alpha_{X}(\xi)< 0$. 
The range $\alpha_{X}(\xi) \leq -1$ corresponds to noninvertibility, and our results cover this case as well. Since these memory properties are relevant across various applications, we do not impose prior restrictions on the memory type of the process.

%%%%%%%%%%%%%%%%%%%%%%%%%%%%%%
Before stating our main results, we introduce additional notation.  
We write $\widehat{\theta}_{n}:=(\widehat{\xi}_{n},\widehat{\sigma}_{n})^{\top}$ and then we may write $\widehat{\vartheta}_{n}=(\widehat{\theta}_{n},\widehat{\mu}_{n})^{\top}$. 
Define $p\times p$ dimensional matrix $\mathcal{F}_{p}(\theta)$ by
\begin{align}\label{Def:mat-Fp}
    \mathcal{F}_{p}(\theta):= 
     \begin{pmatrix}\displaystyle
        \frac{1}{4\pi}\int_{-\pi}^{\pi}\partial_{i}\log{s_{\theta}^{X}(\omega)} \partial_{j}\log{s_{\theta}^{X}(\omega)} \dd\omega
    \end{pmatrix}_{i,j=1,\ldots,p}
    =
    \begin{pmatrix}
        \mathcal{F}_{p-1}(\xi) & a_{p-1}(\theta) \\
        a_{p-1}(\theta)^{\top} & 2\sigma^{-2}
    \end{pmatrix}
\end{align}
where 
\begin{align*}
    a_{p-1}(\theta):=\frac{1}{2\pi\sigma}\int_{-\pi}^{\pi}\partial_{\xi}\log{s_{\xi}^{X}(\omega)}\dd\omega,
    \ \ \mathcal{F}_{p-1}(\xi) := 
    \begin{pmatrix}\displaystyle
        \frac{1}{4\pi}\int_{-\pi}^{\pi}\partial_{i}\log{s_{\xi}^{X}(\omega)} \partial_{j}\log{s_{\xi}^{X}(\omega)} \dd\omega
    \end{pmatrix}_{i,j=1,\ldots,p-1}.
\end{align*}

\color{black}

Then we assume the following condition on $\mathcal{F}_{p}(\xi)$. 
\begin{assumption}\label{Assump:mat-Fp-invertible}%\color{red}
    The matrix $\mathcal{F}_{p}(\xi)$ is invertible for each $\xi\in\Theta_{\xi}$. 
\end{assumption}

Our first main result is a (weak) consistency and an asymptotic normality of 
the sequence of the MLEs $\{\widehat{\theta}_{n}\}_{n\in\N}$ defined in \eqref{Def:EMLE},  
that is a generalization of, for example, %Theorem~??? in \cite{yajima1985}, 
Theorems~3.1 and 3.2 in \cite{Dahlhaus-1989} and Theorem~1 in \cite{Lieberman-2012} to the case of general Gaussian processes using the multi-step estimation procedure based on the exact MLE defined in \eqref{Def:EMLE}. 
\begin{theorem} \label{Thm:MLE}
    Under Assumptions~$\ref{Assump:DSPD1}$ and $\ref{Assump:mat-Fp-invertible}$, the sequence of the exact MLEs $\{\widehat{\theta}_{n}\}_{n\in\N}$ is weakly consistent and asymptotically normal, that is, it holds that
	\begin{align*}
		\sqrt{n} (\widehat{\theta}_{n}-\theta) 
        \rightarrow \mathcal{N}_{p}( \mathbf{0}_{p},\mathcal{F}_{p}(\theta)^{-1} )
        \ \ \mbox{as $n\to\infty$}
	\end{align*}
	in law under the distribution $\P_{\vartheta}^{n}$, where $\mathcal{F}_{p}(\theta)$ is the non-singular matrix defined in \eqref{Def:mat-Fp}. 
\end{theorem}
See Section~\ref{Sec:proof-Thm:MLE-consistency} for the proof of consistency and Section~\ref{Sec:proof-Thm:MLE-asynormal} for the proof of asymptotic normality. 

%%%%%%%%%%%%%%%%%%%%%%%%%%%%%%

To prove the asymptotic normality of the MLEs for the joint estimation $\vartheta=(\theta,\mu)^{\top}$ as well as the MLE for $\mu$, 
we need to further assume the precise asymptotic behavior of the spectral density function $s_{\theta}^{X}(\omega)$ around the frequency $\omega=0$ given in the following assumption.
\begin{assumption}\label{Assump:DSPD2}
    In addition to Assumption~$\ref{Assump:DSPD1}$, we further assume that there exists a continuous function $c_{X}:\Theta_{\xi}\to(0,\infty)$ such that for each $\xi\in\Theta_{\xi}$,
    \begin{align*}
        s_{\theta}^{X}(\omega) \sim \sigma^{2}c_{X}(\xi) |\omega|^{-\alpha_{X}(\xi)}\ \ \mbox{as $|\omega|\to 0$}.
    \end{align*}
\end{assumption}
Based on $\mathcal{F}_{p}(\xi)$ defined in \eqref{Def:mat-Fp}, we further define
\begin{align}\label{Def:mat-I}
    \Phi_{n}(\vartheta):=
    \begin{pmatrix}
      n^{-\frac{1}{2}}I_{p} & \mathbf{0}_{p} \\
      \mathbf{0}_{p}^{\top}& n^{-\frac{1}{2}(1-\alpha_{X}(\xi))}
    \end{pmatrix}
    \ \ \mbox{and}\ \ \mathcal{I}(\theta):=
    \begin{pmatrix}
        \mathcal{F}_{p}(\theta)&\mathbf{0}_{p} \\
        \mathbf{0}_{p}^{\top}&\frac{2\pi\Gamma(1-\alpha_{X}(\xi))%L_{\theta}^{\Delta}(0)
    	\sigma^2c_{X}(\xi)}{\mathrm{B}(1-\alpha_{X}(\xi),1-\alpha_{X}(\xi))}
    \end{pmatrix},
\end{align}
where $B(\alpha,\beta)$ is the beta-function. Note that, under Assumption~\ref{Assump:mat-Fp-invertible}, the matrix $\mathcal{I}(\xi)$ is also invertible for each $\xi\in\Theta_{\xi}$. 
Moreover, we also introduce a normalized score function $\zeta_{n}(\vartheta)$ and an observed Fisher information matrix $\mathcal{I}_{n}(\vartheta)$, respectively defined by
\begin{align}\label{Def:obs-FI}
    \zeta_{n}(\vartheta) := \Phi_{n}(\vartheta)^{\top} \partial_{\theta}\ell_{n}(\vartheta)
    \ \ \mbox{and}\ \ 
    \mathcal{I}_{n}(\vartheta) := -\Phi_{n}(\vartheta)^{\top} \partial_{\theta}^{2}\ell_{n}(\vartheta) \Phi_{n}(\vartheta).
\end{align}
Now we can state our second main result of the asymptotic normality of the MLE for the joint parameter $\vartheta=(\theta,\mu)^{\top}$, summarized in the following theorem with its proof given in Section~\ref{Sec:proof-Thm:MLE2}.
\begin{theorem} \label{Thm:MLE2}
    Under Assumptions~$\ref{Assump:mat-Fp-invertible}$ and $\ref{Assump:DSPD2}$, the sequence of the MLEs $\{\widehat{\vartheta}_{n}\}_{n\in\N}$ satisfies the following asymptotic normality:
	\begin{align}\label{MLE-coupling}
		\Phi_{n}(\vartheta)^{-1} (\widehat{\vartheta}_{n}-\vartheta) 
        =
        \mathcal{I}_{n}(\theta)^{-1}\Indi_{\{\det[\mathcal{I}_{n}(\theta)]>0\}}\zeta_{n}(\vartheta) 
        + o_{\P_{\vartheta}^{n}}(1) 
        \rightarrow \mathcal{N}_{p+1}\left(\mathbf{0}_{p+1},\mathcal{I}(\theta)^{-1}\right)
	\end{align}
	in law under the distribution $\P_{\vartheta}^{n}$ as $n\to\infty$, where $\zeta_{n}(\vartheta)$ and $\mathcal{I}_{n}(\theta)$ are the normalized score function and the observed Fisher information matrix defined in \eqref{Def:obs-FI}. 
\end{theorem}

\begin{remark}
\cite{Wang-Xiao-Yu-Zhang-2023} established the consistency and asymptotic normality of the MLE for all parameters in the fOU process. Nonetheless, their proof is model-specific, and their results are encompassed by our more general framework.
\end{remark}

\begin{remark}
The asymptotic normality properties in Theorems~$\ref{Thm:MLE}$ and $\ref{Thm:MLE2}$ show that the sequences of plug-in MLEs for the estimation of $\theta$ with nuisance parameter $\mu$ and the exact MLEs for the estimation of $\vartheta=(\theta,\mu)^\top$ are respectively asymptotically efficient in the Fisher sense, in that their limiting covariance matrices equal the inverse of the Fisher information matrices, given by the limits of the sequences of the matrices $-n^{-1}\partial_{\theta}^{2}\ell_{n}((\theta,\mu_0)^\top)$ and $\mathcal{I}_n(\vartheta)$ defined in \eqref{Def:obs-FI}.  
These Fisher efficiencies have also been discussed in \cite{Dahlhaus-1989,Lieberman-2012} for the plug-in MLE and in \cite{Wang-Xiao-Yu-Zhang-2023} for the exact MLE.  
However, these studies do not establish the minimax optimality proved later.  

One might wonder whether this minimax optimality could already be deduced by passing to the limit from the Cram\'er–Rao inequality formulated for possibly biased estimators.  
However, this is not the case.  
Such finite-sample inequalities control pointwise variances but do not rule out the existence of super-efficient points.  
A classical example is Hodges’s estimator in the i.i.d. Gaussian location model, which is $\sqrt{n}$-consistent and has the same asymptotic variance as the MLE except at a single point where it is super-efficient.  
This example illustrates that Cram\'er–Rao–type arguments alone are insufficient to rule out super-efficient points, indicating that the existence of general lower bounds for estimators cannot be derived solely from such inequalities.

\cite{LeCam-1953} showed that the set of super-efficient points has Lebesgue measure zero, with subsequent extensions by \cite{Bahadur-1964} and \cite{Pfanzagl-1970}.
These results, however, rely on parametric i.i.d. models or other regularity assumptions, and to the best of our knowledge, no extensions of these results to statistical experiments induced by general Gaussian processes have been established.  
Therefore, one cannot rely on the Fisher efficiencies or Cram\'er–Rao–type inequalities alone to establish the minimax optimality in local neighborhoods of the true parameter.  
To overcome this limitation, one needs the LAN property together with the H\'ajek–Le Cam’s local asymptotic minimax theorem \citep{Hajek-1972,LeCam-1972}, which ensures that no estimator can asymptotically achieve a smaller risk than the bound determined by the Fisher information in shrinking neighborhoods of the true parameter.
This motivates the next subsection, where we establish the LAN property for statistical experiments induced by general Gaussian processes and then derive the minimax efficiency of the exact MLE as well as the plug-in MLE.  
\end{remark}

%%%%%%%%%%%%%%%%%%%%%%%%%%%%%%
\subsection{LAN property and Asymptotic Efficiency of the Exact MLE}\label{Sec:LAN}

The LAN property is a fundamental concept in asymptotic statistics, originally introduced by \cite{Wald-1943} and further developed by \cite{LeCam-1960}. It plays a central role not only in establishing the asymptotic optimality of estimators but also in facilitating statistical inference. In this section, we establish the LAN property of the sequence of statistical experiments ${(\mathbb{R}^{n}, \mathcal{B}(\mathbb{R}^{n}), {\mathbb{P}{\vartheta}^{n}}{\vartheta \in \Theta \times \mathbb{R}})}_{n \in \mathbb{N}}$ for general Gaussian processes. We then provide several results concerning the local asymptotic minimax efficiency of the exact MLE. Additional applications of the LAN property are also discussed.

\begin{theorem}\label{Thm:LAN}
    Consider the sequence of rate matrices $\{\Phi_{n}(\vartheta)\}_{n\in\N}$ defined in \eqref{Def:mat-Fp}.
        Under Assumptions~$\ref{Assump:mat-Fp-invertible}$ and $\ref{Assump:DSPD2}$, the family of distributions $\{\P^{n}_{\vartheta}\}_{\vartheta\in\Theta\times\R}$ 
        satisfies the following LAN property at each interior point $\vartheta=(\theta,\mu)^{\top}$ of $\Theta\times\R$: 
	    \begin{align*}
		  \left| \log\frac{\mathrm{d}\P^{n}_{\vartheta+\Phi_{n}(\vartheta)u}}{\mathrm{d}\P^{n}_{\vartheta}}
		  -\left(
          u^{\top} \zeta_{n}(\vartheta)-\frac{1}{2}u^{\top}\mathcal{I}(\theta)u
          \right)\right|=o_{\P^{n}_{\theta}}(1)\ \ \mbox{as $n\to\infty$},
	    \end{align*}
        where the invertible matrix $\mathcal{I}(\theta)$ is defined in $\eqref{Def:mat-I}$ and the normalized score function $\zeta_{n}(\vartheta)=\Phi_{n}(\vartheta)^{\top}\ell_{n}(\vartheta)$ satisfies the convergence
	    \begin{equation*}
		  \zeta_{n}(\vartheta)\rightarrow\mathcal{N}(0,\mathcal{I}(\theta))\ \ \mbox{as $n\to\infty$}
	   \end{equation*}
       in law under the distribution $\P_{\vartheta}^{n}$.
\end{theorem}

The proof of Theorem~\ref{Thm:LAN} is left to Section~\ref{Sec:proof-Thm:LAN}.  In the relevant literature, Theorem~2.4 in \cite{Cohen-Gamboa-Lacaux-Loubes-2013} 
provides the LAN property for  centered stationary Gaussian time series with long/short/anti-persistent memory property under Assumptions~\ref{Assump:DSPD1} and \ref{Assump:mat-Fp-invertible}.  Theorem~\ref{Thm:LAN} is a generalization of Theorem~2.4 in \cite{Cohen-Gamboa-Lacaux-Loubes-2013} to general Gaussian processes within the long span asymptotics, which is different from several extensions under the high-frequency asymptotics recently made by \cite{Brouste-Fukasawa-2018,Fukasawa-Takabatake-2019,Szymanski-2024,Szymanski-Takabatake-2023-Additive,Chong-Mies-2025}.

The LAN property is applied to derive asymptotically efficient rates and variances for estimating $\vartheta=(\theta,\mu)^{\top}$. These results are derived from lower bounds of estimation provided by H\'ajek's convolution theorem (see
\cite{Hajek-1972, Ibragimov-Hasminski-1981,LeCam-1972}), recalled below.
\begin{theorem}[Theorem II.12.1 in \cite{Ibragimov-Hasminski-1981}]
\label{Thm:Hajek-LeCam}
Suppose that a family of distributions $\{\mathbb{P}^{n}_{\theta}\}_{\theta\in\Theta}$ satisfies the LAN property at the interior point $\theta$ of $\Theta \subset \mathbb R^{d}$ for a sequence of $d\times d$-rate matrices $\{\Phi_{n}(\vartheta)\}_{n\in\N}$ and a $d\times d$-invertible matrix $\mathcal{I}(\theta)$. Then, for any sequence of estimators $\widehat{\theta}_n$ and any symmetric nonnegative quasi-convex function $L$ on $\R^{d}$ such that $e^{-\varepsilon \|z\|_{\R^{d}}^2} L(z) \to 0$ as $\|z\|_{\R^{d}} \to \infty$ for any $\varepsilon > 0$, we have
\begin{align*}
    \varliminf_{c\to\infty}\varliminf_{n \to \infty}
    \sup_{\theta^{\prime}\in\Theta: \left\|\Phi_{n}(\vartheta)^{-1}(\theta^{\prime}-\theta)\right\|_{\mathbb{R}^{d}}\leq c}
    \mathbb{E}_{\theta^{\prime}}^{n}\left[
    L\left(\Phi_{n}(\vartheta)^{-1}(\widehat{\theta}_{n}-\theta^{\prime})\right)
    \right]
    \geq 
    (2\pi)^{-\frac{d}{2}}\int_{\mathbb{R}^d} L\left(\mathcal{I}(\theta)^{-1/2}z\right) \exp\left(-\frac{|z|^2}{2}\right)\,\mathrm{d}z.
\end{align*}
\end{theorem}

Notice that we have already proved that the sequence of the exact MLEs $\{\widehat{\vartheta}_{n}\}_{n\in\mathbb{N}}$ defined in \eqref{Def:EMLE} satisfies the coupling property \eqref{MLE-coupling} in Theorem~$\ref{Thm:MLE2}$ so that, using the result in Section 7.12.(b) of \cite{Hopfner-2014-Book} in addition to Theorem~\ref{Thm:LAN}, we can conclude that the sequence of exact MLEs is asymptotically efficient in the local asymptotic minimax sense as well as in the Fisher sense, and then it attains the local asymptotic minimax bound of estimation given in Theorem~\ref{Thm:Hajek-LeCam}. 
We summarize the aforementioned result in the following corollary. 
\begin{corollary}[Asymptotic Minimax Optimality]
    Consider the sequence of rate matrices $\{\Phi_{n}(\vartheta)\}_{n\in\N}$ defined in \eqref{Def:mat-Fp} and the matrix $\mathcal{I}(\theta)$ defined in \eqref{Def:mat-I}.
    Under Assumptions~$\ref{Assump:mat-Fp-invertible}$ and $\ref{Assump:DSPD2}$, the sequence of the exact MLEs $\{\widehat{\vartheta}_{n}\}_{n\in\mathbb{N}}$ defined in \eqref{Def:EMLE} attains the local asymptotic minimax bound given in Theorem~$\ref{Thm:Hajek-LeCam}$ at each interior point $\vartheta=(\theta,\mu)^{\top}$ of $\Theta\times\R$. Namely, for any symmetric nonnegative quasi-convex function $L$ on $\R^{p+1}$ such that $e^{-\varepsilon \|z\|_{\R^{p+1}}^2} L(z) \to 0$ as $\|z\|_{\R^{p+1}} \to \infty$ for any $\varepsilon > 0$, we obtain
    \begin{align*}
       &\varliminf_{c\to\infty}\varliminf_{n \to \infty}
       \sup_{\vartheta\in\Theta\times\R: \left\|\Phi_{n}(\vartheta_{0})^{-1}(\vartheta-\vartheta_{0})\right\|_{\mathbb{R}^{p+1}}\leq c}
       \mathbb{E}_{\vartheta}^{n}\left[
       L\left(\Phi_{n}(\vartheta_{0})^{-1}(\widehat{\vartheta}_{n}-\vartheta)\right)
       \right] \\
       &\hspace{1cm}
       = (2\pi)^{-\frac{1}{2}(p+1)} \int_{\mathbb{R}^{p+1}} L\left(\mathcal{I}(\xi_{0})^{-1/2}z\right) \exp\left(-\frac{|z|^2}{2}\right)\dd z.
    \end{align*}
\end{corollary}

\begin{remark}
Regarding the efficiency of the two-stage procedure, the sample mean is clearly not an efficient estimator for $\mu$. For the remaining parameters, \cite{Dahlhaus-1989, Dahlhaus-2006} showed that the asymptotic covariance matrix of the plug-in MLE equals the inverse of the Fisher information matrix; however, they did not explicitly establish the existence of a Cramér--Rao lower bound. Although the LAN property has been established by \cite{Cohen-Gamboa-Lacaux-Loubes-2013} for centered stationary Gaussian processes with long memory, short memory, or anti-persistence, their results do not imply the asymptotic efficiency of the plug-in MLE.
\end{remark}

\begin{remark}\label{Remark:App-LAN}
 Under additional technical conditions, the LAN property established in this paper can be used to verify, for instance, the asymptotically uniformly most powerful unbiased (AUMPU) property of the likelihood ratio test \citep{Choi-Hall-Schick-1996}, as well as asymptotic properties of model selection criteria—such as the (weak) consistency of model selection based on the (quasi) Bayesian information criterion (BIC) \citep{Eguchi-Masuda-2018}.   
\end{remark}

\subsection{Comparison between Exact MLE and Plug-in MLE}

As an alternative estimator of $\theta$ in the literature, the plug-in MLE~(PMLE) is defined by
\begin{align}\label{Def:PMLE}
    \widehat{\theta}_{n}^{\mathrm{PMLE}} \in \argmax_{\theta\in\Theta_{\ast}}\ell_{n}((\theta,\widetilde{\mu}_{n})^{\top})
\end{align}
using some compact set $\Theta_{\ast}\subset\Theta$ and an estimator $\widetilde{\mu}_{n}$. 
Under similar assumptions to Assumption~\ref{Assump:DSPD1}, \cite{Dahlhaus-1989, Dahlhaus-2006} and \cite{Lieberman-2012} show that the plug-in MLE is weakly consistent, asymptotically normal with the asymptotic variance $\mathcal{F}_{p}(\theta)^{-1}$ at the point $\vartheta_{0}$? (should be $\theta_0$) when the plug-in estimator $\widetilde{\mu}_{n}$ satisfies the assumption
\begin{align}\label{Assump:mu-PMLE}
    \widetilde{\mu}_{n}=\mu+o_{\P_{\vartheta_{0}}^{n}}(n^{-\frac{1}{2}(1-\alpha_{X}(\xi_{0}))})
    \ \ \mbox{as $n\to\infty$}.
\end{align} 
The assumption \eqref{Assump:mu-PMLE} corresponds to Assumption~5 of \cite{Dahlhaus-1989} for the long memory case $\alpha_{X}(\xi_{0})\in(0,1)$ and Assumption~5 of \cite{Lieberman-2012} for the long/short/anti-persistent memory case $\alpha_{X}(\xi_{0})\in(-\infty,1)$. The plug-in MLE shares the same convergence rate and asymptotic variance as our exact MLE, implying that it is also asymptotically efficient.

Notice that Theorem~\ref{Thm:LAN} combining with Theorem~\ref{Thm:Hajek-LeCam} yields the asymptotic minimax lower bound
\begin{align}
    \varliminf_{T\to\infty}\sup_{\|\vartheta-\vartheta_{0}\|_{\R^{p+1}}<r} 
    \E_{\theta}^{n}\left[ \frac{n^{1-\alpha_{X}(\xi)}}{\sigma^2} (\widetilde{\mu}_{n}-\mu)^{2} \right] \geq \frac{2\pi c_{X}(\xi_{0})\Gamma(1-\alpha_{X}(\xi_{0}))}{B(1-\alpha_{X}(\xi_{0})/2,1-\alpha_{X}(\xi_{0})/2)?}
    \label{lower-bound-mu}
\end{align}
for any $r>0$ and any sequences of estimators $\widetilde{\mu}_{n}$, which implies the convergence rate of the estimator satisfying the assumption \eqref{Assump:mu-PMLE} is the same as that of the optimal estimators given in \eqref{lower-bound-mu}, that is, the minimax optimal rate. Although \cite{Adenstedt-1974} proved that the best linear unbiased estimator (BLUE) of $\mu$ satisfies the assumption \eqref{Assump:mu-PMLE} for all $\alpha_{X}(\xi_{0})\in(-\infty,1)$, the BLUE of $\mu$ is infeasible without knowing the true value of $\xi_{0}$. Although the assumption \eqref{Assump:mu-PMLE} is satisfied by the widely used sample mean, its asymptotic variance does not achieve the minimax optimal bound.  For example, \cite{Adenstedt-1974} quantified the asymptotic inefficiency of the sample mean for the ARFIMA$(0,d,0)$ model by computing the ratio of the asymptotic variance of an efficient estimator for $\mu$ to that of the sample mean:
\begin{align}
\text{as.eff}(y, \mu) = \frac{(2d + 1) \Gamma(d + 1) \Gamma(2 - 2d)}{\Gamma(1 - d)}.
\label{aseff_sm}
\end{align}
This metric is always less than $1$, except when $d = 0$. This plug-in approach is equivalent to the implementation that applies the MLE to demeaned data.
 \cite{Lieberman-2012} proposed an alternative estimator of the form
\begin{align}\label{est-mu-LRR12}
    \widetilde{\mu}_n^{(1)} := (\mathbf{1}_{n}^{\top}\Sigma_{n}(s_\ast)^{-1}\mathbf{1}_{n})^{-1} \mathbf{1}_{n}^{\top}\Sigma_{n}(s_\ast)^{-1}\mathbf{X}_{n},
\end{align}
where $s_\ast:=s_{\theta_\ast}$ with any $\theta_\ast=(\xi_\ast,\sigma_\ast)^\top\in\Theta$ satisfying $\alpha(\xi_\ast)=\inf_{\xi\in\Theta_\xi}\alpha(\xi)$ (by compactness of $\Theta_\xi$ there exists at least one such value), or even $s_\ast(\omega):=(1-\cos(\omega))^{\frac{\alpha_\ast}{2}}$ with $\alpha_\ast\leq\inf_{\xi\in\Theta_\xi}\alpha(\xi)$, 
and proved that the estimator in \eqref{est-mu-LRR12} satisfies the assumption \eqref{Assump:mu-PMLE} for all $\alpha_{X}(\xi_{0})\in(-\infty,1)$ using the results in \cite{Adenstedt-1974}.  This estimator also fails to attain the minimax optimal bound, as it inherently relies on a misspecified structure of the autocovariance function embedded in the Toeplitz matrix $\Sigma_n(s_\ast)$.

%%%%%%%%%%%%%%%%%%%%%%%%%%%%%%
\section{Examples}
\label{Sec:example}
Our assumptions on the spectral density are very general and apply to many well-known processes, including but not limited to the ARFIMA$(p,d,q)$ process with $|d|<1/2$, the fractional Gaussian noise with an unknown mean and the fOU process.\footnote{For other continuous-time processes—such as the continuous-time autoregressive fractionally integrated moving average (CARFIMA) model—we refer the reader to Tetsuya et al. (WP) for a detailed analysis of their spectral densities.}

\subsection{Non-Centered Gaussian ARFIMA$(p,d,q)$ Process}
The {\it non-centered} ARFIMA$(p,d,q)$ process was introduce by 
\cite{Granger-1980} and \cite{Hosking-1981} independently. For notation simplicity, we start with the ARFIMA$(0,d,0)$ model. 
The non-centered ARFIMA$(0,d,0)$ model is specified as 
\begin{equation}
X_{t}-\mu =\sigma (1-L)^{-d}\epsilon _{t}\text{ with }%
\left\vert d \right\vert <1/2,  \label{arfima}
\end{equation}
where $L$ is the lag operator, $(1-L)^{-d}$ is the fractional difference operator with the memory parameter $d$ and $\epsilon_{t}\overset{\mathrm{iid}}{\sim} N(0,1)$. 
It reduces to a Gaussian white noise when $d=0$.
When $d\in (-1/2,1/2)$, %the process 
the ARFIMA$(0,d,0)$ process is stationary and invertible \citep{bloomfield1985}. 
Let $u_{t}:=(1-L)^{-d}\epsilon _{t}$ be the fractionally integrated process and $\gamma _{u}(j):=\Cov[u_{t},u_{t-j}]$ be its $j$th order auto-covariance. According to \cite{Hosking-1981}, the auto-covariance function of $%
u=\{u_{t}\}_{t\in\Z}$ is expressed by
\begin{equation}
\gamma_{u}(j)=\frac{(-1)^{j}\Gamma (1-2d)}{\Gamma (j-d+1)\Gamma (1-j-d)},\ \ j\in\Z.
\end{equation}%
The long-run variance covariance $\sum_{j=-\infty }^{\infty }\gamma _{u}(j)$ 
$= \infty $ when $d\in \left( 0,1/2\right) $ and $\sum_{j=-\infty
}^{\infty }\gamma_{u}(j)=0$ when $d\in \left( -1/2,0\right) $. Therefore, $%
u_{t}$ has a long memory if $d\in \left( 0,1/2\right) $ and is
anti-persistent if $d\in \left( -1/2,0\right) $. The spectral density of the model is given by
\[s_{\theta}^X(\omega)=\frac{\sigma^2}{2\pi} 
\sim \frac{\sigma^2}{2\pi}|\omega|^{-2d}\ \ \text{ as }|\omega|\rightarrow 0.\]
In this case, $\alpha_X(\xi)=2d, c_X(\xi)=(2\pi)^{-1}$.

Let $p,q\in\N\cup\{0\}$ 
and $\xi:=(\phi_1,\ldots,\phi_p,\psi_1,\ldots,\psi_q)\in\R^{p+q}$. 
The {\it non-centered} ARFIMA$(p,d,q)$ process is defined by 
\begin{align}\label{def:eq-ARFIMA}
    \phi_\xi(L)(X_t-\mu)=\sigma \psi_\xi(L) u_t,\ \ t\in\Z,
\end{align}
where $\phi_\xi(z):=1-\phi_1z-\cdots-\phi_pz^p$ and $\psi_\xi(z):=1+\psi_1z+\cdots+\psi_qz^q$. 
Assume that for each \( \xi \), the functions \( \phi_\xi(z) \) and \( \psi_\xi(z) \) have no common roots in \( \mathbb{C} \), and that all their roots lie outside the unit circle. This implies that \( \phi_\xi(z) \neq 0 \) and \( \psi_\xi(z) \neq 0 \) for \( |z| \leq 1 \). 
Then, for \( |d| < 1/2 \), the difference equation in \eqref{def:eq-ARFIMA} admits a unique stationary solution \( X = \{X_t\}_{t\in\mathbb{Z}} \) of the form
\[
X_t = \mu + \sigma \phi_\xi(L)^{-1} \psi_\xi(L) (1-L)^{-d} \epsilon_t, \quad t \in \mathbb{Z}.
\]

The spectral density function of ARFIMA$(p,d,q)$ is given by
\[
s_{\theta}^X(\omega) = \frac{\sigma^2}{2\pi} |1 - e^{-i\omega}|^{-2d} \frac{|\psi_\xi(e^{-i\omega})|^2}{|\phi_\xi(e^{-i\omega})|^2}.
\]
for $\omega \in (-\pi, \pi]$. It can be shown that
\[
s_{\theta}^X(\omega) =\frac{\sigma^{2}}{2\pi}\left(
2-2\cos\left(  \lambda\right)  \right)  ^{-d}\frac{\left(  1+\sum_{j=1}%
^{q}\psi_{j}\cos\left(  j\omega\right)  \right)  ^{2}+\left(  \sum_{j=1}%
^{q}\psi_j\sin\left(  j\omega\right)  \right)  ^{2}}{\left(  1-\sum_{j=1}^{p}%
\phi_{j}\cos\left(  j\omega\right)  \right)  ^{2}+\left(  \sum_{j=1}%
^{p}\phi_{j}\sin\left(  j\omega\right)  \right)  ^{2}}.
\]
Since the assumption \( \phi_\xi(z) \neq 0 \) for \( |z| \leq 1 \) ensures that the spectral density function of the ARMA\( (p,q) \) process,
\[
f_{\mathrm{ARMA}}(\omega) := \frac{\sigma^2}{2\pi} \frac{|\psi_\xi(e^{-i\omega})|^2}{|\phi_\xi(e^{-i\omega})|^2},
\]
is bounded away from zero on \( [-\pi, \pi] \), the singularity of the spectral density of \( X \) in the vicinity of zero frequency is governed by the ARFIMA\( (0,d,0) \) factor \( |1 - e^{-i\omega}|^{-2d} \). As \( \omega \to 0 \), it exhibits the asymptotic behavior
\[
s_{\theta}^X(\omega) \sim \sigma^2c_X(\theta)|\omega|^{-\alpha_X(\xi)}.
\]

In this case, $\alpha_X(\xi) = 2d, c_X(\xi) = (2\pi)^{-1} \left| \frac{\psi_\xi(1)}{\phi_\xi(1)} \right|^2$ in Assumptions~\ref{Assump:DSPD1} and \ref{Assump:DSPD2}. Hence, our results are applicable to the non-centered Gaussian ARFIMA$(p, d, q)$ process. According to Theorem~\ref{Thm:MLE2}, when $d < 0$, the convergence rate for the exact MLE of $\mu$ is $\frac{1}{2}(1 - \alpha_X(\xi)) > 1/2$, indicating super-consistency.

\subsection{Non-Centered Fractional Gaussian Noise}
The fractional Brownian motion (fBm) with Hurst index $H\in(0,1)$, denoted by $B^H = \{B_t^H\}_{t \in \mathbb{R}}$, is a unique centered Gaussian process that is almost surely equal to zero at $t = 0$ and possesses both stationary increments and $H$-self-similarity properties. Specifically, these properties are expressed as  
\begin{align*}
    B_t^H-B_s^H\overset{d}{=}B_{t-s}^H-B_0^H
    \ \ \mbox{and}\ \
    B_{ct}^H\overset{d}{=}c^HB_t^H
\end{align*}
for any $s, t \in \mathbb{R}$ and $c > 0$, where $\overset{d}{=}$ denotes equality in distribution. 
\cite{mandelbrot1968} demonstrated that the fBm can be represented as a causal moving average process involving the past differential increments of a (two-sided) standard Brownian motion $B = \{B_t\}_{t \in \mathbb{R}}$. This representation is given by  
\begin{equation*}
B^{H}_t = \frac{1}{\Gamma (H+0.5)} \left\{ \int_{-\infty }^{0} \left[ \left(
t-s\right) ^{H-0.5}-(-s)^{H-0.5}\right] \dd B_s + \int_{0}^{t} \left(
t-s\right) ^{H-0.5} \dd B_s \right\}, 
\end{equation*}  
where $\Gamma(x)$ denotes the gamma function, 
\footnote{%
This is also referred to as the Type I fBm.} 
which implies that the fBm reduces to the standard Brownian motion when $H=0.5$.

The fBm is a Gaussian process with mean zero and covariance 
\begin{equation}
\Cov[B^{H}_t,B^{H}_s] =\frac{1}{2}\left( \left\vert
t\right\vert ^{2H}+\left\vert s\right\vert ^{2H}-\left\vert t-s\right\vert
^{2H}\right) ,\;\ \forall t,s\in\mathbb{R}.  \label{cov}
\end{equation}%
The increment of fBm is fGn and denoted by $y_{t}$. Using discrete time
notations, we have 
\begin{equation*}
\epsilon_{t} := \sigma ( B^{H}_t-B^{H}_{t-1}).
\end{equation*}%
We extend the fGn to the non-centered case by defining its expectation as $\mu$. The process is given by:
\begin{equation*}
X_t := \mu + \epsilon_t = \mu + \sigma ( B^{H}_t-B^{H}_{t-1}).
\end{equation*}
The auto-covariance function of $X$ is, $\forall k\geq 0$, 
\begin{equation}
 \gamma_X(k):=\Cov[X_{t},X_{t+k}] =\frac{\sigma ^{2}}{2}\left[
(k+1)^{2H}-2k^{2H}+(k-1)^{2H}\right]\sim\sigma ^{2}H(2H-1)k^{2H-2}\text{ for
large }k,  \label{expr:acf-fGn}
\end{equation}%
where the approximation is based on the Taylor expansion. 
When $H\in (0.5,1)$, the asymptotic behavior in \eqref{expr:acf-fGn} implies that the sequence of the auto-covariances of fGn is not absolutely summable so that the fGn has the long memory property. 
When $H\in (0,0.5)$, we can also verify that $\forall k\neq 0$, $%
\Cov[X_{t},X_{t+k}] <0$ and $\sum_{k=-\infty }^{\infty}\Cov[X_{t},X_{t+k}] =0$ so that the fGn has the anti-persistent memory property.

The spectral density of (non-centered) fGn is
given by \cite{sinai1976}:
\begin{equation}
s_{\theta}^X(\omega )=\sigma^2C_{H}\{2(1-\cos (\omega ))\}\sum_{k=-\infty }^{\infty }\left\vert
2\pi k+\omega \right\vert ^{-1-2H} \text{ for }%
\omega \in [-\pi, \pi ],  \label{sd_fgn0}
\end{equation}%
where $C_{H}:=\left( 2\pi \right)
^{-1}\Gamma (2H+1)\sin (\pi H)$. It can be shown that
\[s_{\theta}^X(\omega ) \sim \sigma^2C_H|\omega|^{1-2H},\,\, \text{when}\,\, \omega \rightarrow 0.\]
In this case, the functions $\alpha_X(\xi)$ and $c_X(\xi)$ in Assumptions~\ref{Assump:DSPD1} and \ref{Assump:DSPD2} are $\alpha_X(\xi)=2H-1$ and $c_X(\xi)=C_H$. Hence, our results are applicable to the non-centered fractional Gaussian noise.  According to Theorem~\ref{Thm:MLE2},  when $H<0.5$, the convergence rate for the exact MLE of $\mu$ is $\frac{1}{2}(1-\alpha_X(\xi))>1/2$, implying super-consistency. This result echoes that for ARFIMA.

 \subsection{Fractional Ornstein-Uhlenbeck Process}
The fractional Ornstein-Uhlenbeck (fOU) process is an extension of the classical Ornstein-Uhlenbeck (OU) process, where the driving noise is replaced by a fBm with Hurst index $H \in (0,1)$. This process is particularly useful for modeling systems that exhibit long-range dependence and locally self-similarity, which cannot be captured by the classical OU process. 
The stationary fOU process with a long-run mean $\mu$ has applications in various fields, including mathematical finance, physics, and time series modeling.

The fOU process $Y=\{Y_t\}_{t\in\mathbb{R}}$ is defined by a unique solution of the following linear SDE (Stochastic Differential Equation):
\begin{equation}\label{def:SDE-fOU}
    \dd Y_t = -\kappa (Y_t - \mu) \dd t + \sigma\dd B_t^H,\ \ t\geq 0,
\end{equation}
with initial condition $Y_0$, where $B^H = \{B_t^H\}_{t\in\mathbb{R}}$ is a fBm with Hurst index $H$. 
The explicit solution of this SDE is given by
\begin{equation}\label{def:sol-fOU}
    Y_t = Y_0 e^{-\kappa t} + \mu (1 - e^{-\kappa t}) + \sigma\int_{0}^{t} e^{-\kappa (t-s)} \dd B_s^H,\ \ t\geq 0,
\end{equation}
where the above stochastic integral can be interpreted as the pathwise Riemann-Stieltjes integral or the Wiener integral associated with the fBm for any $H\in(0,1)$.  
The fOU process reduces to the classical OU process when $H=0.5$ and to the fBm when $\kappa =0$. 
When $\kappa >0$, the stationary solution of the SDE \eqref{def:SDE-fOU}, denoted by $\bar{Y}=\{\bar{Y}_t\}_{t\in\R}$, is given by
\begin{align}\label{def:sol-fOU-stationary}
    \bar{Y}_t := \mu + \sigma\int_{-\infty}^{t} e^{-\kappa (t-s)} \dd B_s^H,\ \ t\in\R. 
\end{align}
For $t\geq 0$, the unique solution of the SDE \eqref{def:SDE-fOU} with the initial condition
\begin{align*}
    Y_0 = \mu + \sigma\int_{-\infty}^{0} e^{s} \dd B_s^H
\end{align*}
is exactly equal to the stationary solution $\bar{Y}$ in \eqref{def:sol-fOU-stationary}, and then the error between $\bar{Y}_t$ and $Y_t$ with arbitrary initial condition $Y_0$ is expressed by
\begin{align*}
    |Y_t-\bar{Y}_t|=|Y_0-\bar{Y}_0|e^{-\kappa t},\ \ t\geq 0,
\end{align*}
which implies that the error between the solutions \eqref{def:sol-fOU} and \eqref{def:sol-fOU-stationary} converges to zero exponentially as $t\to\infty$ for arbitrary initial condition $Y_0$. 
In the rest of this section, we consider the case where a data generating process is the discretely and equidistantly observed time series from the stationary solution given in \eqref{def:sol-fOU-stationary}.

Consider a stationary time series $X=\{X_j\}_{j\in\Z}$ of the form $X_j:=Y_{j\Delta}$ for $j\in\Z$ with the sampling frequency $\Delta$. Notice that the time series $X$ is stationary and its auto-covariance is available from \citet{garnier2018} when $\kappa >0$:
\begin{align} 
\begin{split}
\gamma_{\theta}^{X}(j) = \frac{\sigma^2}{2 \kappa^{2H}} \left( \frac{1}{2} \int_{-\infty}^{ \infty} e^{- \vert s \vert} \vert \kappa j\Delta+s \vert^{2H} \mathrm{d}s - \vert \kappa j\Delta \vert^{2H}\right).
\end{split}
\end{align}
From \citet{Cheridito-Kawaguchi-Maejima-2003}, the autocovariance function of the fOU process exhibits the same order of decay as fGn, decaying hyperbolically for $H \neq 1/2$.
\cite{Cheridito-Kawaguchi-Maejima-2003} and \cite{Hult-2003-PhDThesis} provide the spectral density function of the stationary solution $\bar{Y}$ given by
\begin{equation}
s_{\theta}^{\bar{Y}}\left( z\right) =\sigma^2
C_{H} |z|^{1-2H}(\kappa ^{2}+z ^{2})^{-1}\text{ for }%
z \in (-\infty ,\infty ).  \label{eq:sp_fou_cont}
\end{equation}
For discrete observations $X$, the spectral density is given by \citep{Hult-2003-PhDThesis}
\begin{equation}
s_{\theta}^{X}\left( \omega \right) = \sigma^2C_{H} \Delta ^{2H} \sum_{k=-\infty }^{\infty }\frac{%
\left\vert \omega +2\pi k\right\vert ^{1-2H}}{(\kappa \Delta )^{2}+\left(
\omega +2\pi k\right) ^{2}}\;\text{ for }\omega \in [-\pi, \pi].
\label{eq:spec_fou}
\end{equation}

It can be shown that
\[
s_{\theta}^X(\omega) \sim 
\begin{cases} 
\sigma^2C_{H} \Delta ^{2H} \sum_{k=-\infty }^{\infty }\frac{%
\left\vert 2\pi k\right\vert ^{1-2H}}{(\kappa \Delta )^{2}+\left(2\pi k\right) ^{2}}, & \text{when } \omega \rightarrow 0 \text{ and } 0 < H \leq \frac{1}{2}, \\
\sigma^2C_{H} \Delta ^{2H-2}\kappa^{-2} |\omega|^{1 - 2H}, & \text{when } \omega \rightarrow 0 \text{ and } \frac{1}{2} < H < 1,
\end{cases}
\]
 \citep{Shi-Yu-Zhang-2024-b}. Hence, for fOU, we have
\[
\alpha_X(\xi) = 
\begin{cases} 
0, & \text{when } H \leq \frac{1}{2},\\
2H - 1, & \text{when } H > \frac{1}{2}, 
\end{cases}
\]
and
\[
c_X(\xi) = 
\begin{cases} 
C_{H} \Delta ^{2H} \sum_{k=-\infty }^{\infty }\frac{%
\left\vert 2\pi k\right\vert ^{1-2H}}{(\kappa \Delta )^{2}+\left(2\pi k\right) ^{2}}, & \text{when } H \leq \frac{1}{2},\\
C_{H} \Delta ^{2H-2}\kappa^{-2} & \text{when } H > \frac{1}{2}, 
\end{cases}
\]
in Assumptions~\ref{Assump:DSPD1} and \ref{Assump:DSPD2}. Hence, our results are applicable to the fOU process. However, the function $\alpha_X(\xi)$ exhibits a sharp contrast compared to that of ARFIMA and fGn. According to Theorem~\ref{Thm:MLE2}, the convergence rate for the exact MLE of $\mu$ in fOU is $\sqrt{n}$ when $H\leq 1/2$, as recently reported in \citep{Wang-Xiao-Yu-Zhang-2023}.

\section{Mont Carlo Study}\label{Sec:Simulation}

We consider three data-generating processes (DGPs): ARFIMA$(0,d,0)$, fGn and fOU. For simplicity, the long-run mean $\mu$ is set to 0. The results for $\mu \pm 1$  are provided in Appendix \ref{Sec:robust_mu}. The parameter $d$ in the ARFIMA$(0,d,0)$ model takes 9 values: $\{-0.4, -0.3, \dots, 0, \dots, 0.4\}$, while the parameter $H$ in the fGn and fOU takes 9 values: $H = 0.1, 0.2, \dots, 0.9$. The fOU has an additional parameter $\kappa=10$. The scale parameter $\sigma$ is set to $1$ in all cases. For each DGP, we compare our exact MLE with the two MLEs considered in \cite{Cheung-Diebold-1994}. MLE1 refers to the case where $\mu$ is known, MLE2 refers to our exact MLE, and MLE3 refers to the plug-in MLE, where the sample mean is used as an estimator for $\mu$. The number of replications is set to 1000. The sample size is set to $250$ or $1000$. Reported are the bias, standard error (Std), and root mean squared error (RMSE) across all replications for each method. 

\begin{table}[H]
\caption{Bias and Std of alternative MLEs for $ARFIMA(0,d,0)$: $\mu=0$ and $\sigma=1$. MLE1 is MLE with known $\mu$; MLE2 is our exact MLE; MLE3 is PMLE.}\label{table01}
\centering
\scalebox{0.82}{
    \begin{tabular}{cc|ccc|ccc|ccc}
        \hline
        \hline
        & & MLE1 & MLE2 & MLE3  & MLE1 & MLE2 & MLE3  & MLE1 & MLE2 & MLE3 \\
        \hline
         \multicolumn{11}{c}{$n=250$}\\
         \hline
       &  & \multicolumn{3}{c|}{$d=$ -0.40} & \multicolumn{3}{c|}{$d=$ -0.30}  & \multicolumn{3}{c}{$d=$ -0.20}  \\
        \hline
        $\mu$ & Bias &- & -0.0004 & -0.0003 & - & -0.0005 & -0.0010 & - & 0.0010 & 0.0012 \\
            & Std  & - & 0.0096 & 0.0113 & - & 0.0151 & 0.0162 & - & 0.0244 & 0.0248 \\
      %      & RMSE & 0.0000 & 0.0097 & 0.0114 & 0.0000 & 0.0151 & 0.0163 & 0.0000 & 0.0245 & 0.0249 \\
            \multicolumn{11}{c}{}\\
        $d$ & Bias & -0.0029 & -0.0140 & -0.0074 & -0.0029 & -0.0157 & -0.0121 & -0.0056 & -0.0194 & -0.0179 \\
            & Std  & 0.0507 & 0.0510 & 0.0500 & 0.0518 & 0.0553 & 0.0542 & 0.0509 & 0.0536 & 0.0529 \\
     %       & RMSE & 0.0509 & 0.0530 & 0.0506 & 0.0520 & 0.0576 & 0.0556 & 0.0513 & 0.0571 & 0.0559 \\
            \multicolumn{11}{c}{}\\
         $\sigma$ & Bias & -0.0001 & -0.0022 & -0.0012 & -0.0019 & -0.0042 & -0.0038 & -0.0025 & -0.0049 & -0.0047 \\
         & Std  & 0.0445  & 0.0445  & 0.0446  & 0.0447  & 0.0449  & 0.0449  & 0.0447  & 0.0446  & 0.0446  \\
   %      & RMSE & 0.0445  & 0.0446  & 0.0446  & 0.0447  & 0.0451  & 0.0451  & 0.0447  & 0.0449  & 0.0449  \\
            \hline
        & & \multicolumn{3}{c|}{$d=$ -0.10} & \multicolumn{3}{c|}{$d=$ 0.00}  & \multicolumn{3}{c}{$d=$ 0.10}  \\
        \hline
        $\mu$ & Bias & - & -0.0003 & -0.0004 & - & -0.0018 & -0.0017 & - & 0.0028 & 0.0027 \\
            & Std  & - & 0.0394 & 0.0396 & - & 0.0637 & 0.0636 & - & 0.1024 & 0.1024 \\
       %     & RMSE & 0.0000 & 0.0394 & 0.0397 & 0.0000 & 0.0638 & 0.0637 & 0.0000 & 0.1026 & 0.1026 \\
            \multicolumn{11}{c}{}\\
        $d$ & Bias & -0.0048 & -0.0189 & -0.0183 & -0.0043 & -0.0185 & -0.0184 & -0.0051 & -0.0189 & -0.0189 \\
            & Std  & 0.0492 & 0.0522 & 0.0519 & 0.0500 & 0.0528 & 0.0527 & 0.0492 & 0.0527 & 0.0527 \\
      %      & RMSE & 0.0495 & 0.0556 & 0.0551 & 0.0503 & 0.0560 & 0.0559 & 0.0496 & 0.0561 & 0.0561 \\
            \multicolumn{11}{c}{}\\
$\sigma$ & Bias & -0.0040 & -0.0064 & -0.0063 & -0.0027 & -0.0049 & -0.0049 & -0.0034 & -0.0055 & -0.0055 \\
         & Std  & 0.0454  & 0.0455  & 0.0455  & 0.0450  & 0.0451  & 0.0451  & 0.0443  & 0.0442  & 0.0442  \\
     %    & RMSE & 0.0456  & 0.0459  & 0.0459  & 0.0451  & 0.0454  & 0.0454  & 0.0444  & 0.0445  & 0.0445  \\
            \hline
        & & \multicolumn{3}{c|}{$d=$ 0.20} & \multicolumn{3}{c|}{$d=$ 0.30}  & \multicolumn{3}{c}{$d=$ 0.40}  \\
        \hline
        $\mu$ & Bias & - & -0.0013 & -0.0022 & - & 0.0069 & 0.0093 & - & 0.0301 & 0.0286 \\
            & Std  & - & 0.1811 & 0.1824 & - & 0.3464 & 0.3509 & - & 0.6823 & 0.6887 \\
       %     & RMSE & 0.0000 & 0.1814 & 0.1827 & 0.0000 & 0.3470 & 0.3516 & 0.0000 & 0.6840 & 0.6903 \\
            \multicolumn{11}{c}{}\\
        $d$ & Bias & -0.0046 & -0.0191 & -0.0191 & -0.0049 & -0.0220 & -0.0219 & -0.0129 & -0.0301 & -0.0299 \\
            & Std  & 0.0493 & 0.0529 & 0.0530 & 0.0467 & 0.0527 & 0.0528 & 0.0412 & 0.0462 & 0.0462 \\
       %     & RMSE & 0.0496 & 0.0563 & 0.0564 & 0.0470 & 0.0572 & 0.0572 & 0.0433 & 0.0552 & 0.0551 \\
            \multicolumn{11}{c}{}\\
$\sigma$ & Bias & -0.0047 & -0.0067 & -0.0067 & -0.0027 & -0.0049 & -0.0049 & -0.0034 & -0.0051 & -0.0050 \\
         & Std  & 0.0441  & 0.0442  & 0.0442  & 0.0430  & 0.0429  & 0.0429  & 0.0443  & 0.0443  & 0.0443  \\
    %     & RMSE & 0.0444  & 0.0447  & 0.0447  & 0.0431  & 0.0431  & 0.0431  & 0.0444  & 0.0446  & 0.0446  \\
        \hline
         \multicolumn{11}{c}{$n=1000$}\\
          \hline
       & & \multicolumn{3}{c|}{$d=$ -0.40} & \multicolumn{3}{c|}{$d=$ -0.30}  & \multicolumn{3}{c}{$d=$ -0.20}  \\
        \hline
        $\mu$ & Bias & - & -0.0001 & -0.0001 & - & 0.0001 & 0.0001 & - & -0.0000 & 0.0000 \\
            & Std  & - & 0.0028 & 0.0034 & - & 0.0052 & 0.0056 & - & 0.0088 & 0.0091 \\
    %          & RMSE & 0.0000 & 0.0028 & 0.0034 & 0.0000 & 0.0052 & 0.0056 & 0.0000 & 0.0088 & 0.0091 \\
            \multicolumn{11}{c}{}\\
        $d$ & Bias & -0.0008 & -0.0045 & -0.0023 & -0.0006 & -0.0046 & -0.0037 & -0.0012 & -0.0047 & -0.0043 \\
            & Std  & 0.0253 & 0.0260 & 0.0258 & 0.0244 & 0.0248 & 0.0247 & 0.0252 & 0.0255 & 0.0254 \\
     %       & RMSE & 0.0253 & 0.0264 & 0.0259 & 0.0244 & 0.0252 & 0.0249 & 0.0252 & 0.0259 & 0.0257 \\
            \multicolumn{11}{c}{}\\
     $\sigma$ & Bias & -0.0016 & -0.0021 & -0.0019 & -0.0006 & -0.0011 & -0.0010 & 0.0002 & -0.0003 & -0.0002 \\
         & Std  & 0.0217  & 0.0218  & 0.0217  & 0.0226  & 0.0226  & 0.0226  & 0.0217 & 0.0217  & 0.0217  \\
  %         & RMSE & 0.0218  & 0.0219  & 0.0218  & 0.0226  & 0.0226  & 0.0226  & 0.0217 & 0.0217  & 0.0217  \\
            \hline
        & & \multicolumn{3}{c|}{$d=$ -0.10} & \multicolumn{3}{c|}{$d=$ 0.00}  & \multicolumn{3}{c}{$d=$ 0.10}  \\
        \hline
        $\mu$ & Bias & - & 0.0012 & 0.0013 & - & -0.0001 & -0.0002 & - & -0.0016 & -0.0017 \\
            & Std  & - & 0.0165 & 0.0165 & - & 0.0315 & 0.0315 & - & 0.0597 & 0.0598 \\
    %          & RMSE & 0.0000 & 0.0165 & 0.0165 & 0.0000 & 0.0315 & 0.0315 & 0.0000 & 0.0597 & 0.0598 \\
            \multicolumn{11}{c}{}\\
        $d$ & Bias & -0.0017 & -0.0057 & -0.0056 & -0.0015 & -0.0055 & -0.0055 & -0.0014 & -0.0054 & -0.0054 \\
            & Std  & 0.0242 & 0.0249 & 0.0249 & 0.0246 & 0.0250 & 0.0250 & 0.0243 & 0.0247 & 0.0247 \\
    %          & RMSE & 0.0243 & 0.0255 & 0.0255 & 0.0246 & 0.0256 & 0.0256 & 0.0244 & 0.0253 & 0.0253 \\
            \multicolumn{11}{c}{}\\
      $\sigma$ & Bias & -0.0010 & -0.0015 & -0.0015 & -0.0018 & -0.0024 & -0.0024 & -0.0000 & -0.0005 & -0.0005 \\
         & Std  & 0.0221  & 0.0221  & 0.0221  & 0.0219  & 0.0218  & 0.0218  & 0.0228  & 0.0228  & 0.0228  \\
    %       & RMSE & 0.0222  & 0.0222  & 0.0222  & 0.0219  & 0.0220  & 0.0220  & 0.0228  & 0.0228  & 0.0228  \\
            \hline
        & & \multicolumn{3}{c|}{$d=$ 0.20} & \multicolumn{3}{c|}{$d=$ 0.30}  & \multicolumn{3}{c}{$d=$ 0.40}  \\
        \hline
        $\mu$ & Bias & - & -0.0033 & -0.0031 & - & -0.0061 & -0.0077 & - & 0.0397 & 0.0366 \\
            & Std  & - & 0.1177 & 0.1182 & - & 0.2532 & 0.2555 & - & 0.5473 & 0.5534 \\
    %          & RMSE & 0.0000 & 0.1177 & 0.1182 & 0.0000 & 0.2533 & 0.2556 & 0.0000 & 0.5487 & 0.5546 \\
           \multicolumn{11}{c}{}\\
        $d$ & Bias & -0.0036 & -0.0078 & -0.0078 & -0.0025 & -0.0070 & -0.0070 & -0.0045 & -0.0089 & -0.0088 \\
            & Std  & 0.0251 & 0.0259 & 0.0259 & 0.0240 & 0.0248 & 0.0248 & 0.0243 & 0.0255 & 0.0255 \\
     %         & RMSE & 0.0254 & 0.0270 & 0.0270 & 0.0242 & 0.0257 & 0.0257 & 0.0247 & 0.0270 & 0.0270 \\
            \multicolumn{11}{c}{}\\
      $\sigma$ & Bias & -0.0013 & -0.0018 & -0.0018 & -0.0019 & -0.0023 & -0.0023 & -0.0012 & -0.0016 & -0.0015 \\
         & Std  & 0.0228  & 0.0228  & 0.0228  & 0.0224  & 0.0225  & 0.0225  & 0.0221  & 0.0221  & 0.0221  \\
    %       & RMSE & 0.0229  & 0.0229  & 0.0229  & 0.0225  & 0.0226  & 0.0226  & 0.0222  & 0.0222  & 0.0222  \\
            \hline
        \hline
    \end{tabular}
}
\end{table}

\begin{table}[H]
\caption{Bias, Std and RMSE of alternative MLEs for fGn: $\mu=0$ and $\sigma=1$. MLE1 is MLE with known $\mu$; MLE2 is our exact MLE; MLE3 is PMLE.}
\label{table02}
\centering
\scalebox{0.82}{
    \begin{tabular}{cc|ccc|ccc|ccc}
        \hline
         \hline
        & & MLE1 & MLE2 & MLE3  & MLE1 & MLE2 & MLE3  & MLE1 & MLE2 & MLE3 \\
        \hline
          \multicolumn{11}{c}{$n=250$}\\
          \hline
        & & \multicolumn{3}{c|}{$H=$ 0.10} & \multicolumn{3}{c|}{$H=$ 0.20}  & \multicolumn{3}{c}{$H=$ 0.30}  \\
        \hline
        $\mu$ & Bias & - & 0.0001 & 0.0001 & - & 0.0002 & 0.0001 & - & -0.0000 & -0.0000 \\
            & Std  & - & 0.0031 & 0.0041 & - & 0.0035 & 0.0039 & - & 0.0039 & 0.0040 \\
        %    & RMSE & 0.0000 & 0.0031 & 0.0041 & 0.0000 & 0.0035 & 0.0039 & 0.0000 & 0.0039 & 0.0040 \\
            \multicolumn{11}{c}{}\\
        $H$ & Bias & 0.0006 & -0.0039 & 0.0002 & -0.0024 & -0.0079 & -0.0060 & -0.0027 & -0.0098 & -0.0091 \\
            & Std  & 0.0232 & 0.0236 & 0.0230 & 0.0302 & 0.0306 & 0.0302 & 0.0352 & 0.0362 & 0.0360 \\
       %       & RMSE & 0.0233 & 0.0240 & 0.0231 & 0.0303 & 0.0317 & 0.0309 & 0.0354 & 0.0376 & 0.0372 \\
            \multicolumn{11}{c}{}\\
      $\sigma$ & Bias & 0.0065 & -0.0154 & 0.0042 & -0.0033 & -0.0313 & -0.0221 & -0.0017 & -0.0382 & -0.0345 \\
         & Std  & 0.1157 & 0.1149  & 0.1153 & 0.1519  & 0.1501  & 0.1498  & 0.1833  & 0.1824  & 0.1820  \\
     %      & RMSE & 0.1159 & 0.1160  & 0.1153 & 0.1520  & 0.1534  & 0.1514  & 0.1833  & 0.1863  & 0.1853\\
            \hline
        & & \multicolumn{3}{c|}{$H=$ 0.40} & \multicolumn{3}{c|}{$H=$ 0.50}  & \multicolumn{3}{c}{$H=$ 0.60}  \\
        \hline
        $\mu$ & Bias & - & -0.0001 & -0.0001 & - & -0.0000 & -0.0000 & - & -0.0003 & -0.0003 \\
            & Std  & - & 0.0038 & 0.0038 & - & 0.0040 & 0.0040 & - & 0.0040 & 0.0040 \\
    %          & RMSE & 0.0000 & 0.0038 & 0.0038 & 0.0000 & 0.0040 & 0.0040 & 0.0000 & 0.0040 & 0.0040 \\
            \multicolumn{11}{c}{}\\
        $H$ & Bias & -0.0021 & -0.0097 & -0.0094 & -0.0015 & -0.0108 & -0.0107 & -0.0037 & -0.0139 & -0.0139 \\
            & Std  & 0.0389 & 0.0399 & 0.0398 & 0.0406 & 0.0425 & 0.0424 & 0.0414 & 0.0430 & 0.0430 \\
     %         & RMSE & 0.0390 & 0.0412 & 0.0410 & 0.0407 & 0.0439 & 0.0438 & 0.0416 & 0.0452 & 0.0453 \\
            \multicolumn{11}{c}{}\\
     $\sigma$ & Bias & 0.0083 & -0.0319 & -0.0308 & 0.0147 & -0.0357 & -0.0357 & 0.0080 & -0.0497 & -0.0499 \\
         & Std  & 0.2155 & 0.2125  & 0.2122  & 0.2348 & 0.2330  & 0.2327  & 0.2497 & 0.2449  & 0.2449  \\
    %       & RMSE & 0.2157 & 0.2149  & 0.2144  & 0.2353 & 0.2357  & 0.2355  & 0.2498 & 0.2499  & 0.2499  \\
            \hline
        & & \multicolumn{3}{c|}{$H=$ 0.70} & \multicolumn{3}{c|}{$H=$ 0.80}  & \multicolumn{3}{c}{$H=$ 0.90}  \\
        \hline
        $\mu$ & Bias & - & -0.0000 & 0.0000 & - & -0.0002 & -0.0002 & - & 0.0001 & 0.0000 \\
            & Std  & - & 0.0040 & 0.0040 & - & 0.0041 & 0.0042 & - & 0.0038 & 0.0038 \\
    %          & RMSE & 0.0000 & 0.0040 & 0.0040 & 0.0000 & 0.0041 & 0.0042 & 0.0000 & 0.0038 & 0.0038 \\
            \multicolumn{11}{c}{}\\
        $H$ & Bias & -0.0018 & -0.0129 & -0.0129 & -0.0005 & -0.0139 & -0.0139 & -0.0067 & -0.0212 & -0.0210 \\
            & Std  & 0.0416 & 0.0437 & 0.0438 & 0.0402 & 0.0433 & 0.0433 & 0.0379 & 0.0405 & 0.0405 \\
    %          & RMSE & 0.0417 & 0.0457 & 0.0457 & 0.0402 & 0.0455 & 0.0455 & 0.0385 & 0.0458 & 0.0457 \\
            \multicolumn{11}{c}{}\\
      $\sigma$ & Bias & 0.0284 & -0.0398 & -0.0397 & 0.0483 & -0.0437 & -0.0430 & 0.0377 & -0.0942 & -0.0926 \\
         & Std  & 0.2771 & 0.2706  & 0.2709  & 0.3256 & 0.3173  & 0.3179  & 0.4020 & 0.3565  & 0.3577  \\
    %       & RMSE & 0.2786 & 0.2735  & 0.2738  & 0.3291 & 0.3203  & 0.3208  & 0.4037 & 0.3687  & 0.3695  \\
          \hline
           \multicolumn{11}{c}{$n=1000$}\\
          \hline
        & & \multicolumn{3}{c|}{$H=$ 0.10} & \multicolumn{3}{c|}{$H=$ 0.20}  & \multicolumn{3}{c}{$H=$ 0.30}  \\
        \hline
        $\mu$ & Bias & - & -0.0000 & -0.0000 & - & -0.0000 & -0.0001 & - & -0.0001 & -0.0001 \\
            & Std  & - & 0.0008 & 0.0011 & - & 0.0012 & 0.0014 & - & 0.0015 & 0.0015 \\
       %       & RMSE & 0.0000 & 0.0008 & 0.0011 & 0.0000 & 0.0012 & 0.0014 & 0.0000 & 0.0015 & 0.0015 \\
            \multicolumn{11}{c}{}\\
        $H$ & Bias & 0.0000 & -0.0011 & -0.0001 & 0.0006 & -0.0011 & -0.0007 & -0.0006 & -0.0027 & -0.0025 \\
            & Std  & 0.0115 & 0.0116 & 0.0116 & 0.0152 & 0.0153 & 0.0153 & 0.0171 & 0.0175 & 0.0174 \\
        %      & RMSE & 0.0115 & 0.0116 & 0.0116 & 0.0152 & 0.0154 & 0.0153 & 0.0171 & 0.0177 & 0.0176 \\
            \multicolumn{11}{c}{}\\
      $\sigma$ & Bias & 0.0013 & -0.0043 & 0.0005 & 0.0065 & -0.0026 & -0.0003 & 0.0012 & -0.0099 & -0.0089 \\
         & Std  & 0.0575 & 0.0575  & 0.0580 & 0.0787 & 0.0789  & 0.0788  & 0.0887 & 0.0895  & 0.0894  \\
     %      & RMSE & 0.0576 & 0.0577  & 0.0580 & 0.0790 & 0.0789  & 0.0789  & 0.0888 & 0.0900  & 0.0898\\
            \hline
        & & \multicolumn{3}{c|}{$H=$ 0.40} & \multicolumn{3}{c|}{$H=$ 0.50}  & \multicolumn{3}{c}{$H=$ 0.60}  \\
        \hline
        $\mu$ & Bias & - & -0.0001 & -0.0001 & - & -0.0001 & -0.0001 & - & -0.0002 & -0.0002 \\
            & Std  & - & 0.0018 & 0.0018 & - & 0.0020 & 0.0020 & - & 0.0023 & 0.0023 \\
     %         & RMSE & 0.0000 & 0.0018 & 0.0018 & 0.0000 & 0.0020 & 0.0020 & 0.0000 & 0.0023 & 0.0023 \\
            \multicolumn{11}{c}{}\\
        $H$ & Bias & -0.0002 & -0.0028 & -0.0027 & -0.0010 & -0.0037 & -0.0037 & -0.0003 & -0.0033 & -0.0033 \\
            & Std  & 0.0192 & 0.0194 & 0.0194 & 0.0189 & 0.0192 & 0.0192 & 0.0206 & 0.0210 & 0.0210 \\
     %         & RMSE & 0.0192 & 0.0196 & 0.0196 & 0.0190 & 0.0196 & 0.0196 & 0.0206 & 0.0213 & 0.0213 \\
            \multicolumn{11}{c}{}\\
       $\sigma$ & Bias & 0.0022 & -0.0116 & -0.0113 & -0.0017 & -0.0169 & -0.0169 & 0.0060 & -0.0116 & -0.0116 \\
         & Std  & 0.1037 & 0.1035  & 0.1034  & 0.1075  & 0.1072  & 0.1072  & 0.1230 & 0.1231  & 0.1231  \\
      %     & RMSE & 0.1037 & 0.1041  & 0.1040  & 0.1076  & 0.1085  & 0.1085  & 0.1231 & 0.1237  & 0.1237 \\
            \hline
        & & \multicolumn{3}{c|}{$H=$ 0.70} & \multicolumn{3}{c|}{$H=$ 0.80}  & \multicolumn{3}{c}{$H=$ 0.90}  \\
        \hline
        $\mu$ & Bias & - & 0.0000 & -0.0000 & - & 0.0000 & 0.0001 & - & -0.0001 & -0.0001 \\
            & Std  & - & 0.0027 & 0.0027 & - & 0.0031 & 0.0031 & - & 0.0034 & 0.0034 \\
       %       & RMSE & 0.0000 & 0.0027 & 0.0027 & 0.0000 & 0.0031 & 0.0031 & 0.0000 & 0.0034 & 0.0034 \\
            \multicolumn{11}{c}{}\\
        $H$ & Bias & -0.0001 & -0.0035 & -0.0035 & 0.0007 & -0.0030 & -0.0030 & -0.0018 & -0.0065 & -0.0064 \\
            & Std  & 0.0203 & 0.0208 & 0.0208 & 0.0203 & 0.0208 & 0.0208 & 0.0211 & 0.0216 & 0.0216 \\
     %         & RMSE & 0.0203 & 0.0210 & 0.0211 & 0.0203 & 0.0210 & 0.0210 & 0.0212 & 0.0226 & 0.0226 \\
            \multicolumn{11}{c}{}\\
       $\sigma$ & Bias & 0.0079 & -0.0135 & -0.0134 & 0.0185 & -0.0089 & -0.0086 & 0.0155 & -0.0289 & -0.0283 \\
         & Std  & 0.1302 & 0.1296  & 0.1296  & 0.1500 & 0.1490  & 0.1490  & 0.2205 & 0.2136  & 0.2138  \\
     %      & RMSE & 0.1304 & 0.1303  & 0.1303  & 0.1511 & 0.1492  & 0.1493  & 0.2210 & 0.2155  & 0.2157  \\
            \hline
        \hline
    \end{tabular}
}
\end{table}

Table \ref{table01} reports the results for ARFIMA and Table \ref{table02} reports the results for fGn. From Tables \ref{table01}-\ref{table02}, we observe the following findings. First, in terms of convergence rates, the performance of MLE2 aligns well with our asymptotic theory. For $\mu$, the convergence rate is $n^{-(1 - \alpha_X(\xi))/2}$, which becomes slower as $d$ increases toward $1/2$ from $-1/2$ (or as $H$ increases toward $1$ from $0$). A similar pattern can be observed for the sample mean, as it shares the same convergence rate as given in (\ref{Assump:mu-PMLE}) and (\ref{lower-bound-mu}). For the remaining parameters, the convergence rate remains at the root-$n$. Second, MLE2 always performs better than MLE3, except when $d = 0$ in ARFIMA or $H = 1/2$ in fGn, where the two methods perform nearly identically. Third, interestingly, this superior performance in estimating $\mu$ by MLE2 does not translate into better performance in estimating other parameters. Using the true value of $\mu$, MLE1 does not lead to a better performance in estimating other parameters. The three ML methods lead to a similar finite sample performance for parameters other than $\mu$.

To see how the relative inefficiency of the sample mean over MLE2 of $\mu$, the two dashed lines in Figure \ref{fig:asy_eff_mu} plot the ratio of the sample variance of MLE2 for $\mu$ to that of MLE3 as a function of $d$  for the two models when $n = 1000$. Clearly, the relative inefficiency goes up rapidly as $d$ nears $-0.5$ in ARFIMA and fGn. For comparison, also plotted by the solid line is the theoretical asymptotic inefficiency given in (\ref{aseff_sm}). The red dashed line is closely aligned with the theory. 

\begin{figure}[H]
    \centering
    \includegraphics[scale=0.3]{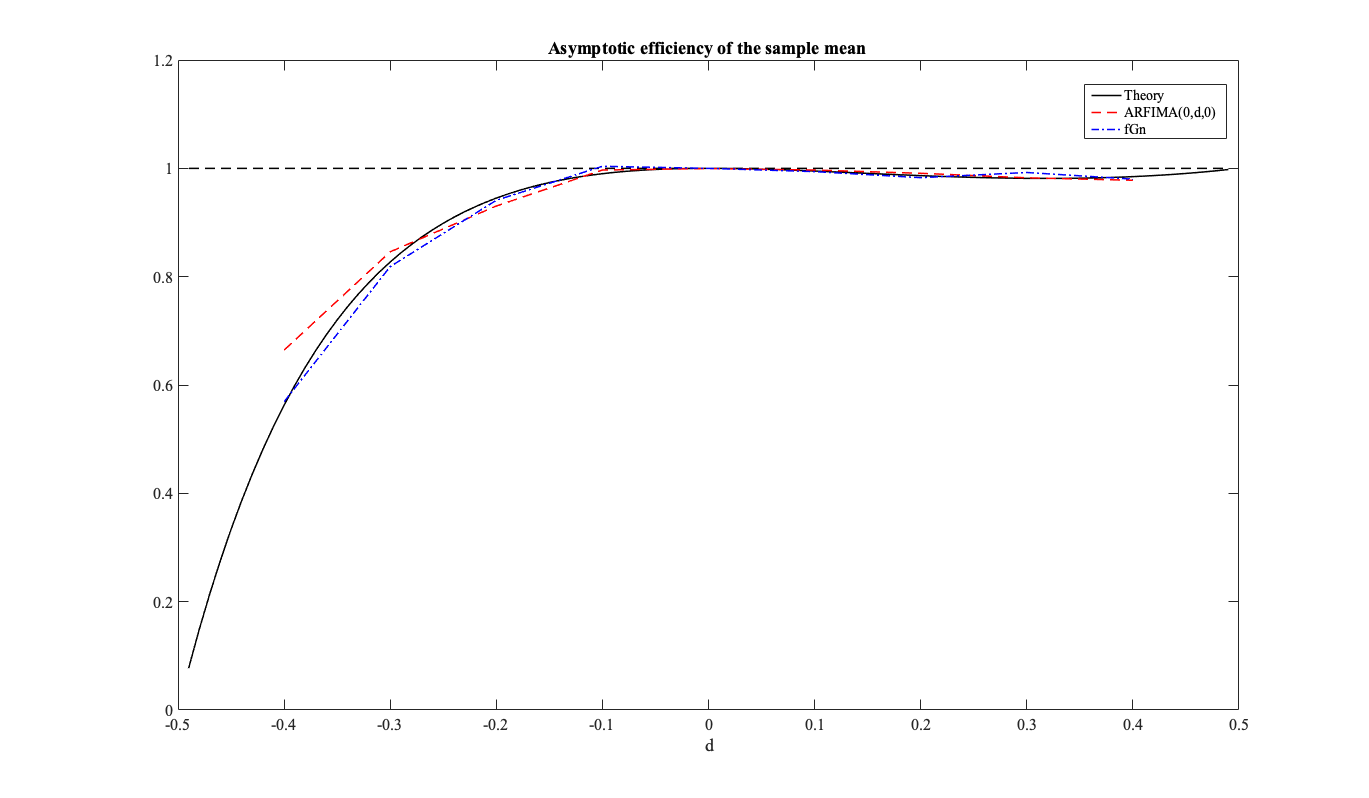}
    \caption{Relative inefficiency of the sample mean over the exact MLE as a function of $d$ for ARFIMA$(0,d,0)$ and fGn when $n=1000$. For fGn, $d=H-1/2$.}
   \label{fig:asy_eff_mu}
\end{figure}

\begin{table}[H]
\caption{Bias, Std and RMSE of alternative MLEs for fOU: $\kappa=10$, $\sigma=1$ and $n=250$. MLE1 is MLE with known $\mu$; MLE2 is our exact MLE; MLE3 is PMLE.}\label{table05}
\centering
\scalebox{0.85}{
    \begin{tabular}{cc|ccc|ccc|ccc}
        \hline
        & & MLE1 & MLE2 & MLE3  & MLE1 & MLE2 & MLE3  & MLE1 & MLE2 & MLE3 \\
        \hline
        & & \multicolumn{3}{c|}{$H=$ 0.10} & \multicolumn{3}{c|}{$H=$ 0.20}  & \multicolumn{3}{c}{$H=$ 0.30}  \\
        \hline
        $\mu$ & Bias & - & -0.0021 & -0.0022 & - & 0.0004 & -0.0004 & - & 0.0014 & 0.0003 \\
            & Std  & - & 0.1166 & 0.1005 & - & 0.1047 & 0.0978 & - & 0.0984 & 0.0956 \\
            & RMSE & - & 0.1166 & 0.1005 & - & 0.1047 & 0.0978 & - & 0.0984 & 0.0956 \\
            \multicolumn{11}{c}{}\\
        $H$ & Bias & 0.0036 & 0.0080 & 0.0081 & 0.0036 & 0.0081 & 0.0084 & 0.0038 & 0.0078 & 0.0082 \\
            & Std  & 0.0297 & 0.0307 & 0.0306 & 0.0402 & 0.0403 & 0.0403 & 0.0462 & 0.0455 & 0.0455 \\
            & RMSE & 0.0300 & 0.0317 & 0.0317 & 0.0404 & 0.0411 & 0.0411 & 0.0463 & 0.0462 & 0.0462 \\
            \multicolumn{11}{c}{}\\
        $\kappa$ & Bias & -0.2783 & 3.3820 & 3.3933 & 1.0258 & 4.3478 & 4.3914 & 1.7384 & 4.7370 & 4.7953 \\
            & Std  & 8.0909 & 10.0533 & 9.9980 & 7.3880 & 8.5961 & 8.5568 & 7.1869 & 8.0556 & 8.0456 \\
            & RMSE & 8.0957 & 10.6069 & 10.5582 & 7.4588 & 9.6331 & 9.6179 & 7.3941 & 9.3451 & 9.3662 \\
            \multicolumn{11}{c}{}\\
      $\sigma$ & Bias & 0.0265 & 0.0521 & 0.0528 & 0.0376 & 0.0648 & 0.0663 & 0.0479 & 0.0721 & 0.0742 \\
         & Std  & 0.1596 & 0.1743 & 0.1743 & 0.2296 & 0.2409 & 0.2409 & 0.2774 & 0.2796 & 0.2794 \\
         & RMSE & 0.1618 & 0.1819 & 0.1821 & 0.2327 & 0.2495 & 0.2499 & 0.2815 & 0.2887 & 0.2891\\
            \hline
        & & \multicolumn{3}{c|}{$H=$ 0.40} & \multicolumn{3}{c|}{$H=$ 0.50}  & \multicolumn{3}{c}{$H=$ 0.60}  \\
        \hline
        $\mu$ & Bias & - & 0.0020 & 0.0006 & - & 0.0018 & 0.0003 & - & 0.0022 & 0.0005 \\
            & Std  & - & 0.0939 & 0.0938 & - & 0.0899 & 0.0919 & - & 0.0862 & 0.0902 \\
            & RMSE & - & 0.0940 & 0.0938 & - & 0.0899 & 0.0919 & - & 0.0863 & 0.0902 \\
            \multicolumn{11}{c}{}\\
        $H$ & Bias & 0.0041 & 0.0075 & 0.0079 & 0.0046 & 0.0071 & 0.0075 & 0.0048 & 0.0059 & 0.0065 \\
            & Std  & 0.0500 & 0.0488 & 0.0487 & 0.0527 & 0.0509 & 0.0509 & 0.0538 & 0.0516 & 0.0515 \\
            & RMSE & 0.0501 & 0.0493 & 0.0493 & 0.0530 & 0.0514 & 0.0514 & 0.0541 & 0.0519 & 0.0519 \\
            \multicolumn{11}{c}{}\\
        $\kappa$ & Bias & 2.1900 & 4.8790 & 4.9283 & 2.4532 & 4.8017 & 4.8287 & 2.4756 & 4.4301 & 4.4500 \\
            & Std  & 7.1367 & 7.7810 & 7.7719 & 7.1795 & 7.5057 & 7.5117 & 7.0800 & 7.2298 & 7.2463 \\
            & RMSE & 7.4651 & 9.1841 & 9.2027 & 7.5870 & 8.9102 & 8.9298 & 7.5003 & 8.4791 & 8.5035 \\
            \multicolumn{11}{c}{}\\
      $\sigma$ & Bias & 0.0577 & 0.0776 & 0.0801 & 0.0699 & 0.0831 & 0.0860 & 0.0801 & 0.0834 & 0.0875 \\
         & Std  & 0.3104 & 0.3088 & 0.3086 & 0.3485 & 0.3373 & 0.3378 & 0.3763 & 0.3568 & 0.3581 \\
         & RMSE & 0.3157 & 0.3184 & 0.3188 & 0.3554 & 0.3474 & 0.3486 & 0.3847 & 0.3664 & 0.3686\\
            \hline
        & & \multicolumn{3}{c|}{$H=$ 0.70} & \multicolumn{3}{c|}{$H=$ 0.80}  & \multicolumn{3}{c}{$H=$ 0.90}  \\
        \hline
        $\mu$ & Bias & - & 0.0024 & 0.0006 & - & 0.0024 & 0.0005 & - & 0.0027 & 0.0004 \\
            & Std  & - & 0.0831 & 0.0892 & - & 0.0794 & 0.0882 & - & 0.0776 & 0.0873 \\
            & RMSE & - & 0.0831 & 0.0892 & - & 0.0794 & 0.0882 & - & 0.0777 & 0.0873 \\
            \multicolumn{11}{c}{}\\
        $H$ & Bias & 0.0040 & 0.0039 & 0.0045 & 0.0038 & 0.0013 & 0.0023 & 0.0033 & -0.0016 & -0.0009 \\
            & Std  & 0.0531 & 0.0510 & 0.0509 & 0.0516 & 0.0494 & 0.0494 & 0.0415 & 0.0425 & 0.0422 \\
            & RMSE & 0.0532 & 0.0512 & 0.0511 & 0.0518 & 0.0494 & 0.0494 & 0.0416 & 0.0425 & 0.0422 \\
            \multicolumn{11}{c}{}\\
        $\kappa$ & Bias & 2.0427 & 3.6373 & 3.6150 & 1.2981 & 2.2510 & 2.2808 & -0.4930 & -0.0475 & -0.1009 \\
            & Std  & 6.7179 & 6.8232 & 6.8247 & 6.2694 & 6.1526 & 6.1850 & 3.9814 & 4.3394 & 4.3388 \\
            & RMSE & 7.0215 & 7.7321 & 7.7230 & 6.4023 & 6.5514 & 6.5921 & 4.0118 & 4.3397 & 4.3399 \\
            \multicolumn{11}{c}{}\\
  $\sigma$ & Bias & 0.0862 & 0.0809 & 0.0850 & 0.1143 & 0.0839 & 0.0922 & 0.1704 & 0.1174 & 0.1247 \\
         & Std  & 0.3987 & 0.3789 & 0.3800 & 0.4584 & 0.4182 & 0.4205 & 0.5160 & 0.5082 & 0.5106 \\
         & RMSE & 0.4079 & 0.3874 & 0.3894 & 0.4724 & 0.4266 & 0.4305 & 0.5434 & 0.5216 & 0.5256 \\
        \hline
    \end{tabular}
}
\end{table}

\begin{table}[H]
\caption{Bias, Std and RMSE of alternative MLEs for fOU: $\kappa=10$, $\sigma=1$ and $n=1000$. MLE1 is MLE with known $\mu$; MLE2 is our exact MLE; MLE3 is PMLE.}\label{table06}
\centering
\scalebox{0.85}{
    \begin{tabular}{cc|ccc|ccc|ccc}
        \hline
        & & MLE1 & MLE2 & MLE3  & MLE1 & MLE2 & MLE3  & MLE1 & MLE2 & MLE3 \\
        & & \multicolumn{3}{c|}{$H=$ 0.10} & \multicolumn{3}{c|}{$H=$ 0.20}  & \multicolumn{3}{c}{$H=$ 0.30}  \\
        \hline
        $\mu$ & Bias & - & 0.0010 & 0.0008 & - & 0.0012 & 0.0011 & - & 0.0014 & 0.0012 \\
            & Std  & - & 0.0289 & 0.0303 & - & 0.0319 & 0.0325 & - & 0.0368 & 0.0368 \\
            & RMSE & - & 0.0289 & 0.0303 & - & 0.0320 & 0.0326 & - & 0.0368 & 0.0368 \\
            \multicolumn{11}{c}{}\\
        $H$ & Bias & 0.0010 & 0.0016 & 0.0015 & 0.0008 & 0.0014 & 0.0014 & 0.0009 & 0.0014 & 0.0014 \\
            & Std  & 0.0145 & 0.0145 & 0.0144 & 0.0191 & 0.0190 & 0.0190 & 0.0222 & 0.0220 & 0.0220 \\
            & RMSE & 0.0145 & 0.0145 & 0.0145 & 0.0192 & 0.0190 & 0.0190 & 0.0222 & 0.0220 & 0.0220 \\
            \multicolumn{11}{c}{}\\
        $\kappa$ & Bias & -0.1356 & 0.4662 & 0.4094 & 0.3459 & 0.9577 & 0.9334 & 0.5840 & 1.1698 & 1.1679 \\
            & Std  & 3.7385 & 3.7798 & 3.7731 & 3.3303 & 3.3923 & 3.3894 & 3.3122 & 3.3685 & 3.3658 \\
            & RMSE & 3.7410 & 3.8085 & 3.7953 & 3.3482 & 3.5249 & 3.5155 & 3.3633 & 3.5658 & 3.5626 \\
            \multicolumn{11}{c}{}\\
        $\sigma$ & Bias & 0.0065 & 0.0100 & 0.0094 & 0.0059 & 0.0095 & 0.0094 & 0.0067 & 0.0097 & 0.0099 \\
         & Std  & 0.0747 & 0.0748 & 0.0747 & 0.1021 & 0.1017 & 0.1018 & 0.1221 & 0.1215 & 0.1215 \\
         & RMSE & 0.0750 & 0.0755 & 0.0753 & 0.1023 & 0.1022 & 0.1022 & 0.1223 & 0.1219 & 0.1219\\
            \hline
        & & \multicolumn{3}{c|}{$H=$ 0.40} & \multicolumn{3}{c|}{$H=$ 0.50}  & \multicolumn{3}{c}{$H=$ 0.60}  \\
        \hline
        $\mu$ & Bias & - & 0.0015 & 0.0014 & - & 0.0016 & 0.0015 & - & 0.0017 & 0.0016 \\
            & Std  & - & 0.0420 & 0.0420 & - & 0.0475 & 0.0481 & - & 0.0534 & 0.0551 \\
            & RMSE & - & 0.0420 & 0.0421 & - & 0.0475 & 0.0481 & - & 0.0534 & 0.0551 \\
            \multicolumn{11}{c}{}\\
        $H$ & Bias & 0.0010 & 0.0014 & 0.0014 & 0.0012 & 0.0013 & 0.0014 & 0.0015 & 0.0012 & 0.0013 \\
            & Std  & 0.0243 & 0.0241 & 0.0241 & 0.0259 & 0.0255 & 0.0255 & 0.0269 & 0.0265 & 0.0265 \\
            & RMSE & 0.0243 & 0.0241 & 0.0241 & 0.0259 & 0.0256 & 0.0256 & 0.0270 & 0.0266 & 0.0265 \\
            \multicolumn{11}{c}{}\\
        $\kappa$ & Bias & 0.7287 & 1.2778 & 1.2806 & 0.8155 & 1.3301 & 1.3322 & 0.8647 & 1.3207 & 1.3197 \\
            & Std  & 3.3593 & 3.3948 & 3.3937 & 3.4057 & 3.4304 & 3.4298 & 3.4398 & 3.4390 & 3.4412 \\
            & RMSE & 3.4374 & 3.6273 & 3.6273 & 3.5020 & 3.6793 & 3.6795 & 3.5468 & 3.6839 & 3.6856 \\
            \multicolumn{11}{c}{}\\
   $\sigma$ & Bias & 0.0078 & 0.0099 & 0.0101 & 0.0093 & 0.0102 & 0.0103 & 0.0120 & 0.0105 & 0.0107 \\
         & Std  & 0.1388 & 0.1376 & 0.1375 & 0.1539 & 0.1521 & 0.1520 & 0.1700 & 0.1670 & 0.1669 \\
         & RMSE & 0.1390 & 0.1380 & 0.1379 & 0.1542 & 0.1524 & 0.1524 & 0.1704 & 0.1673 & 0.1673 \\
            \hline
        & & \multicolumn{3}{c|}{$H=$ 0.70} & \multicolumn{3}{c|}{$H=$ 0.80}  & \multicolumn{3}{c}{$H=$ 0.90}  \\
        \hline
        $\mu$ & Bias & - & 0.0017 & 0.0017 & - & 0.0017 & 0.0017 & - & 0.0018 & 0.0016 \\
            & Std  & - & 0.0597 & 0.0631 & - & 0.0663 & 0.0723 & - & 0.0728 & 0.0827 \\
            & RMSE & - & 0.0597 & 0.0631 & - & 0.0663 & 0.0723 & - & 0.0728 & 0.0827 \\
            \multicolumn{11}{c}{}\\
        $H$ & Bias & 0.0015 & 0.0009 & 0.0010 & 0.0017 & 0.0004 & 0.0004 & 0.0029 & 0.0004 & 0.0007 \\
            & Std  & 0.0274 & 0.0271 & 0.0271 & 0.0271 & 0.0267 & 0.0267 & 0.0225 & 0.0225 & 0.0225 \\
            & RMSE & 0.0274 & 0.0271 & 0.0271 & 0.0272 & 0.0267 & 0.0267 & 0.0227 & 0.0225 & 0.0225 \\
            \multicolumn{11}{c}{}\\
        $\kappa$ & Bias & 0.7907 & 1.2154 & 1.2087 & 0.6244 & 0.9236 & 0.9076 & 0.0312 & 0.1706 & 0.1774 \\
            & Std  & 3.4432 & 3.4458 & 3.4471 & 3.2790 & 3.2548 & 3.2523 & 2.0054 & 1.9812 & 1.9411 \\
            & RMSE & 3.5329 & 3.6538 & 3.6529 & 3.3379 & 3.3833 & 3.3766 & 2.0056 & 1.9885 & 1.9492 \\
            \multicolumn{11}{c}{}\\
     $\sigma$ & Bias & 0.0140 & 0.0106 & 0.0109 & 0.0230 & 0.0123 & 0.0125 & 0.0530 & 0.0268 & 0.0296 \\
         & Std  & 0.1878 & 0.1848 & 0.1845 & 0.2200 & 0.2122 & 0.2124 & 0.2549 & 0.2464 & 0.2479 \\
         & RMSE & 0.1883 & 0.1851 & 0.1849 & 0.2212 & 0.2126 & 0.2127 & 0.2603 & 0.2478 & 0.2497 \\
        \hline
    \end{tabular}
}
\end{table}

Tables \ref{table05}-\ref{table06} report the results for fOU, from which we observe the following findings. First, in terms of convergence rates, the performance of MLE2 aligns well with our asymptotic theory. For $\mu$, the convergence rate is root-$n$ when $H \leq 1/2$, and transitions to $n^{1 - H}$ when $H > 1/2$. The standard deviation of the estimator for $\mu$ decreases substantially as the sample size increases from $250$ to $1000$ when $H \leq 1/2$; however, as $H$ approaches $1$, the percentage reduction becomes markedly smaller. A similar pattern can be observed for the sample mean, as it shares the same convergence rate as given in (\ref{Assump:mu-PMLE}) and (\ref{lower-bound-mu}). For the remaining parameters, the convergence rate remains at the  root-$n$. Second, we see a clear dominance of MLE2 over MLE3 in terms of finite sample performance of estimates of $\mu$, when both $n$ is large and $H$ is near either zero or one. Third, this superior performance in estimating $\mu$ by MLE2 does not translate into a better performance in estimating other parameters. Using the true value of $\mu$, MLE1 does not lead to better performance in estimating the other 3 parameters. The three ML methods lead to a similar finite sample performance for parameters other than $\mu$.\footnote{In the online supplement (Section \ref{Sec:fore_fou}), we conduct a forecasting horse race for realized volatility using the fOU process with three alternative estimators: MLE2, MLE3, and the CoF estimator. As expected, MLE2 delivers the best forecasting performance, followed by MLE3 and then the CoF estimator.}

\section{Conclusion}
\label{Sec:con}
 Gaussian processes have gained significant attention due to their broad applicability across various scientific and applied disciplines. To obtain the MLE, two common approaches are typically employed. The first approach maximizes the likelihood assuming \( \mu \) is known and set to $0$, which results in an unrealistic MLE. The second approach uses the sample mean as an estimator for \( \mu \), leading to a plug-in MLE. However, both methods fail to address the inefficiency of the estimator for \( \mu \), and concerns have been raised about the finite sample performance of the plug-in MLE. \cite{Adenstedt-1974} proposed an efficient but infeasible estimator for \( \mu \).

In this paper, we introduce a novel exact ML method for all the parameters for general Gaussian processes with long-memory, short-memory, or anti-persistence properties. We prove that the exact MLE exhibits the properties of consistency and asymptotic normality. We also establish the LAN property of the sequence of statistical experiments for general
Gaussian processes in the sense of Le Cam, which directly yields efficiency.  Our method offers a comprehensive understanding of MLE for fractional Gaussian models: first, we show that the estimators for all parameters are optimal, effectively complementing the infeasible estimator for $\mu$ proposed by \cite{Adenstedt-1974}; second, we evaluate the difference between the plug-in MLE, exact MLE, and the MLE with known $\mu$.  The plug-in MLE performs as good as the exact MLE for all parameters except for \( \mu \). The discrepancy between plug-in MLE and the MLE with known $\mu$ is not due to an inefficient estimator for \( \mu \).

The Whittle MLE is asymptotically equivalent to the exact MLE under certain regularity conditions. Although its finite-sample performance is generally inferior to that of the exact MLE, the performance gap narrows as the sample size increases. At the same time, the computational burden of the exact MLE increases substantially due to the need to invert the covariance matrix at each evaluation of the likelihood function—a step that the Whittle method avoids. The Whittle ML method remains an attractive alternative. However, existing theoretical results for the Whittle MLE primarily pertain to ARFIMA models and do not extend to continuous-time models. In future work, we aim to investigate the optimality of the Whittle MLE for general Gaussian processes.

%%%%%%%%%%%%%%%%%%%%%%%%%%%%%%%%%%%%%%%%%%%%%%%%%%
\noindent %\medskip %\baselineskip=12pt 
\bibliographystyle{Chicago}
\bibliography{Reference}
%%%%%%%%%%%%%%%%%%%%%%%%%%%%%%%%%%%%%%%%%%%%%%%%%%

\appendix
\section{Appendix: Proof of Theorems}
\subsection{Proof of Consistency in Theorem~\ref{Thm:MLE}}\label{Sec:proof-Thm:MLE-consistency}
Let $\mathbf{Z}_{n}:=\mathbf{X}_{n}-\mu_{0}\mathbf{1}_{n}$ and
\begin{align*}
    \bar{\ell}_{n}(\xi) := \ell_{n}((\xi,\sigma_{n}(\xi),\mu_{n}(\xi))^{\top})
    = -\frac{n}{2}\log(2\pi)
    -\frac{n}{2}\log{\sigma^{2}}
    -\frac{1}{2}\log\mathrm{det}\bigl[\Sigma_{n}(s_{\xi}^{X})\bigr]
    -\frac{1}{2\sigma^{2}}\bar{\sigma}^{2}_{n}(\xi),
\end{align*}
where $\theta=(\xi,\sigma)^{\top}$. Additionally, we introduce
\begin{align}\label{expr:sig2-xi}
    \tilde{\sigma}^{2}_{n}(\xi) := \sigma_{n}^{2}((\xi,\mu_{0})^{\top})
    = \frac{1}{n}\left(\mathbf{X}_{n}-\mu_{0}\mathbf{1}_{n}\right)^{\top}\Sigma_{n}(s_{\xi}^{X})^{-1}\left(\mathbf{X}_{n}-\mu_{0}\mathbf{1}_{n}\right)
    \ \ \mbox{and}\ \ 
    \sigma^{2}(\xi) := \frac{\sigma^{2}_{0}}{2\pi}\int_{-\pi}^{\pi}\frac{s_{\xi_{0}}^{X}(\omega)}{s_{\xi}^{X}(\omega)}\dd\omega.
\end{align}

Let $\iota\in(0,1)$. 
For $\alpha_{X}(\xi)$, we introduce a restricted parameter space of $\Theta_{\xi}$ by 
\begin{align*}
    \Theta_{\xi}(\iota) := \{\xi\in\Theta_{\xi}:\alpha_{X}(\xi)-\alpha_{X}(\xi_{0})\geq -1+\iota\}.
\end{align*}
For a $\R$-valued function $f$ on some set $A$, we write $f_{\pm}(a):=\max\{\pm f(a),0\}$ for $a\in A$. Then we have $f=f_{+}-f_{-}$.

Recall that $\bar{\ell}_{n}(\xi) = \ell_{n}((\xi,\sigma_{n}(\xi),\mu_{n}(\xi))^{\top})$ that can be written as
\begin{align}
    \bar{\ell}_{n}(\xi) =& -\frac{n}{2}\log(2\pi) -\frac{n}{2}\log{\bar{\sigma}^{2}_{n}(\xi)} 
    -\frac{1}{2}\log\mathrm{det}\bigl[\Sigma_{n}(s_{\xi}^{X})\bigr]
    -\frac{1}{2\bar{\sigma}^{2}_{n}(\xi)}n\bar{\sigma}^{2}_{n}(\xi) 
    \nonumber \\
    =& -\frac{n}{2}(1+\log(2\pi)) -\frac{n}{2}\log{\bar{\sigma}^{2}_{n}(\xi)} 
    -\frac{1}{2}\log\mathrm{det}\bigl[\Sigma_{n}(s_{\xi}^{X})\bigr]. 
    \label{expr:ell-bar}
\end{align}
Moreover, we also recall a restricted parameter space of $\Theta_{\xi}$ defined by 
\begin{align*}
    \Theta_{\xi}(\iota) := \{\xi\in\Theta_{\xi}:\alpha_{X}(\xi)-\alpha_{X}(\xi_{0})\geq -1+\iota\}.
\end{align*}

Set $L_{n}(\xi):=-\frac{2}{n}\bar{\ell}_{n}(\xi)$ and $\sigma^{2}(\xi):=\frac{\sigma^{2}_{0}}{2\pi}\int_{-\pi}^{\pi}\frac{s_{\xi_{0}}^{X}(\omega)}{s_{\xi}^{X}(\omega)}\dd\omega$. Now we introduce a function $L(\xi)$ by
\begin{align*}
    L(\xi) := (1+\log(2\pi)) + \log\left( \frac{1}{2\pi}\int_{-\pi}^{\pi}\frac{s_{\xi_{0}}^{X}(\omega)}{s_{\xi}^{X}(\omega)}\dd\omega \right)
    +\frac{1}{2\pi}\int_{-\pi}^{\pi}\log{s_{\xi}^{X}(\omega)}\dd\omega,
\end{align*}
which is actually a limit function of $L_{n}(\xi)$; see \eqref{suff2:consistency-uniform-conv} for details. 
Note that $L(\xi)$ is finite on $\Theta_{\xi}(\iota)$ for any $\iota\in(0,1)$ and we can write
\begin{align}\label{expr:lim-ell-bar-diff}
    L(\xi) - L(\xi_{0})
    = \log\left(\int_{-\pi}^{\pi}\frac{s_{\xi_{0}}^{X}(\omega)}{s_{\xi}^{X}(\omega)}\frac{\dd\omega}{2\pi} \right)
    -\int_{-\pi}^{\pi}\log\frac{s_{\xi_{0}}^{X}(\omega)}{s_{\xi}^{X}(\omega)}\frac{\dd\omega}{2\pi},
\end{align}
so that Jensen's inequality of the strictly concave function and the identifiability condition on the family of the spectral density functions $\{s_{\theta}^{X}\}_{\theta\in\Theta}$ in Assumption~\ref{Assump:DSPD1} give 
\begin{align}\label{suff1:consistency-identifiability}
    \inf_{\xi\in\Theta_{\xi}(\iota),\|\xi-\xi_{0}\|_{\R^{p-1}}\geq \varepsilon} L(\xi) > L(\xi_{0}),\ \ \forall \varepsilon>0,~\forall\iota\in(0,1).
\end{align}

To prove consistency of $\{\widehat{\xi}_{n}\}_{n\in\N}$, we firstly prove the uniform convergence
\begin{align}\label{suff2:consistency-uniform-conv}
    \sup_{\xi\in\Theta_{\xi}(\iota)} |L_{n}(\xi)-L(\xi)| = o_{\P^{n}_{\vartheta_{0}}}(1) \ \ \mbox{as $n\to\infty$}.
\end{align}
Note that we can write
\begin{align}\label{expr:suff2-consistency-error}
    L_{n}(\xi)-L(\xi)
    = \left( \log \bar{\sigma}^{2}_{n}(\xi) - \log \sigma^{2}(\xi) \right)
    +\left( \frac{1}{n} \log\mathrm{det}\bigl[\Sigma_{n}(s_{\xi}^{X})\bigr] - \frac{1}{2\pi}\int_{-\pi}^{\pi}\log{s_{\xi}^{X}(\omega)}\dd\omega \right).
\end{align}

Here we recall the uniform convergence version of Szeg\"o's theorem \citep{Lieberman-2012}:
\begin{align}\label{Szego-Thm}
    \sup_{\xi\in\Theta_{\xi}}\left| \frac{1}{n}\log\det\bigl[\Sigma_{n}(s_{\xi}^{X})\bigr] - \frac{1}{2\pi}\int_{-\pi}^{\pi}\log{s_{\xi}^{X}(\omega)}\dd\omega \right| = o(1)\ \ \mbox{as $n\to\infty$}.
\end{align}
Moreover, we can also show the uniform convergence
\begin{align}\label{unifconv-sig2-bar}
    \sup_{\xi\in\Theta_{\xi}(\iota)} \left| \bar{\sigma}^{2}_{n}(\xi)-\sigma^{2}(\xi) \right| = o_{\P_{\vartheta_{0}}^{n}}(1)\ \ \mbox{as $n\to\infty$},
\end{align}
whose proof is left to Section~\ref{Sec:proof-unifconv-sig2-bar} in Online Appendix. Then we also obtain
\begin{align}\label{unifconv-log-sig2-bar}
    \sup_{\xi\in\Theta_{\xi}(\iota)} \left| \log\bar{\sigma}^{2}_{n}(\xi) - \log\sigma^{2}(\xi) \right| = o_{\P^{n}_{\vartheta_{0}}}(1) \ \ \mbox{as $n\to\infty$},
\end{align}
which can be proved using \eqref{unifconv-sig2-bar} immediately. However, we also give a detailed proof in Section~\ref{Sec:proof-unifconv-log-sig2-bar} in Online Appendix for completeness. 
Then we conclude \eqref{suff2:consistency-uniform-conv} using \eqref{expr:lim-ell-bar-diff}, \eqref{expr:suff2-consistency-error}, \eqref{Szego-Thm} and \eqref{unifconv-log-sig2-bar}.\\

Now we give a proof of consistency of $\{\widehat{\xi}_{n}\}_{n\in\N}$ using \eqref{suff1:consistency-identifiability} and \eqref{suff2:consistency-uniform-conv}. For each $\iota\in(0,1)$, we define
\begin{align*}
    \widehat{\xi}_{n}(\iota) = \argmax_{\xi\in\Theta_{\xi}(\iota)} \bar{\ell}_{n}(\xi) = \argmin_{\xi\in\Theta_{\xi}(\iota)} L_{n}(\xi).
\end{align*}
Similarly to \cite{Robinson-1995-LWE}, \cite{Velasco-Robinson-2000} and \cite{Lieberman-2012}, we divide the proof of consistency into the following two steps.\\

\noindent\textbf{Step~1}: 
We prove that for each $\iota\in(0,1)$, $\widehat{\xi}_{n}(\iota)$ is a consistent estimator of $\xi$, {\it i.e.} $\widehat{\xi}_{n}(\iota)\to \xi_{0}$ in $\P_{\vartheta_{0}}^{n}$-probability.

\begin{proof}[Proof of Step~1]
    The conclusion immediately follows from \eqref{suff1:consistency-identifiability}, \eqref{suff2:consistency-uniform-conv} and the definition of $\widehat{\xi}_{n}(\iota)$.
\end{proof}

~\\
\noindent\textbf{Step~2}:
We prove that there exists $\iota\in(0,1)$ such that $\widehat{\xi}_{n}-\widehat{\xi}_{n}(\iota)\to 0$ in $\P_{\vartheta_{0}}^{n}$-probability as $n\to\infty$.\\

\begin{proof}[Proof of Step~2]
	If $\Theta_{\xi}=\Theta_{\xi}(\iota)$ holds for some $\iota\in(0,1)$, the equality $\widehat{\xi}_{n}=\widehat{\xi}_{n}(\iota)$ holds so that we immediately conclude the assertion of Step~2. 
    In the rest of the proof, we assume that $\Theta_{\xi}\setminus\Theta_{\xi}(\iota)$ is a nonempty set for any $\iota\in(0,1)$. 
    Then, for any $\iota\in(0,1)$ and $\epsilon,\epsilon_{1}>0$, we can show
	\begin{align}
	  \P_{\vartheta_{0}}^{n}\left[\|\widehat{\xi}_{n}-\widehat{\xi}_{n}(\iota)\|_{\mathbb{R}^2}>\epsilon\right] 
      &\leq \P_{\vartheta_{0}}^{n}\biggl[ \inf_{\xi\in\Theta_{\xi}\setminus\Theta_{\xi}(\iota)} L_{n}(\xi) 
      \leq \inf_{\xi\in\Theta_{\xi}(\iota)} L_{n}(\xi) \biggr] 
      \nonumber \\
      &\leq \P_{\vartheta_{0}}^{n}\bigl[\inf_{\xi\in\Theta_{\xi}\setminus\Theta_{\xi}(\iota)} L_{n}(\xi) \leq L(\xi_{0})+\epsilon_{1}\bigr] 
      +\P_{\vartheta_{0}}^{n} \bigl[|L_{n}(\widehat{\xi}_{n}(\iota)) - L(\xi_{0})| \geq\epsilon_{1} \bigr].
      \label{suff1:consistency-Step2}
	\end{align}
    Note that for any $\epsilon_{1},\epsilon_{2}>0$, the second term of \eqref{suff1:consistency-Step2} is dominated by
    \begin{align}
        \P_{\vartheta_{0}}^{n} \bigl[|L_{n}(\widehat{\xi}_{n}(\iota)) - L(\xi_{0})| \geq\epsilon_{1} \bigr] 
        \leq& \P_{\vartheta_{0}}^{n} \bigl[|\widehat{\xi}_{n}(\iota) - \xi_{0}| \geq\epsilon_{2} \bigr] 
        + \P_{\vartheta_{0}}^{n} \biggl[\sup_{\xi\in\Theta_{\xi}(\iota)} |L_{n}(\xi) - L(\xi)|\geq\epsilon_{1} \biggr]
        \nonumber \\
        &+ \P_{\vartheta_{0}}^{n} \bigl[|\widehat{\xi}_{n}(\iota) - \xi_{0}| < \epsilon_{2}, |L(\widehat{\xi}_{n}(\iota)) - L(\xi_{0})| \geq\epsilon_{1} \bigr].
        \label{suff2:consistency-Step2}
    \end{align}
    Then the continuity of $L(\xi)$ on $\Theta_{\xi}(\iota)$ shows that for any $\epsilon_{1}>0$, there exists $\epsilon_{2}>0$ such that the third term of \eqref{suff2:consistency-Step2} is equal to zero. 
    Moreover, we can also note that the first and second terms of \eqref{suff2:consistency-Step2} are negligible as $n\to\infty$ using the result in Step~1 and the uniform convergence \eqref{suff2:consistency-uniform-conv}, respectively.

     Finally, we evaluate the first term of \eqref{suff1:consistency-Step2}. 
     Since $\alpha(\xi_{1})\geq\alpha(\xi_{2})$ for any $\xi_{1}\in\Theta_{\xi}(\iota)$ and $\xi_{1}\in\Theta_{\xi}\setminus\Theta_{\xi}(\iota)$, using Lemma~5.3 in \cite{Dahlhaus-1989} and Lemma~6 in the full version of \cite{Lieberman-2012}, we can show that there exists a constant $C_{1}>0$ such that for any $\xi_{1}\in\Theta_{\xi}(\iota)$ and $\xi_{1}\in\Theta_{\xi}\setminus\Theta_{\xi}(\iota)$,
     \begin{align}\label{ineq1:consistency-Step3}
         \frac{\bar{\sigma}^{2}_{n}(\xi_{1})}{\bar{\sigma}^{2}_{n}(\xi_{2})}
         \leq \sup_{\mathbf{x}\in\R^{n}} \frac{ \|\Sigma_{n}(s_{\xi_{1}}^{X})^{-\frac{1}{2}}\mathbf{x}\|_{\R^{n}} }{ \|\Sigma_{n}(s_{\xi_{2}}^{X})^{-\frac{1}{2}}\mathbf{x}\|_{\R^{n}} }
         = \|\Sigma_{n}(s_{\xi_{1}}^{X})^{-\frac{1}{2}}\Sigma_{n}(s_{\xi_{2}}^{X})^{\frac{1}{2}}\|_{\mathrm{op}}
         \leq \frac{1}{C_{1}}.
     \end{align}
     Set $\xi_{1}(\iota):=-1+\alpha(\xi_{0})-\iota$ and
     \begin{align*}
         r_{n,1}(\xi):=\log\bar{\sigma}^{2}_{n}(\xi) - \log\sigma^{2}(\xi),\ \ 
         r_{n,2}(\xi):=\frac{1}{n}\log\mathrm{det}\bigl[\Sigma_{n}(s_{\xi}^{X})\bigr] - \frac{1}{2\pi}\int_{-\pi}^{\pi} \log{s_{\xi}^{X}(\omega)}\dd\omega.
     \end{align*}
     Since $\alpha(\xi_{1}(\iota))\geq\alpha(\xi_{2})$ for any $\xi_{2}\in\Theta_{\xi}\setminus\Theta_{\xi}(\iota)$, the inequality \eqref{ineq1:consistency-Step3} yields
     \begin{align*}
		L_{n}(\xi_{2})
        &\geq (1+\log(2\pi)) + \log{C_{1}} + \log\bar{\sigma}^{2}_{n}(\xi_{1}(\iota)) +  \frac{1}{n}\log\mathrm{det}\bigl[\Sigma_{n}(s_{\xi_{2}}^{X})\bigr] \\
        &\geq (1+\log(2\pi)) + \log{C_{1}} + \log\sigma^{2}(\xi_{1}(\iota)) + \frac{1}{2\pi}\int_{-\pi}^{\pi} \log{s_{\xi_{2}}^{X}(\omega)}\dd\omega 
        + r_{n,1}(\xi_{1}(\iota)) + r_{n,2}(\xi_{2}),
	 \end{align*}
     and
     \begin{align*}
         \log\sigma^{2}(\xi_{1}(\iota)) + \frac{1}{2\pi}\int_{-\pi}^{\pi} \log{s_{\xi_{2}}^{X}(\omega)}\dd\omega 
         &\geq \log\left(\frac{c_{+}}{2\pi c_{-}}\int_{-\pi}^{\pi}|\omega|^{-1+\iota}\dd\omega\right) +  \frac{1}{2\pi}\int_{-\pi}^{\pi} \log\bigl(c_{-}|\omega|^{-\alpha(\xi_{2})}\bigr)\dd\omega \\
         &\geq \log\left(\frac{c_{+}}{\pi c_{-}}\right) + \log\left(\frac{\pi^{\iota}}{\iota}\right) + \log{c_{-}} - \alpha(\xi_{1}(\iota))(\log{\pi}-1).
     \end{align*}
     Therefore, on the set $A_{1}(\delta_{1})\cap A_{2}(\delta_{2})$ with $A_{1}(\delta):=\{\sup_{\xi\in\Theta_{\xi}(\iota)}|r_{n,1}(\xi)|<\delta\}$ and $A_{2}(\delta):=\{\sup_{\xi\in\Theta_{\xi}}|r_{n,2}(\xi)|<\delta\}$, we obtain $\inf_{\xi\in\Theta_{\xi}\setminus\Theta_{\xi}(\iota)} L_{n}(\xi)\geq L(\iota,\delta_{1},\delta_{2})$, where 
     \begin{align*}
         L(\iota,\delta_{1},\delta_{2}):=(1+\log(2\pi)) + \log{C_{1}} + \log\left(\frac{c_{+}}{\pi c_{-}}\right) 
         + \log\left(\frac{\pi^{\iota}}{\iota}\right) + \log{c_{-}} - \alpha(\xi_{1}(\iota))(\log{\pi}-1) - (\delta_{1}+\delta_{2}).
     \end{align*}
     Since $L(\iota,\delta_{1},\delta_{2})$ diverges to infinity as $\iota\to 0$, for any $\epsilon_{1},\delta_{1},\delta_{2}>0$, there exists $\iota\equiv\iota(\epsilon_{1},\delta_{1},\delta_{2})\in(0,1)$ such that $L(\iota,\delta_{1},\delta_{2})>L(\xi_{0})+\epsilon_{1}$. Then we obtain
     \begin{align*}
&\P_{\vartheta_{0}}^{n}\bigl[\inf_{\xi\in\Theta_{\xi}\setminus\Theta_{\xi}(\iota)} L_{n}(\xi) \leq L(\xi_{0})+\epsilon_{1}\bigr] \\
         &\quad\leq \sum_{j=1}^{2}\P_{\vartheta_{0}}^{n}\bigl[A_{j}(\delta_{j})^{\mathrm{c}}\bigr] 
         + \P_{\vartheta_{0}}^{n}\bigl[ A_{1}(\delta_{1})\cap A_{2}(\delta_{2})\cap \bigl\{ L(\iota,\delta_{1},\delta_{2}) \leq L(\xi_{0})+\epsilon_{1} \bigr\} \bigr] \to 0 
     \end{align*}
     as $n\to\infty$ using \eqref{Szego-Thm} and \eqref{unifconv-log-sig2-bar}.
     This completes the proof of Step~2 and the consistency result.
\end{proof}

\subsection{Proof of Asymptotic Normality in Theorem~\ref{Thm:MLE}}\label{Sec:proof-Thm:MLE-asynormal}

Before proving the asymptotic normality of the sequence of the exact MLEs, we summarize notations used in the proof and prepare several limit theorems repeatedly used in the proof. 
Recall that 
\begin{align*}
    \bar{\ell}_{n}(\xi) = -\frac{n}{2}(1+\log(2\pi)) -\frac{n}{2}\log{\bar{\sigma}^{2}_{n}(\xi)} 
    -\frac{1}{2}\log\mathrm{det}\bigl[\Sigma_{n}(s_{\xi}^{X})\bigr],
\end{align*}
see \eqref{expr:ell-bar}. 
Then we can show
\begin{align}
    -\frac{2}{n}\partial_{i}\bar{\ell}_{n}(\xi) 
    =& \frac{1}{\bar{\sigma}^{2}_{n}(\xi)}\partial_{i}\bar{\sigma}^{2}_{n}(\xi)
    +\frac{1}{n}\mathrm{Tr}\left[\Sigma_{n}(\partial_{i}s_{\xi}^{X})\Sigma_{n}(s_{\xi}^{X})^{-1}\right], 
    \label{expr:Ln-dv1} \\
    -\frac{2}{n}\partial_{i,j}^{2}\bar{\ell}_{n}(\xi) 
    =& -\left( -\frac{ \partial_{i}\bar{\sigma}^{2}_{n}(\xi) }{ \bar{\sigma}^{2}_{n}(\xi) } \right) 
    \left( -\frac{ \partial_{j}\bar{\sigma}^{2}_{n}(\xi) }{ \bar{\sigma}^{2}_{n}(\xi) } \right)
    + \frac{1}{\bar{\sigma}^{2}_{n}(\xi)}\partial_{i,j}^{2}\bar{\sigma}^{2}_{n}(\xi) 
    \nonumber \\
    &+\frac{1}{n}\mathrm{Tr}\left[\Sigma_{n}(\partial_{i,j}^{2}s_{\xi}^{X})\Sigma_{n}(s_{\xi}^{X})^{-1}\right]
    -\frac{1}{n}\mathrm{Tr}\left[ \Sigma_{n}(\partial_{i}s_{\xi}^{X}) \Sigma_{n}(s_{\xi}^{X})^{-1}\Sigma_{n}(\partial_{j}s_{\xi}^{X})\Sigma_{n}(s_{\xi}^{X})^{-1} \right].
    \label{expr:Ln-dv2}
\end{align}

Here notice that $-\frac{2}{n}\partial_{i}\bar{\ell}_{n}(\xi)$ in the expression \eqref{expr:Ln-dv1} can be decomposed by
\begin{align*}
    -\frac{2}{n}\partial_{i}\bar{\ell}_{n}(\xi) 
    = \left( \frac{1}{\bar{\sigma}^{2}_{n}(\xi)} - \frac{1}{\sigma_{0}^{2}} \right) \partial_{i}\bar{\sigma}^{2}_{n}(\xi)
    + \frac{1}{\sigma_{0}^{2}}\partial_{i}\bar{\sigma}^{2}_{n}(\xi)
    -\frac{1}{n}\mathrm{Tr}\left[\Sigma_{n}(\partial_{i}s_{\xi}^{X})\Sigma_{n}(s_{\xi}^{X})^{-1}\right] 
\end{align*}
so that we obtain the expression
\begin{align}\label{expr:ell-bar-dv1}
    n^{-\frac{1}{2}}\partial_{\xi}\bar{\ell}_{n}(\xi_{0}) 
    =& - \frac{ \sqrt{n} }{ 2\sigma_{0}^{2} } (\bar{\sigma}^{2}_{n}(\xi_{0})-\sigma_{0}^{2})
    \left( -\frac{ \partial_{\xi}\bar{\sigma}^{2}_{n}(\xi_{0}) }{ \bar{\sigma}^{2}_{n}(\xi_{0}) } \right)
    - \frac{\sqrt{n}}{2\sigma_{0}^{2}}\left( \partial_{\xi}\bar{\sigma}^{2}_{n}(\xi_{0})
    -\E_{\vartheta_{0}}^{n}\left[\partial_{\xi}\tilde{\sigma}^{2}_{n}(\xi_{0})\right] \right)
    \nonumber \\
    =& - \frac{ \sqrt{n} }{ 2\sigma_{0}^{2} } (\bar{\sigma}^{2}_{n}(\xi_{0})-\sigma_{0}^{2}) a_{p-1}(\xi_{0})
    - \frac{\sqrt{n}}{2\sigma_{0}^{2}}\left( \partial_{\xi}\bar{\sigma}^{2}_{n}(\xi_{0})
    -\E_{\vartheta_{0}}^{n}\left[\partial_{\xi}\tilde{\sigma}^{2}_{n}(\xi_{0})\right] \right) 
    + o_{\P_{\vartheta_{0}}^{n}}(1),
\end{align}
where we used \eqref{unifconv-sig2-bar} and Lemma~\ref{Lemma:unifconv-sig2-dv1-dv2} in the last equality. 
Moreover, combining \eqref{expr:ell-bar-dv1} with Lemma~\ref{Lemma:CLT-sig2}, we can show that
\begin{align}\label{CLT-ell-bar}
    n^{-\frac{1}{2}}\partial_{\xi}\bar{\ell}_{n}(\xi_{0}) \to \mathcal{N}(0,V_{p-1}(\xi_{0}))
\end{align}
in law under the distribution $\P_{\vartheta_{0}}^{n}$ as $n\to\infty$, where
\begin{align*}
    V_{p-1}(\xi_{0}) :=& \lim_{n\to\infty} \Var_{\vartheta_{0}}^{n}\left[ 
    - \frac{ \sqrt{n} }{ 2\sigma_{0}^{2} } (\bar{\sigma}^{2}_{n}(\xi_{0})-\sigma_{0}^{2}) a_{p-1}(\xi_{0}) 
    - \frac{\sqrt{n}}{2\sigma_{0}^{2}}\left( \partial_{\xi}\bar{\sigma}^{2}_{n}(\xi_{0})
    - \E_{\vartheta_{0}}^{n}\left[\partial_{\xi}\tilde{\sigma}^{2}_{n}(\xi_{0})\right] \right) \right] \\
    =& \frac{1}{2}a_{p-1}(\xi_{0})a_{p-1}(\xi_{0})^{\top} + \mathcal{F}_{p-1}(\xi_{0}) + 2a_{p-1}(\xi_{0})\left(-\frac{1}{2}a_{p-1}(\xi_{0})\right)^{\top} 
    = \mathcal{G}_{p-1}(\xi_{0}).
\end{align*}
Moreover, using consistency of $\{\widehat{\xi}_{n}\}_{n\in\N}$ as $n\to\infty$, we can show that $ n^{-\frac{1}{2}}\partial_{\xi}\bar{\ell}_{n}(\widehat{\xi}_{n})=o_{\P_{\vartheta_{0}}^{n}}(1)$ as $n\to\infty$ so that, using the Taylor theorem and \eqref{CLT-ell-bar}, we obtain 
\begin{align}\label{Taylor-ell-bar-dv1}
    \int_{0}^{1}\mathcal{G}_{p-1,n}(\xi_{0}+u(\widehat{\xi}_{n}-\xi_{0}))\dd u \sqrt{n}(\widehat{\xi}_{n}-\xi_{0})
    =  \frac{1}{\sqrt{n}}\partial_{\xi}\bar{\ell}_{n}(\xi_{0})
     -\frac{1}{\sqrt{n}}\partial_{\xi}\bar{\ell}_{n}(\widehat{\xi}_{n}) 
    \to \mathcal{N}(0,V_{p-1}(\xi_{0}))
\end{align}
in law under the distribution $\P_{\vartheta_{0}}^{n}$ as $n\to\infty$, where $\mathcal{G}_{p-1,n}(\xi):=-n^{-1}\partial_{\xi}^2\bar{\ell}_{n}(\xi)$. 
Then, combining the uniform convergence in \eqref{unifconv-ell-bar-dv2} with the continuity of the function $\xi\mapsto \mathcal{G}^{i,j}(\xi,\theta_{0})$ at $\xi=\xi_{0}$ and using \eqref{Taylor-ell-bar-dv1} and Slutsky's lemma, we obtain the stochastic expansion
\begin{align}\label{expr:xi-hat-exp1}
    \mathcal{G}_{p-1}(\xi_{0})\sqrt{n}(\widehat{\xi}_{n}-\xi_{0}) = \frac{1}{\sqrt{n}}\partial_{\xi}\bar{\ell}_{n}(\xi_{0}) + o_{\P_{\vartheta_{0}}^{n}}(1) 
    \ \ \mbox{as $n\to\infty$}. 
\end{align}

Moreover, using Taylor's theorem and Lemma~\ref{Lemma:unifconv-sig2-dv1-dv2}, we can also show that
\begin{align}\label{expr:sig-hat-error}
    \sqrt{n}(\bar{\sigma}_{n}(\widehat{\xi}_{n})-\sigma_0)
    = \frac{\sqrt{n}}{2\sigma_0}(\bar{\sigma}_{n}^2(\widehat{\xi}_{n})-\sigma_0^2) + o_{\P_{\vartheta_{0}}^{n}}(1)
    \ \ \mbox{as $n\to\infty$}
\end{align}
and
\begin{align}
    \frac{\sqrt{n}}{2\sigma_{0}}(\bar{\sigma}_{n}^2(\widehat{\xi}_{n})-\sigma_0^2)
    =& \frac{\sqrt{n}}{2\sigma_{0}}(\bar{\sigma}^{2}_{n}(\widehat{\xi}_{n})-\bar{\sigma}^{2}_{n}(\xi_{0})) 
    + \frac{\sqrt{n}}{2\sigma_{0}}(\bar{\sigma}^{2}_{n}(\xi_{0})-\sigma^{2}_{0})
    \nonumber \\
    =& \frac{\sigma^2_0}{2} \left\langle \sqrt{n}(\widehat{\xi}_{n}-\xi_{0}), \frac{1}{\sigma^3_0} \int_{0}^{1}\partial_{\xi}\bar{\sigma}^{2}_{n}(\xi_{0}+v(\widehat{\xi}_{n}-\xi_{0}))\dd v \right\rangle_{\R^{p-1}} 
    + \frac{\sqrt{n}}{2\sigma_{0}}(\bar{\sigma}^{2}_{n}(\xi_{0})-\sigma^{2}_{0}) + o_{\P_{\vartheta_{0}}^{n}}(1)
    \nonumber \\
    =& \frac{\sigma^2_0}{2} \left\langle \sqrt{n}(\widehat{\xi}_{n}-\xi_{0}), -a_{p-1}(\theta_{0}) \right\rangle_{\R^{p-1}} 
    + \frac{\sqrt{n}}{2\sigma_{0}}(\bar{\sigma}^{2}_{n}(\xi_{0})-\sigma^{2}_{0}) + o_{\P_{\vartheta_{0}}^{n}}(1)
    \label{expr:sig2-hat-error}
\end{align}
as $n\to\infty$. Combining \eqref{expr:ell-bar-dv1} with \eqref{expr:xi-hat-exp1}, we get
\begin{align}
    \sqrt{n}(\widehat{\xi}_{n}-\xi_{0}) 
    =& - \frac{\sigma^2_0}{2}\mathcal{G}_{p-1}(\xi_{0})^{-1} a_{p-1}(\theta_{0}) \frac{ \sqrt{n} }{ \sigma^3_0 } (\bar{\sigma}^{2}_{n}(\xi_{0})-\sigma_{0}^{2}) 
    \nonumber \\
    &- \mathcal{G}_{p-1}(\xi_{0})^{-1} \frac{\sqrt{n}}{2\sigma_{0}^{2}}\left( \partial_{\xi}\bar{\sigma}^{2}_{n}(\xi_{0})
    -\E_{\vartheta_{0}}^{n}\left[\partial_{\xi}\tilde{\sigma}^{2}_{n}(\xi_{0})\right] \right) 
    + o_{\P_{\vartheta_{0}}^{n}}(1) \ \ \mbox{as $n\to\infty$}
    \label{expr:xi-hat-error-2}
\end{align}
so that, using \eqref{expr:sig-hat-error} and \eqref{expr:sig2-hat-error}, we obtain
\begin{align}
    \sqrt{n}(\bar{\sigma}_{n}(\widehat{\xi}_{n})-\sigma_0) 
    =& \left\{ \frac{\sigma^4_0}{4} a_{p-1}(\theta_{0})^{\top}\mathcal{G}_{p-1}(\xi_{0})^{-1}a_{p-1}(\theta_{0}) + \frac{\sigma^2_0}{2}
    \right\} \frac{\sqrt{n}}{\sigma^3_0}(\bar{\sigma}^{2}_{n}(\xi_{0})-\sigma_{0}^{2}) + o_{\P_{\vartheta_{0}}^{n}}(1) 
    \nonumber \\
    &+ \frac{\sigma^2_0}{2} \left\langle -\frac{\sqrt{n}}{2\sigma_{0}^{2}}\left( \partial_{\xi}\bar{\sigma}^{2}_{n}(\xi_{0})
    -\E_{\vartheta_{0}}^{n}\left[\partial_{\xi}\tilde{\sigma}^{2}_{n}(\xi_{0})\right] \right), -\mathcal{G}_{p-1}(\xi_{0})^{-1}a_{p-1}(\theta_{0}) \right\rangle_{\R^{p-1}}
    \label{expr:sig-hat-error-2}
\end{align}
as $n\to\infty$. 
Therefore, using \eqref{expr:xi-hat-error-2} and  \eqref{expr:sig-hat-error-2}, the estimation error $\sqrt{n}(\widehat{\theta}_n-\theta_0)=(\sqrt{n}(\widehat{\xi}_{n}-\xi_{0}),\sqrt{n}(\widehat{\sigma}_{n}-\sigma_{0}))^\top$ is expressed by
\begin{align}
    \sqrt{n}(\widehat{\theta}_n-\theta_0) 
    =& 
    \begin{pmatrix}
        \mathcal{G}_{p-1}(\xi_{0})^{-1} & -\frac{\sigma^2_0}{2}a_{p-1}(\theta_{0})^{\top}\mathcal{G}_{p-1}(\xi_{0})^{-1} \\
        -\frac{\sigma^2_0}{2}a_{p-1}(\theta_{0})^{\top}\mathcal{G}_{p-1}(\xi_{0})^{-1} & \frac{\sigma^4_0}{4} a_{p-1}(\theta_{0})^{\top}\mathcal{G}_{p-1}(\xi_{0})^{-1}a_{p-1}(\theta_{0}) + \frac{\sigma^2_0}{2}
    \end{pmatrix} \bar{\zeta}_n + o_{\P_{\vartheta_{0}}^{n}}(1) 
    \nonumber \\
    =& 
    \mathcal{F}_{p}(\theta_0)^{-1} \bar{\zeta}_n + o_{\P_{\vartheta_{0}}^{n}}(1),
    \label{expr:theta-hat-error}
\end{align}
where the last equality can be proved in the similar way to the proof of Lemma~4 in \cite{Fukasawa-Takabatake-2019}. 
Combining the expression \eqref{expr:theta-hat-error} with Lemma~\ref{Lemma:CLT-sig2}, we complete the proof of Theorem~\ref{Thm:MLE}.

%%%%%%%%%%%%%%%%%%%%%%%%%%%%%%%%%%%%%%%%%%%%%%%%%%

%%%%%%%%%%%%%%%%%%%%%%%%%%%%%%
\subsection{Proof of Theorem~\ref{Thm:MLE2}}\label{Sec:proof-Thm:MLE2}

Before proving Theorem~\ref{Thm:MLE2}, we prove the asymptotic normality of the MLE of $\mu$. The following Proposition is needed and its proof can be found in Online Supplement.
\begin{proposition}\label{Thm:MLE-mu}
    Assume that Assumption~$\ref{Assump:DSPD2}$ holds and a sequence of estimators $\{\widehat{\xi}_{n}\}_{n\in\N}$ satisfying that the sequence of rescaled estimation errors $\{\sqrt{n}(\widehat{\xi}_{n}-\xi_{0})\}_{n\in\N}$ is stochastically bounded under the sequence of distributions $\{\P_{\vartheta_{0}}^{n}\}_{n\in\N}$. Then we can show that
    \begin{align*}
        n^{\frac{1}{2}(1-\alpha_{X}(\xi_{0}))} ( \mu_{n}(\widehat{\xi}_{n}) - \mu_{0}) \to \mathcal{N}\left( 0,\frac{2\pi\sigma^2_0c_{X}(\xi_{0})\Gamma(1-\alpha_{X}(\xi_{0}))}{B(1-\alpha_{X}(\xi_{0})/2,1-\alpha_{X}(\xi_{0})/2)} \right)
    \end{align*}
    in law under the distribution $\P_{\vartheta_{0}}^{n}$ as $n\to\infty$ for any interior point $\vartheta_{0}$ of $\Theta$.
\end{proposition}

Recall that $\Phi_{n}(\vartheta)=\diag(n^{-\frac{1}{2}}I_{p},n^{-\frac{1}{2}(1-\alpha_X(\xi_0))})$ and
\begin{align*}
    &\sigma^{2}_{n}(\xi,\mu) = \frac{1}{n}\left(\mathbf{X}_{n}-\mu\mathbf{1}_{n}\right)^{\top}\Sigma_{n}(s_{\xi}^{X})^{-1}\left(\mathbf{X}_{n}-\mu\mathbf{1}_{n}\right),\ \ \tilde{\sigma}^{2}_{n}(\xi)=\sigma^{2}_{n}(\xi,\mu_{0}),\ \
    \mu_{n}(\xi) = \frac{ \mathbf{1}_{n}^{\top}\Sigma_{n}(s_{\xi}^{X})^{-1}\mathbf{X}_{n} }{ \mathbf{1}_{n}^{\top}\Sigma_{n}(s_{\xi}^{X})^{-1}\mathbf{1}_{n} },\\
    &\ell_{n}(\vartheta)
    = -\frac{n}{2}\log(2\pi) -\frac{n}{2}\log{\sigma^{2}}
    -\frac{1}{2}\log\mathrm{det}\bigl[\Sigma_{n}(s_{\xi}^{X})\bigr]
    -\frac{n}{2\sigma^{2}}\sigma^{2}_{n}(\xi,\mu),
\end{align*}
and then, using the formulas of derivatives of the log-determinant and the inverse matrix, {\it e.g.}~see \cite{Harville-1998}, the first order derivatives of the Gaussian log-likelihood function $\ell_{n}(\vartheta)$ with respect to parameters can be written as
\begin{align}\label{expr:score}
	\begin{cases}
    &\partial_{\xi}\ell_{n}(\vartheta)= -\frac{1}{2}\mathrm{Tr}\left[\Sigma_{n}(s_{\xi})^{-1}\Sigma_{n}(\partial_{\xi}s_{\xi})\right] 
    -\frac{n}{2\sigma^{2}}\partial_{\xi}\sigma^{2}_{n}(\xi,\mu), \\
    &\partial_{\sigma}\ell_{n}(\vartheta)= -\frac{n}{\sigma} +\frac{n}{\sigma^{3}} \sigma^{2}_{n}(\xi,\mu),\ \ 
    \partial_{\mu}\ell_{n}(\vartheta)= \frac{1}{\sigma^{2}}\mathbf{1}_{n}^{\top}\Sigma_{n}(s_{\xi}^{X})^{-1}\left(\mathbf{X}_{n}-\mu\mathbf{1}_{n}\right).
    \end{cases}
\end{align}
Then we can write
\begin{align*}
    &\partial_{\xi}\ell_{n}(\vartheta_{0}) = -\frac{n}{2\sigma^{2}_{0}}\left( \partial_{\xi}\tilde{\sigma}^{2}_{n}(\xi_{0}) - \E_{\vartheta_{0}}^{n}[\partial_{\xi}\tilde{\sigma}^{2}_{n}(\xi_{0})] \right),\\ 
    &\partial_{\sigma}\ell_{n}(\vartheta_{0}) = \frac{n}{\sigma^{3}_{0}}\left( \tilde{\sigma}^{2}_{n}(\xi_{0}) - \sigma^{2}_{0} \right),\ \ 
    \partial_{\mu}\ell_{n}(\vartheta_{0}) = \frac{1}{\sigma^{2}_{0}}(\mathbf{1}_{n}^{\top}\Sigma_{n}(s_{\xi_{0}}^{X})^{-1}\mathbf{1}_{n}) (\mu_{n}(\xi_{0})-\mu_{0})
\end{align*}
so that the normalized score function $\zeta_{n}(\vartheta_{0})=\Phi_{n}(\vartheta_{0})^{\top}\partial_{\vartheta}\ell_{n}(\vartheta_{0})$ is expressed by
\begin{align}\label{expr:zeta}
    \zeta_{n}(\vartheta_{0}) = 
    \diag\left( -\frac{\sqrt{n}}{2\sigma_0^{2}}I_{p-1},\frac{\sqrt{n}}{\sigma_0^3},\frac{(\mathbf{1}_{n}^{\top}\Sigma_{n}(s_{\xi_{0}}^{X})^{-1}\mathbf{1}_{n})}{\sigma_0^2n^{\frac{1}{2}(1-\alpha_X(\xi_0))}} \right) 
    \begin{pmatrix}
        \partial_{\xi}\tilde{\sigma}^{2}_{n}(\xi_{0}) - \E_{\vartheta_{0}}^{n}[\partial_{\xi}\tilde{\sigma}^{2}_{n}(\xi_{0})] \\
        \tilde{\sigma}^{2}_{n}(\xi_{0}) - \sigma^{2}_{0} \\
        \mu_{n}(\xi_{0}) - \mu_{0}
    \end{pmatrix}.
\end{align}
Then we can prove the following central limit theorem
\begin{align}\label{CLT-zeta-n}
    \mathcal{L}(\zeta_{n}(\vartheta)|\P_{\vartheta}^{n}) \overset{n\to\infty}{\rightarrow} \mathcal{N}(0,\mathcal{I}(\theta)), \mbox{ as } n\rightarrow \infty,
\end{align}
whose proof is left to Section~\ref{Sec:proof-CLT-score} in Online Supplement. 
Moreover, we can also show that 
\begin{align}\label{expr:app-sig2-tilde}
        \sqrt{n}
        \begin{pmatrix}
            \partial_{\xi}\bar{\sigma}^{2}_{n}(\xi_{0}) - \E_{\vartheta_{0}}^{n}[\partial_{\xi}\tilde{\sigma}^{2}_{n}(\xi_{0})] \\
            \bar{\sigma}^{2}_{n}(\xi_{0}) - \sigma^{2}_{0}
        \end{pmatrix}
        = \sqrt{n}
        \begin{pmatrix}
            \partial_{\xi}\tilde{\sigma}^{2}_{n}(\xi_{0}) - \E_{\vartheta_{0}}^{n}[\partial_{\xi}\tilde{\sigma}^{2}_{n}(\xi_{0})] \\
            \tilde{\sigma}^{2}_{n}(\xi_{0}) - \E_{\vartheta_{0}}^{n}[\tilde{\sigma}^{2}_{n}(\xi_{0})]
        \end{pmatrix}
        + o_{\P_{\vartheta_{0}}^{n}}(1)\ \ \mbox{as $n\to\infty$}.
\end{align}
whose proof is left to Section~\ref{Sec:proof-app-sig2-tilde}  in Online Supplement. Combining the expression of the estimation error $\sqrt{n}(\widehat{\theta}_n-\theta_0)$ in \eqref{expr:theta-hat-error} with that of $\zeta_{n}(\vartheta_{0})$ in \eqref{expr:zeta} and using the equalities in \eqref{Thm:MLE-mu_Suff1} and \eqref{expr:app-sig2-tilde}, we conclude 
\begin{align*}
    \Phi_{n}(\vartheta_{0})^{-1}(\widehat{\vartheta}_{n}-\vartheta_{0}) = \zeta_{n}(\vartheta_{0}) + o_{\P_{\vartheta_{0}}^{n}}(1)\ \ \mbox{as $n\to\infty$}
\end{align*}
so that we complete the proof of Theorem~\ref{Thm:MLE2}.

\subsection{Proof of Theorem~\ref{Thm:LAN}}\label{Sec:proof-Thm:LAN}

Recall that $\Phi_{n}(\vartheta)=\diag(n^{-\frac{1}{2}}I_{p},n^{-\frac{1}{2}(1-\alpha_X(\xi))})$. 
We first give an outline of the proof of Theorem~\ref{Thm:LAN}. 
Using the Taylor theorem, the log-likelihood ratio is written as
\begin{align}
	\log\frac{\mathrm{d}\P^{n}_{\vartheta+\Phi_{n}(\vartheta)u}}{\mathrm{d}\P^{n}_{\vartheta}}(\mathbf{X}_{n})
	=&u^{\top}\zeta_{n}(\vartheta)
    -\frac{1}{2}\int_{0}^{1}(1-z) \partial_{\vartheta}^{2}\ell_{n}(\vartheta+z\Phi_{n}(\vartheta)u) \left[\left(\Phi_{n}(\vartheta)u\right)^{\otimes 2}\right] \,\mathrm{d}z 
    \nonumber \\
    =& u^{\top}\zeta_{n}(\vartheta)
    -\frac{1}{2}\int_{0}^{1}(1-z) \mathcal{I}_{n}(\vartheta+z\Phi_{n}(\vartheta)u) \left[\left(R_{n}(\vartheta,z\Phi_{n}(\vartheta)u)u\right)^{\otimes 2}\right] \,\mathrm{d}z,
    \label{Likelihood-Ratio-Taylor}
\end{align}
where $R_{n}(\vartheta,v):=\Phi_{n}\left(\vartheta+v\right)^{-1}\Phi_{n}(\vartheta)$ for $v\in\R^{p+1}$. 
Since we have already proved the CLT of the stochastic sequence $\{\zeta_{n}(\vartheta)\}_{n=1}^{\infty}$ in \eqref{CLT-zeta-n} and we have the inequality
\begin{align*}
    &\left| \int_{0}^{1}(1-z) \mathcal{I}_{n}(\vartheta+z\Phi_{n}(\vartheta)u) \left[\left(R_{n}(\vartheta,z\Phi_{n}(\vartheta)u)u\right)^{\otimes 2}\right] \,\mathrm{d}z - u^{\top}\mathcal{I}(\vartheta)u \right| \\
    &\leq \| R_{n}(\vartheta,z\Phi_{n}(\vartheta)u)^{\top}\mathcal{I}_{n}(\vartheta+z\Phi_{n}(\vartheta)u)R_{n}(\vartheta,z\Phi_{n}(\vartheta)u)  - \mathcal{I}(\vartheta) \|_{1} \|u\|_{\R^{p+1}}^{2},
\end{align*}
where $\|A\|_{1}:=\sum_{i,j=1}^{p}|a_{ij}|$ for a $p\times p$-matrix $A=(a_{ij})_{i,j=1,\cdots,p}$, we can conclude Theorem~\ref{Thm:LAN} using the triangle inequality and the multiplicativity of the matrix norm $\|\cdot\|_{1}$ and the uniform continuity of $\mathcal{I}(\vartheta)$ on compact subsets of $\Theta$ once we have proved the convergences
\begin{align}
    & \sup_{v\in\R^{p+1}:\|v\|_{\R^{p+1}}\leq c\|\Phi_{n}(\vartheta)\|_{\mathrm{op}}}\|R_{n}(\vartheta,v)-I_{p+1}\|_{1} = o_{\P_{\vartheta}^{n}}(1)\ \ \mbox{as $n\to\infty$} 
    \label{key:LAN:ratio-rate-mat},\\
    &\|\mathcal{I}_{n}(\vartheta)-\mathcal{I}(\vartheta)\|_{1} = o_{\P_{\vartheta}^{n}}(1)\ \ \mbox{as $n\to\infty$},
    \label{key:LAN:conv-Fisher} \\
    &\sup_{u\in\mathbb{U}_{n,c}(\vartheta)} 
    \|\mathcal{I}_n(\vartheta+\Phi_n(\vartheta)u)-\mathcal{I}_n(\vartheta)\|_{1} = o_{\P_{\vartheta}^{n}}(1)\ \ \mbox{as $n\to\infty$},
    \label{key:LAN:reminder}
\end{align}
for any $c>0$, where $\mathbb{U}_{n}(\vartheta):=\Phi_{n}^{-1}(\Theta)(\Theta-\vartheta)=\{u\in\R^{p+1}:\vartheta+\Phi_n(\vartheta)u\in\Theta\}$ and $\mathbb{U}_{n,c}(\vartheta):=\{u\in\mathbb{U}_{n}(\vartheta):\|u\|_{\R^{p+1}}\leq c\}$ for $c>0$. 
In the rest of the Appendix, we try to prove the above three results.

\subsection{Proof of \eqref{CLT-zeta-n}}\label{Sec:proof-CLT-score}
 Set $\mathbf{Z}_{n}:=\Sigma_{n}(s_{\theta_{0}}^{X})^{-\frac{1}{2}}(\mathbf{X}_{n}-\mu_{0}\mathbf{1}_{n})\sim\mathcal{N}(0,I_{n})$. For $\mathbf{u}_{p}=(u_{1},\ldots,u_{p})^{\top}\in\R^{p}$ and $u_{p+1}\in\R$, we write
    \begin{align*}
        J_{n}(u_{1},\ldots,u_{p+1}) := 
        \begin{pmatrix}
        \mathbf{u}_{p}^{\top} & u_{p+1}
        \end{pmatrix}
     \diag\left(-\frac{\sqrt{n}}{2\sigma^{2}_{0}}I_{p-1},\frac{\sqrt{n}}{\sigma^3_0},%\frac{n^{\frac{1}{2}(1-\alpha_{X}(\xi_{0}))}}{\sigma_{0}}
     \frac{(\mathbf{1}_{n}^{\top}\Sigma_{n}(s_{\theta_{0}}^{X})^{-1}\mathbf{1}_{n})}{\sigma_0^2n^{\frac{1}{2}(1-\alpha_X(\xi_0))}}\right)
    \begin{pmatrix}
        \partial_{\xi}\tilde{\sigma}^{2}_{n}(\xi_{0}) - \E_{\vartheta_{0}}^{n}[\partial_{\xi}\tilde{\sigma}^{2}_{n}(\xi_{0})] \\
        \tilde{\sigma}^{2}_{n}(\xi_{0}) - \sigma^{2}_{0} \\
        \mu_{n}(\xi_{0}) - \mu_{0}
    \end{pmatrix}.
    \end{align*}
    Then $J_{n}(u_{1},\ldots,u_{p+1})$ is rewritten as
\begin{align*}
    J_{n}(u_{1},\ldots,u_{p+1})&= \frac{1}{2\sqrt{n}}\sum_{j=1}^{p-1} u_{j}\mathbf{Z}_{n}^{\top} \Sigma_{n}(s_{\theta_{0}}^{X})^{-\frac{1}{2}}\Sigma_{n}(\partial_{j}s_{\theta_{0}}^{X})\Sigma_{n}(s_{\theta_{0}}^{X})^{-\frac{1}{2}} \mathbf{Z}_{n} \\
    &\hspace{2cm }+ \frac{1}{\sigma_0\sqrt{n}} u_{p}\mathbf{Z}_{n}^{\top}\mathbf{Z}_{n} 
    + \frac{(\mathbf{1}_{n}^{\top}\Sigma_{n}(s_{\theta_{0}}^{X})^{-1}\mathbf{1}_{n})}{n^{\frac{1}{2}(1-\alpha_X(\xi_0))}} u_{p+1} \frac{ (\Sigma_{n}(s_{\theta_{0}}^{X})^{-\frac{1}{2}}\mathbf{1}_{n})^{\top}\mathbf{Z}_{n} }{ \mathbf{1}_{n}^{\top}\Sigma_{n}(s_{\theta_{0}}^{X})^{-1}\mathbf{1}_{n} } \\
    &= \frac{1}{\sqrt{n}}\mathbf{Z}_{n}^{\top} \Sigma_{n}(s_{\theta_{0}}^{X})^{-\frac{1}{2}}\Sigma_{n}(g_{\theta_{0}}^{u_{1},\ldots,u_{p}})\Sigma_{n}(s_{\theta_{0}}^{X})^{-\frac{1}{2}} \mathbf{Z}_{n} 
    + n^{-\frac{1}{2}(1-\alpha_X(\xi_0))} u_{p+1} (\Sigma_{n}(s_{\theta_{0}}^{X})^{-\frac{1}{2}}\mathbf{1}_{n})^{\top}\mathbf{Z}_{n},
\end{align*}
where
\begin{align*}
    g_{\theta_{0}}^{u_{1},\ldots,u_{p}}(\omega) := 
    \frac{1}{2}\sum_{j=1}^{p-1} u_{j}\partial_{j}s_{\theta_{0}}^{X}(\omega) 
    + \frac{1}{\sigma_0} u_{p}s_{\theta_{0}}^{X}(\omega) 
\end{align*}
Since the matrix $\Sigma_{n}(s_{\theta_{0}}^{X})^{-\frac{1}{2}}\Sigma_{n}(g_{\theta_{0}}^{u_{1},\ldots,u_{p}})\Sigma_{n}(s_{\theta_{0}}^{X})^{-\frac{1}{2}}$ is symmetric, there exists a $n$th square matrix $V_{n}$ such that $V_{n}$ is a orthogonal matrix and
\begin{align*}
    V_{n}\Sigma_{n}(s_{\theta_{0}}^{X})^{-\frac{1}{2}}\Sigma_{n}(g_{\theta_{0}}^{u_{1},\ldots,u_{p}})\Sigma_{n}(s_{\theta_{0}}^{X})^{-\frac{1}{2}}V_{n}^{\top} = \diag(\lambda_{1,n},\ldots,\lambda_{n,n}),
\end{align*}
where $\{\lambda_{j,n}\}_{j=1,\ldots,n}$ is eigen values of the matrix $\Sigma_{n}(s_{\theta_{0}}^{X})^{-\frac{1}{2}}\Sigma_{n}(g_{\theta_{0}}^{u_{1},\ldots,u_{p}})\Sigma_{n}(s_{\theta_{0}}^{X})^{-\frac{1}{2}}$. 
Then we set a $n$-dimensional random vector $\mathbf{W}_{n}=(W_{1,n},\ldots,W_{n,n})^{\top}$ and a $n$-dimensional (non-random) vector $\mathbf{A}_{n}=(A_{1,n},\ldots,A_{n,n})^{\top}$ by
\begin{align*}
    \mathbf{W}_{n}:=V_{n}\mathbf{Z}_{n} 
    \ \ \mbox{and}\ \ 
    \mathbf{A}_{n} := u_{p+1} V_{n}\Sigma_{n}(s_{\theta_{0}}^{X})^{-\frac{1}{2}}\mathbf{1}_{n}.
\end{align*}
Notice that $\mathbf{W}_{n}$ is a $n$-dimensional standard Gaussian vector and we can write 
\begin{align*}
    J_{n}(u_{1},\ldots,u_{p+1}) = \sum_{j=1}^{n}\left\{ n^{-\frac{1}{2}}\lambda_{j,n}(W_{j,n}^{2}-1) 
    + n^{-\frac{1}{2}(1-\alpha_X(\xi_0))} A_{j,n}W_{j,n} \right\}.
\end{align*}
Set $U_{j,n}:=n^{-\frac{1}{2}}\lambda_{j,n}(W_{j,n}^{2}-1) + \sigma^{-2}_0 n^{-\frac{1}{2}(1-\alpha_X(\xi_0))} A_{j,n}W_{j,n}$. Notice that $\{U_{j,n}\}_{j=1,\ldots,n}$ is an independent triangle array  with mean zero and variance
\begin{align*}
    \Var[U_{j,n}] =& n^{-1}\lambda_{j,n}^{2}\E[(W_{j,n}^{2}-1)^{2}] 
    + n^{-(1-\alpha_X(\xi_0))} A_{j,n}^{2}\E[W_{j,n}^{2}] \\
    &+ 2n^{-\frac{1}{2}} n^{-\frac{1}{2}(1-\alpha_X(\xi_0))} \lambda_{j,n}A_{j,n}\E[(W_{j,n}^{2}-1)W_{j,n}] \\
    =& 2n^{-1}\lambda_{j,n}^{2} + n^{-(1-\alpha_X(\xi_0))} A_{j,n}^{2},
\end{align*}
where we used the facts that 
\begin{align*}
    \E[(W_{j,n}^{2}-1)^{2}]=\E[W_{j,n}^{4}] -2\E[W_{j,n}^{2}] + 1=2,\ \ \E[(W_{j,n}^{2}-1)W_{j,n}]=0.
\end{align*}
Then we can also show that
\begin{align*}
    \Var\left[\sum_{j=1}^{n}U_{j,n}\right] 
    =& \sum_{j=1}^{n}\Var[U_{j,n}]
    = \frac{2}{n}\sum_{j=1}^{n}\lambda_{j,n}^{2} + n^{-(1-\alpha_X(\xi_0))} \sum_{j=1}^{n}A_{j,n}^{2} \\
    =& \frac{2}{n}\mathrm{Tr}[\diag(\lambda_{1,n},\ldots,\lambda_{n,n})^{2}] 
    + n^{-(1-\alpha_X(\xi_0))} \|\mathbf{A}_{n}\|_{\R^{n}}^{2} \\
    =& \frac{2}{n}\mathrm{Tr}[(\Sigma_{n}(g_{\theta_{0}}^{u_{1},\ldots,u_{p}})\Sigma_{n}(s_{\theta_{0}}^{X})^{-1})^{2}] 
    + u_{p+1}^{2} n^{-(1-\alpha_X(\xi_0))} (\mathbf{1}_{n}^{\top}\Sigma_{n}(s_{\theta_{0}}^{X})^{-1}\mathbf{1}_{n}) \\
    =& \mathbf{u}_{p}^{\top}
    \begin{pmatrix}
        \widetilde{\mathcal{F}}_{p-1,n}(\theta_{0}) & \frac{1}{\sigma_0n}\mathrm{Tr}[\Sigma_{n}(\partial_{\xi}s_{\theta_{0}}^{X})\Sigma_{n}(s_{\theta_{0}}^{X})^{-1}] \\
        \mathrm{sym}. & 2\sigma_0^{-2}
    \end{pmatrix}
    \mathbf{u}_{p} + u_{p+1}^{2} n^{-(1-\alpha_X(\xi_0))} (\mathbf{1}_{n}^{\top}\Sigma_{n}(s_{\theta_{0}}^{X})^{-1}\mathbf{1}_{n}),
\end{align*}
where 
\begin{align*}\small
    \widetilde{\mathcal{F}}_{p-1,n}(\theta_{0})
    := \frac{1}{2n}
    \begin{pmatrix}
       \mathrm{Tr}[(\Sigma_{n}(\partial_{1}s_{\theta_{0}}^{X})\Sigma_{n}(s_{\theta_{0}}^{X})^{-1})^{2}] 
       &  \cdots & \mathrm{Tr}[\Sigma_{n}(\partial_{1}s_{\theta_{0}}^{X})\Sigma_{n}(s_{\theta_{0}}^{X})^{-1}\Sigma_{n}(\partial_{p-1}s_{\theta_{0}}^{X})\Sigma_{n}(s_{\theta_{0}}^{X})^{-1}]   \\
       \vdots & \ddots & \vdots \\
       \mathrm{Tr}[\Sigma_{n}(\partial_{p-1}s_{\theta_{0}}^{X})\Sigma_{n}(s_{\theta_{0}}^{X})^{-1}\Sigma_{n}(\partial_{1}s_{\theta_{0}}^{X})\Sigma_{n}(s_{\theta_{0}}^{X})^{-1}] 
       &  \cdots & \mathrm{Tr}[(\Sigma_{n}(\partial_{p-1}s_{\theta_{0}}^{X})\Sigma_{n}(s_{\theta_{0}}^{X})^{-1})^{2}] 
    \end{pmatrix},
\end{align*}
so that Lemma~\ref{Lemma:trace-app-inv} and Theorems~4.1 and 5.2 in \cite{Adenstedt-1974} yield
\begin{align*}
    \lim_{n\to\infty}\Var\left[\sum_{j=1}^{n}U_{j,n}\right] 
    = \mathbf{u}_{p}^{\top}\mathcal{F}_{p}(\theta_{0})\mathbf{u}_{p} 
    + u_{p+1}^{2} \frac{2\pi\sigma^2_0c_{X}(\xi_{0})\Gamma(1-\alpha_{X}(\xi_{0}))}{B(1-\alpha_{X}(\xi_{0})/2,1-\alpha_{X}(\xi_{0})/2)} 
    = \mathbf{u}_{p+1}^{\top}\mathcal{I}(\theta_{0})\mathbf{u}_{p+1}.
\end{align*}
Moreover, we can show that
\begin{align*}
    \E[U_{j,n}^{4}] 
    \leq 4n^{-2}\E[|\lambda_{j,n}(W_{j,n}^{2}-1)|^{4}] \ 
    + 4 \sigma^{-4}_0 n^{2(1-\alpha_X(\xi_0))} (\mathbf{1}_{n}^{\top}\Sigma_{n}(s_{\xi_{0}}^{X})^{-1}\mathbf{1}_{n})^{-4} \E[|A_{j,n}W_{j,n}|^{4}].
\end{align*}
Here notice that, since $\{W_{j,n}\}_{j=1}^{n}$ is an independent centered sequence for each $n\in\N$, we can show 
\begin{align*}
    \E\biggl[\biggl|\sum_{j=1}^{n}A_{j,n}W_{j,n}\biggr|^{4}\biggr]
    = \sum_{j_{1},j_{2},j_{3},j_{4}=1}^{n} \prod_{i=1}^{4}A_{j_{i},n} \E\biggl[\prod_{k=1}^{4}W_{j_{k},n}\biggr]
    = \sum_{j=1}^{n}A_{j,n}^{4}\E[W_{j,n}^{4}] +3\left(\sum_{j=1}^{n}A_{j,n}^{2}\E[W_{j,n}^{2}]\right)^{2},
\end{align*}
which implies
\begin{align*}
    \sum_{j=1}^{n}\E[|A_{j,n}W_{j,n}|^{4}]
    =& \E\biggl[\biggl|\sum_{j=1}^{n}A_{j,n}W_{j,n}\biggr|^{4}\biggr] -3\left(\sum_{j=1}^{n}|A_{j,n}|^{2}\right)^{2} \\
    =& u_{p+1}^{4}\left( \E[|(\Sigma_{n}(s_{\theta_{0}}^{X})^{-\frac{1}{2}}\mathbf{1}_{n})^{\top}\mathbf{Z}_{n}|^{4}] - 3\|\mathbf{A}_{n}\|_{\R^{n}}^{2} \right) \\
    =& u_{p+1}^{4}\left( 3(\mathbf{1}_{n}^{\top}\Sigma_{n}(s_{\theta_{0}}^{X})^{-1}\mathbf{1}_{n})^{2} - 3(\mathbf{1}_{n}^{\top}\Sigma_{n}(s_{\theta_{0}}^{X})^{-1}\mathbf{1}_{n})^{2} \right) = 0, 
\end{align*}
where we used the fact that $(\Sigma_{n}(s_{\theta_{0}}^{X})^{-\frac{1}{2}}\mathbf{1}_{n})^{\top}\mathbf{Z}_{n}\sim\mathcal{N}(0,\mathbf{1}_{n}^{\top}\Sigma_{n}(s_{\theta_{0}}^{X})^{-1}\mathbf{1}_{n})$. 
Moreover, we can show that
\begin{align*}
    \E[(W_{j,n}^{2}-1)^{4}] 
   = \E[W_{j,n}^{8}-4W_{j,n}^{6} + 6W_{j,n}^{4} -4W_{j,n}^{2} +1] 
   = 7\cdot 5\cdot 3 - 4\cdot 5\cdot 3 + 6\cdot 3 - 4 + 1 = 60
\end{align*}
so that we obtain
\begin{align*}
    \sum_{j=1}^{n}\E[U_{j,n}^{4}]
    \leq \frac{240}{n^{2}}\sum_{j=1}^{n}\lambda_{j,n}^{4}
    = \frac{240}{n^{2}}\mathrm{Tr}[\diag(\lambda_{1,n},\ldots,\lambda_{n,n})^{4}]
    = \frac{240}{n^{2}}\mathrm{Tr}[(\Sigma_{n}(g_{\theta_{0}}^{u_{1},\ldots,u_{p}})\Sigma_{n}(s_{\xi_{0}}^{X})^{-1})^{4}] 
    \to 0 
\end{align*}
using Lemma~\ref{Lemma:trace-app-inv}. Therefore, we succeeded to verify Lindeberg's condition so that we conclude the result. 

%%%%%%%%%%%%%%%%%%%%%%%%%%%%%%
\subsection{Proof of \eqref{expr:app-sig2-tilde}}\label{Sec:proof-app-sig2-tilde}
Recall that  
    \begin{align*}
        &e_{n,1}(\xi):=\bar{\sigma}^{2}_{n}(\xi)-\tilde{\sigma}^{2}_{n}(\xi) 
        = -n^{-1}(\mu_{n}(\xi)-\mu_{0})^{2} \mathbf{1}_{n}^{\top}\Sigma_{n}(s_{\xi}^{X})^{-1}\mathbf{1}_{n}, \\
        &\partial_{\xi}e_{n,1}(\xi)
        = -2n^{-1}(\mu_{n}(\xi)-\mu_{0}) \partial_{\xi}\mu_{n}(\xi) \mathbf{1}_{n}^{\top}\Sigma_{n}(s_{\xi}^{X})^{-1}\mathbf{1}_{n} 
        -n^{-1}(\mu_{n}(\xi)-\mu_{0})^{2} \mathbf{1}_{n}^{\top}\partial_{\xi}\Sigma_{n}(s_{\xi}^{X})^{-1}\mathbf{1}_{n},
    \end{align*}
    see \eqref{expr:en1} and \eqref{expr:en1-dv1}. 
    Notice that, using Lemmas~\ref{Lemma:Key-Ineq-mu} and \ref{Lemma:moment-dv-mu-xi} and the Cauchy-Schwarz inequality, we can show that 
    \begin{align*}
        \sqrt{n}e_{n,1}(\xi_{0})=o_{\P_{\vartheta_{0}}^{n}}(1)
        \ \ \mbox{and}\ \ 
        \sqrt{n}\partial_{\xi}e_{n,1}(\xi_{0})=o_{\P_{\vartheta_{0}}^{n}}(1)
        \ \ \mbox{as $n\to\infty$}.
    \end{align*}
    Moreover, $\E_{\vartheta_{0}}^{n}[\tilde{\sigma}^{2}_{n}(\xi_{0})]=n^{-1}\mathrm{Tr}[\Sigma_{n}(s_{\vartheta_{0}}^{X})\Sigma_{n}(s_{\xi_{0}}^{X})^{-1}]=\sigma^{2}_{0}$ holds so that we complete the proof of \eqref{expr:app-sig2-tilde}.

\newpage
\section{Online Supplement to ``Optimal Estimation for General Gaussian Processes'' by Tetsuya Takabatake, Jun Yu, Chen Zhang (Not for Publication)}
\subsection{Some Useful Lemmas}
In this subsection, we summarize preliminary results used in the proof of theorems given in Section~\ref{Sec:main-results}. 
The following two lemmas are useful to prove Theorems~\ref{Thm:MLE} and \ref{Thm:MLE2} and frequently used in the proof of Theorems~\ref{Thm:MLE} and \ref{Thm:MLE2}. 

\begin{lemma}\label{Lemma:Key-Ineq-mu}
    Under Assumption~$\ref{Assump:DSPD1}$, we can show that for any $i,j,k\in\{1,2,\cdots,p-1\}$,  $\varepsilon>0$ and $\theta\in\Theta$, 
    \begin{align*}
        &\left|\mathbf{1}_{n}^{\top}\partial_{i}\Sigma_{n}(s_{\theta}^{X})^{-1}\mathbf{1}_{n}\right|
        \leq C (\mathbf{1}_{n}^{\top}\Sigma_{n}(s_{\theta}^{X})^{-1}\mathbf{1}_{n}) n^{\varepsilon} \\
        &\left|\mathbf{1}_{n}^{\top}\partial_{i,j}^{2}\Sigma_{n}(s_{\theta}^{X})^{-1}\mathbf{1}_{n}\right|
        \leq C (\mathbf{1}_{n}^{\top}\Sigma_{n}(s_{\theta}^{X})^{-1}\mathbf{1}_{n}) n^{\varepsilon}, \\
        &\left|\mathbf{1}_{n}^{\top}\partial_{i,j,k}^{3}\Sigma_{n}(s_{\theta}^{X})^{-1}\mathbf{1}_{n}\right|
        \leq C (\mathbf{1}_{n}^{\top}\Sigma_{n}(s_{\theta}^{X})^{-1}\mathbf{1}_{n}) n^{\varepsilon}
    \end{align*}
    and
    \begin{align*}
        &\mathbf{1}_{n}^{\top}\Sigma_{n}(s_{\theta}^{X})^{-1} \Sigma_{n}(s_{\theta_{0}}^{X}) \Sigma_{n}(s_{\theta}^{X})^{-1}\mathbf{1}_{n}
        \leq C (\mathbf{1}_{n}^{\top}\Sigma_{n}(s_{\theta}^{X})^{-1}\mathbf{1}_{n}) n^{(\alpha_{X}(\xi_{0})-\alpha_{X}(\xi))_{+}}, \\
        &\mathbf{1}_{n}^{\top} \partial_{i}\Sigma_{n}(s_{\theta}^{X})^{-1} \Sigma_{n}(s_{\theta_{0}}^{X}) \partial_{i}\Sigma_{n}(s_{\theta}^{X})^{-1} \mathbf{1}_{n}
        \leq C (\mathbf{1}_{n}^{\top}\Sigma_{n}(s_{\theta}^{X})^{-1}\mathbf{1}_{n}) n^{(\alpha_{X}(\xi_{0})-\alpha_{X}(\xi))_{+}+\varepsilon}, \\
        &\mathbf{1}_{n}^{\top} \partial_{i,j}^{2}\Sigma_{n}(s_{\theta}^{X})^{-1} \Sigma_{n}(s_{\theta_{0}}^{X}) \partial_{i,j}^{2}\Sigma_{n}(s_{\theta}^{X})^{-1} \mathbf{1}_{n}
        \leq C (\mathbf{1}_{n}^{\top}\Sigma_{n}(s_{\theta}^{X})^{-1}\mathbf{1}_{n}) n^{(\alpha_{X}(\xi_{0})-\alpha_{X}(\xi))_{+}+\varepsilon}, \\
        &\mathbf{1}_{n}^{\top} \partial_{i,j,k}^{3}\Sigma_{n}(s_{\theta}^{X})^{-1} \Sigma_{n}(s_{\theta_{0}}^{X}) \partial_{i,j,k}^{3}\Sigma_{n}(s_{\theta}^{X})^{-1} \mathbf{1}_{n}
        \leq C (\mathbf{1}_{n}^{\top}\Sigma_{n}(s_{\theta}^{X})^{-1}\mathbf{1}_{n}) n^{(\alpha_{X}(\xi_{0})-\alpha_{X}(\xi))_{+}+\varepsilon}.
    \end{align*}
\end{lemma}

\begin{lemma}\label{Lemma:moment-dv-mu-xi}
    For any $q\in\N$, there exists a positive constant $C_{q}$ such that for any $i,j,k\in\{1,2,\cdots,p-1\}$,  $\varepsilon>0$ and $\theta\in\Theta$, 
    \begin{align*}
        &\E_{\vartheta_{0}}^{n}\left[ \left|\mu_{n}(\xi)-\mu_{0}\right|^{q} \right]
        \leq C_{q} (\mathbf{1}_{n}^{\top}\Sigma_{n}(s_{\xi}^{X})^{-1}\mathbf{1}_{n})^{-\frac{q}{2}} n^{\frac{q}{2}(\alpha_{X}(\xi_{0})-\alpha_{X}(\xi))_{+}}, \\
        &\E_{\vartheta_{0}}^{n}\left[ \left|\partial_{i}\mu_{n}(\xi)\right|^{q} \right]
        \leq C_{q} (\mathbf{1}_{n}^{\top}\Sigma_{n}(s_{\xi}^{X})^{-1}\mathbf{1}_{n})^{-\frac{q}{2}} n^{\frac{q}{2}(\alpha_{X}(\xi_{0})-\alpha_{X}(\xi))_{+}+\varepsilon}, \\
        &\E_{\vartheta_{0}}^{n}\left[ \left|\partial_{i,j}^{2}\mu_{n}(\xi)\right|^{q} \right] 
        \leq C_{q} (\mathbf{1}_{n}^{\top}\Sigma_{n}(s_{\xi}^{X})^{-1}\mathbf{1}_{n})^{-\frac{q}{2}} n^{\frac{q}{2}(\alpha_{X}(\xi_{0})-\alpha_{X}(\xi))_{+}+\varepsilon}, \\
        &\E_{\vartheta_{0}}^{n}\left[ \left|\partial_{i,j,k}^{3}\mu_{n}(\xi)\right|^{q} \right]
        \leq C_{q} (\mathbf{1}_{n}^{\top}\Sigma_{n}(s_{\xi}^{X})^{-1}\mathbf{1}_{n})^{-\frac{q}{2}} n^{\frac{q}{2}(\alpha_{X}(\xi_{0})-\alpha_{X}(\xi))_{+}+\varepsilon}.
    \end{align*}
\end{lemma}

To prove consistency and the asymptotic normality of the exact MLE in Theorem~\ref{Thm:MLE}, we need to verify uniform convergences of $\sigma^{2}(\xi)$ and its derivatives. Then we repeatedly use the following Sobolev inequality, 
which can be proved using Theorem~4.12 of \cite{Adams-Fournier-2003-Sobolev} and the Fubini theorem. 
\begin{lemma}[The Sobolev Inequality]\label{Lemma:Sobolev}
    Let $d\in\mathbb{N}$, $\mathring{\Theta}_{\ast}$ be a bounded open cube in $\mathbb{R}^{d}$, $\Theta_\ast$ be the closure of $\Theta_{\ast}$, and $\{(\mathcal{X}_{n},\mathcal{A}_{n},\mathbb{P}^{n}_{\ast})\}_{n\in\mathbb{N}}$ be a sequence of complete probability spaces.
    Assume that $\{u_{n}(\theta,x_n)\}_{(\theta,x_n)\in\Theta_{\ast}\times\mathcal{X}_n}$ is a sequence of pathwise continuously differentiable random fields, {\it i.e.} for each $n\in\mathbb{N}$, it holds that
    \begin{itemize}
        \item the function $\theta\mapsto u_{n}(\theta,\omega_{n})$ is continuously differentiable on $\mathring\Theta_{\ast}$ and uniformly continuous on $\Theta_\ast$ for $\mathbb{P}^{n}_{\ast}$-a.s. $x_n\in\mathcal{X}_{n}$, 
        \item the functions $u_{n}(\cdot,\cdot)$ and $\partial_\theta u_{n}(\cdot,\cdot)$ on $\Theta_\ast\times\mathcal{X}_n$ are $\mathcal{B}(\Theta_\ast)\otimes\mathcal{X}_n$-measurable, where $\mathcal{B}(\Theta_\ast)$ denotes the Borel $\sigma$-algebra on the set $\Theta_\ast$. 
    \end{itemize}
    Then for any $q\in\mathbb{N}$ satisfying $q>d$, there exists a positive constant $C_1=C_1(q,d)$, which is independent of $n$, $x_n$ and $u_n(\cdot,\cdot)$, such that 
    \begin{align*}
        &\sup_{\theta^{\prime}\in\Theta_\ast} \left|u_{n}(\theta^{\prime},x_{n})\right|^{q}
        \leq C_1 \left[
        \int_{\Theta_\ast}
        \left|u_{n}(\theta^{\prime},x_{n})\right|^{q}\,\mathrm{d}\theta^{\prime}
        + \int_{\Theta_\ast}
        \left\|\partial_{\theta}u_{n}(\theta^{\prime},x_{n})\right\|_{\mathbb{R}^{d}}^{q}\,\mathrm{d}\theta^{\prime}
        \right]
    \end{align*}
    hold for any $n\in\mathbb{N}$ and $\mathbb{P}^{n}_{\ast}$-a.s. $\omega_{n}\in\mathcal{X}_{n}$. 
    In particular, we get
    \begin{align*}
        \E^{\mathbb{P}^{n}_{\ast}}\biggl[ \sup_{\theta^{\prime}\in\Theta_\ast} \left|u_{n}(\theta^{\prime},\cdot)\right|^{q} \biggr]
        \leq C_2 \sup_{\theta^{\prime}\in\Theta_\ast} \left\{
        \E^{\mathbb{P}^{n}_{\ast}}\bigl[ \bigl|u_{n}(\theta^{\prime},\cdot)\bigr|^{q} \bigr]
        + \E^{\mathbb{P}^{n}_{\ast}}\bigl[ \left\|\partial_{\theta}u_{n}(\theta^{\prime},\cdot)\right\|_{\mathbb{R}^{d}}^{q} \bigr]
        \right\}
    \end{align*}
    for the positive constant $C_2 := C_1(q,d) m_d(\Theta_\ast)$, where $m_d(A)$ denotes the Lebesgue measure of a measurable set $A$ of $\R^d$. 
\end{lemma}

\begin{remark}\rm
    We provide a clarification regarding the proof of Lemma~\ref{Lemma:Sobolev}.
    In the Sobolev inequality and the Sobolev embedding theorem, the geometric structure of the domain and its boundary plays an essential role.
    The version of the Sobolev embedding theorem given in Theorem~4.12 of \cite{Adams-Fournier-2003-Sobolev} is stated under the assumption that the domain $\Theta_\ast$ satisfies the {\it strong locally Lipschitz condition}; see Section~4.9 of \cite{Adams-Fournier-2003-Sobolev} for its definition.
    However, since $\Theta_\ast$ is bounded, it suffices to assume that $\Theta_\ast$ has a {\it locally Lipschitz boundary}, that is, for each point $\theta_\ast \in \partial\Theta_\ast$, there exists a neighborhood $U(\theta_\ast)$ such that $U(\theta_\ast) \cap \partial\Theta_\ast$ is the graph of a Lipschitz continuous function.
    In particular, any bounded open cube in $\mathbb{R}^{d}$ has a locally Lipschitz boundary.
    Therefore, Theorem~4.12 in \cite{Adams-Fournier-2003-Sobolev} directly implies the result stated in Lemma~\ref{Lemma:Sobolev}.
\end{remark}

Finally, we recall Lemma~3 of \cite{Lieberman-2012}, which provides a precise approximation error bound for the trace of the product of Toeplitz matrices and the inverses of (possibly different) Toeplitz matrices. The entries of these matrices are defined via the Fourier transforms of the spectral density function and their derivatives with respect to model parameters.
This result is particularly useful for evaluating the cumulants of quadratic forms of Gaussian vectors arising from stationary Gaussian time series, as well as the cumulants of their derivatives with respect to model parameters, where the matrix in the quadratic form is given by the inverse of a Toeplitz matrix associated with a spectral density function.

Before stating the result of  Lemma~3 of \cite{Lieberman-2012}, we prepare notation. 
Let $\Pi:=[-\pi,\pi]$. For non-negative sequences $\{a_{n}\}_{n\in\mathbb{N}}$ and $\{b_{n}\}_{n\in\mathbb{N}}$, we write $a_{n}\lesssim b_{n}$ if there exists a constant $C>0$ such that $a_{n}\leq Cb_{n}$ for sufficiently large $n$. 
For a set $\Theta_{\ast}$ of $\mathbb{R}^{q}$ and sequences of positive functions $\{a_{n}(\theta)\}_{n\in\mathbb{N}}$ and $\{b_{n}(\theta)\}_{n\in\mathbb{N}}$ on $\Theta_{\ast}$, we write $a_{n}(\theta)\lesssim_{u} b_{n}(\theta)$ (resp.~$a_{n}(\theta)=o_{u}(b_{n}(\theta))$ as $n\to\infty$) {compact} uniformly on $\Theta_{\ast}$ if $\sup_{\theta\in\mathcal{K}}|a_{n}(\theta)/b_{n}(\theta)|\lesssim 1$ (resp.~$\sup_{\theta\in\mathcal{K}}|a_{n}(\theta)/b_{n}(\theta)|=o(1)$ as $n\to\infty$) for any compact subset $\mathcal{K}$ of $\Theta_{\ast}$. 
Moreover, for a set $A$ of $\mathbb{R}^{r}$ and sequences of functions $\{a_{n}(x,\theta)\}_{n\in\mathbb{N}}$ and $\{b_{n}(x,\theta)\}_{n\in\mathbb{N}}$ on $\mathbb{R}^{r}\times\Theta_{\ast}$ which are always positive on $A\times\Theta_{\ast}$, we write $a_{n}(x,\theta)\lesssim_{u} b_{n}(x,\theta)$ uniformly on $A\times\Theta_{\ast}$ if $\sup_{x\in A}|a_{n}(x,\theta)/b_{n}(x,\theta)|\lesssim_{u} 1$ {compact} uniformly on $\Theta_{\ast}$. 
Finally, we introduce the following function spaces ${\mathcal{F}_{\gamma}}$ and ${\mathcal{F}_{\gamma}^{(1)}}$ depending on some continuous function $\gamma:\Theta\to(-\infty,1)$:
	\begin{align*}
		&{\mathcal{F}_{\gamma}}:=\left\{f\in L^{1}(\Pi\times\Theta):~|x|^{{\gamma(\theta)}}|f(x,{\theta})|\lesssim_{u} 1~\mbox{uniformly on $\Pi_{0}\times\Theta$}\right\},\\ 
		&{\mathcal{F}_{\gamma}^{(1)}}:=\left\{f\in {\mathcal{F}_{\gamma}}\cap C^{1}_{0}(\Pi\times\Theta):|x|^{{\gamma(\theta)}+1}\left|\frac{\mathrm{d}f}{\mathrm{d}x}(x,{\theta})\right|\lesssim_{u} 1~\mbox{uniformly on $\Pi_{0}\times\Theta$}\right\},
	\end{align*}
	where $L^{1}(\Pi\times\Theta)$ (resp.~$C^{1}_{0}(\Pi\times\Theta)$) denotes the set of functions $f(x,{\theta})$ on $\Pi\times\Theta$ such that ${x}\mapsto f(x,{\theta})$ is integrable on $\Pi$~(resp.~{continuously} differentiable on $\Pi_{0}$) for each ${\theta}\in\Theta$.

\begin{lemma}[cf. Lemma~3 of \cite{Lieberman-2012}]\label{Lemma:trace-app-inv}
    Let $q\in\mathbb{N}$ {and $\alpha_{j}$ and $\beta_{j}$ be continuous functions on $\Theta_\xi$ to $(-\infty,1)$ for each $j=1,\cdots,p$}. 
	For all $j=1,\cdots,q$, we assume {
		$g_{j}\in{\mathcal{F}_{\alpha_{j}}^{(1)}}$, $h_{j}\in{\mathcal{F}_{\beta_{j}}^{(1)}}$} and $x\mapsto g_{j}(x,{\theta})$ and $x\mapsto h_{j}(x,{\theta})$ can be extended to periodic functions on $\mathbb{R}$ with period $2\pi$ for each ${\theta}\in\Theta$. 
	Set $\psi_{q}(\underline{\xi}):=\sum_{r=1}^{q}(\alpha(\xi_{2r-1})-\alpha(\xi_{2r}))$ for $\underline{\xi}=(\xi_{1},\ldots,\xi_{2r})^{\top}$. Then we obtain
	\begin{equation*}
		n^{-\psi_{q}(\underline{\xi})-\epsilon}\left|
		\mathrm{Tr}\left[\prod_{r=1}^{q}\Sigma_{n}(f_{r,\theta_{2r-1}})\Sigma_{n}(f_{\theta_{2r}})^{-1}\right]
	   	-\frac{n}{2\pi}\int_{-\pi}^{\pi}\prod_{r=1}^{q}\frac{f_{r,\theta_{2r-1}}(x)}{f_{\theta_{2r}}^{n}(x)}\,\mathrm{d}x
	   	\right| = o(1)
	\end{equation*}
	as $n\to\infty$ uniformly on compact subsets of $\Theta_{\xi}(\iota)$ for any $\epsilon>0$ and $\iota\in(0,1)$. 
\end{lemma}
We provide a remark on the proof of Lemma~3 in \cite{Lieberman-2012}.
The proof relies critically on Theorem~2 of \cite{Lieberman-Phillips-2004}, whose proof, however, contains non-trivial errors, as pointed out by \cite{Ginovyan-Sahakyan-2013}.
One of the authors has recently succeeded in correcting these errors and in establishing a slightly generalized version of Theorem~2 in \cite{Lieberman-Phillips-2004}; see \cite{Takabatake-2024-JTSA-main} for details.
As a result, the assertion of Lemma~3 in \cite{Lieberman-2012} remains valid and can now be justified using Theorem~1 in the supplementary article of \cite{Takabatake-2024-JTSA-main}, in place of the original Theorem~2 in \cite{Lieberman-Phillips-2004}.

Here we prepare the uniform convergence of $\partial_{i}\bar{\sigma}^{2}_{n}(\xi)$ and $ \partial_{i,j}^{2}\bar{\sigma}^{2}_{n}(\xi)$ on $\Theta_{\xi}(\iota)$, whose proof is left to Section~\ref{Sec:proof-Lemma:unifconv-sig2-dv1-dv2}.
\begin{lemma}\label{Lemma:unifconv-sig2-dv1-dv2}
    Under Assumption~$\ref{Assump:DSPD1}$, we can show
    \begin{align*}
        &\sup_{\xi\in\Theta_{\xi}(\iota)} \left| \partial_{i}\bar{\sigma}^{2}_{n}(\xi) - \partial_{i}\sigma^{2}(\xi) \right| = o_{\P_{\vartheta_{0}}^{n}}(1), \\
        &\sup_{\xi\in\Theta_{\xi}(\iota)} \left| \partial_{i,j}^{2}\bar{\sigma}^{2}_{n}(\xi) - \partial_{i,j}^{2}\sigma^{2}(\xi) \right| = o_{\P_{\vartheta_{0}}^{n}}(1)
    \end{align*}
    for each $i,j=1,\ldots,p-1$ and $\iota\in(0,1)$, where $\sigma^{2}(\xi)$ is defined in \eqref{expr:sig2-xi}. 
    In particular, we can write
    \begin{align*}
        &\partial_{i}\sigma^{2}(\xi) = \frac{\sigma^{2}_{0}}{2\pi}\int_{-\pi}^{\pi} \left( -\frac{\partial_{i}s_{\xi}^{X}(\omega)}{s_{\xi}^{X}(\omega)} \right)
         \left( \frac{s_{\xi_{0}}^{X}(\omega)}{s_{\xi}^{X}(\omega)} \right)\dd\omega 
        = -\frac{\sigma^{2}_{0}}{2\pi}\int_{-\pi}^{\pi} \partial_{i}\log s_{\xi}^{X}(\omega)
         \left( \frac{s_{\xi_{0}}^{X}(\omega)}{s_{\xi}^{X}(\omega)} \right)\dd\omega, \\
        &\partial_{i,j}^{2}\sigma^{2}(\xi) 
        = \frac{\sigma^{2}_{0}}{2\pi}\int_{-\pi}^{\pi} \left( -\frac{\partial_{i,j}^{2}s_{\xi}^{X}(\omega)}{s_{\xi}^{X}(\omega)} \right) \left( \frac{s_{\xi_{0}}^{X}(\omega)}{s_{\xi}^{X}(\omega)} \right)\dd\omega
        + \frac{\sigma^{2}_{0}}{2\pi}\int_{-\pi}^{\pi} 2\bigl( \partial_{i}\log s_{\xi}^{X}(\omega) \bigr) \bigl( \partial_{j}\log s_{\xi}^{X}(\omega) \bigr)\left( \frac{s_{\xi_{0}}^{X}(\omega)}{s_{\xi}^{X}(\omega)} \right)\dd\omega.
    \end{align*}
\end{lemma}

Define $(p-1)$th square matrix valued continuous functions $\mathcal{G}_{p-1,n}(\xi):=(\mathcal{G}_{n}^{i,j}(\xi))_{i,j=1,\ldots,p-1}$ and $\mathcal{G}_{p-1}(\xi,\theta_{0}):=(\mathcal{G}^{i,j}(\xi,\theta_{0}))_{i,j=1,\ldots,p-1}$ by $\mathcal{G}_{n}^{i,j}(\xi):=-\frac{1}{n}\partial_{i,j}^{2}\bar{\ell}_{n}(\xi)$ and 
\begin{align*}
    \mathcal{G}^{i,j}(\xi,\theta_{0}) :=& -\frac{1}{2}\left( -\frac{\partial_{i}\sigma^{2}(\xi)}{\sigma^{2}(\xi)} \right) 
    \left( -\frac{\partial_{j}\sigma^{2}(\xi)}{\sigma^{2}(\xi)} \right) 
    + \frac{1}{2\sigma^{2}(\xi)}\partial_{i,j}^{2}\sigma^{2}(\xi) \\
    &+\frac{1}{4\pi}\int_{-\pi}^{\pi} \frac{\partial_{i,j}^{2}s_{\xi}^{X}(\omega)}{s_{\xi}^{X}(\omega)} \dd\omega 
    -\frac{1}{4\pi}\int_{-\pi}^{\pi}  \left( \partial_{i}\log s_{\xi}^{X}(\omega) \right) \left( \partial_{j}\log s_{\xi}^{X}(\omega) \right) \dd\omega
\end{align*}
for each $i,j=1,\ldots,p-1$. Notice that, using the expressions of $\partial_{i}\sigma^{2}(\xi_{0})$ and $ \partial_{i,j}^{2}\sigma^{2}(\xi_{0})$ in Lemma~\ref{Lemma:unifconv-sig2-dv1-dv2}, we can write
\begin{align*}
    \mathcal{G}_{p-1}(\xi_{0}) := \mathcal{G}_{p-1}(\xi_{0},\theta_{0}) = -\frac{1}{2}a_{p-1}(\xi_{0})a_{p-1}(\xi_{0})^{\top} + \mathcal{F}_{p-1}(\xi_{0}).
\end{align*}
Moreover, using the expression in \eqref{expr:Ln-dv2}, Lemma~\ref{Lemma:trace-app-inv}, the uniform convergence in \eqref{unifconv-sig2-bar} and Lemma~\ref{Lemma:unifconv-sig2-dv1-dv2}, we obtain
\begin{align}\label{unifconv-ell-bar-dv2}
    \sup_{\xi\in\Theta_{\xi}(\iota)} | \mathcal{G}_{n}^{i,j}(\xi) - \mathcal{G}^{i,j}(\xi,\theta_{0}) | = o_{\P_{\vartheta_{0}}^{n}}(1)
    \ \ \mbox{as $n\to\infty$}
\end{align}
for each $i,j=1,\ldots,p-1$.\\

Moreover, we also prepare the following central limit theorem, whose proof is omitted since a stronger result than Lemma~\ref{Lemma:CLT-sig2} is proved in \eqref{CLT-zeta-n} and it can be proved as a corollary of \eqref{CLT-zeta-n}. 
\begin{lemma}\label{Lemma:CLT-sig2}
    Under Assumption~$\ref{Assump:DSPD1}$, we can show
    \begin{align*}
        \bar{\zeta}_n := \diag\left(-\frac{\sqrt{n}}{2\sigma^{2}_{0}}I_{p-1},\frac{\sqrt{n}}{\sigma^3_0}\right)
        \begin{pmatrix}
            \partial_{\xi}\bar{\sigma}^{2}_{n}(\xi_{0}) - \E_{\vartheta_{0}}^{n}[\partial_{\xi}\tilde{\sigma}^{2}_{n}(\theta_{0})] \\
            \bar{\sigma}^{2}_{n}(\xi_{0}) - \sigma^{2}_{0} 
        \end{pmatrix}
        \to \mathcal{N}(0,\mathcal{F}_{p}(\theta_{0}))
        \ \ \mbox{as $n\to\infty$},
    \end{align*}
    where the matrix $\mathcal{F}_{p}(\theta)$ is defined in \eqref{Def:mat-Fp}. 
\end{lemma}

\begin{proof}[Proof of Lemma~\ref{Lemma:Key-Ineq-mu}]
First, by the chain rule we have
\begin{align*}
    \partial_{i} \Sigma_{n}(s_{\xi}^{X})^{-1} 
    =& - \Sigma_{n}(s_{\xi}^{X})^{-1}\Sigma_{n}(\partial_{i}s_{\xi}^{X})\Sigma_{n}(s_{\xi}^{X})^{-1}, \\
    \partial_{i,j}^{2} \Sigma_{n}(s_{\xi}^{X})^{-1} 
    =& \Sigma_{n}(s_{\xi}^{X})^{-1}\Sigma_{n}(\partial_{j}s_{\xi}^{X})\Sigma_{n}(s_{\xi}^{X})^{-1} \Sigma_{n}(\partial_{i}s_{\xi}^{X})\Sigma_{n}(s_{\xi}^{X})^{-1} \\
    & - \Sigma_{n}(s_{\xi}^{X})^{-1}\Sigma_{n}(\partial_{i,j}^{2}s_{\xi}^{X})\Sigma_{n}(s_{\xi}^{X})^{-1}  
    + \Sigma_{n}(s_{\xi}^{X})^{-1}\Sigma_{n}(\partial_{i}s_{\xi}^{X})\Sigma_{n}(s_{\xi}^{X})^{-1} \Sigma_{n}(\partial_{j}s_{\xi}^{X})\Sigma_{n}(s_{\xi}^{X})^{-1}, \\
    \partial_{i,j,k}^{3} \Sigma_{n}(s_{\xi}^{X})^{-1} 
    =& -\Sigma_{n}(s_{\xi}^{X})^{-1}\Sigma_{n}(\partial_{k}s_{\xi}^{X})\Sigma_{n}(s_{\xi}^{X})^{-1}  \Sigma_{n}(\partial_{j}s_{\xi}^{X})\Sigma_{n}(s_{\xi}^{X})^{-1} \Sigma_{n}(\partial_{i}s_{\xi}^{X})\Sigma_{n}(s_{\xi}^{X})^{-1} \\
    & + \Sigma_{n}(s_{\xi}^{X})^{-1}\Sigma_{n}(\partial_{j,k}^{2}s_{\xi}^{X})\Sigma_{n}(s_{\xi}^{X})^{-1} \Sigma_{n}(\partial_{i}s_{\xi}^{X})\Sigma_{n}(s_{\xi}^{X})^{-1} \\
    & -\Sigma_{n}(s_{\xi}^{X})^{-1}\Sigma_{n}(\partial_{j}s_{\xi}^{X})\Sigma_{n}(s_{\xi}^{X})^{-1} \Sigma_{n}(\partial_{k}s_{\xi}^{X})\Sigma_{n}(s_{\xi}^{X})^{-1} \Sigma_{n}(\partial_{i}s_{\xi}^{X})\Sigma_{n}(s_{\xi}^{X})^{-1} \\
    & + \Sigma_{n}(s_{\xi}^{X})^{-1}\Sigma_{n}(\partial_{j}s_{\xi}^{X})\Sigma_{n}(s_{\xi}^{X})^{-1} \Sigma_{n}(\partial_{i,k}^{2}s_{\xi}^{X})\Sigma_{n}(s_{\xi}^{X})^{-1} \\
    & -\Sigma_{n}(s_{\xi}^{X})^{-1}\Sigma_{n}(\partial_{j}s_{\xi}^{X})\Sigma_{n}(s_{\xi}^{X})^{-1} \Sigma_{n}(\partial_{i}s_{\xi}^{X})\Sigma_{n}(s_{\xi}^{X})^{-1} \Sigma_{n}(\partial_{k}s_{\xi}^{X})\Sigma_{n}(s_{\xi}^{X})^{-1} \\
    & + \Sigma_{n}(s_{\xi}^{X})^{-1} \Sigma_{n}(\partial_{k}s_{\xi}^{X})\Sigma_{n}(s_{\xi}^{X})^{-1} \Sigma_{n}(\partial_{i,j}^{2}s_{\xi}^{X})\Sigma_{n}(s_{\xi}^{X})^{-1} 
    - \Sigma_{n}(s_{\xi}^{X})^{-1}\Sigma_{n}(\partial_{i,j,k}^{3}s_{\xi}^{X})\Sigma_{n}(s_{\xi}^{X})^{-1} \\
    & + \Sigma_{n}(s_{\xi}^{X})^{-1} \Sigma_{n}(\partial_{i,j}^{2}s_{\xi}^{X})\Sigma_{n}(s_{\xi}^{X})^{-1} \Sigma_{n}(s_{\xi}^{X})^{-1} \Sigma_{n}(\partial_{k}s_{\xi}^{X}) \\
    & -\Sigma_{n}(s_{\xi}^{X})^{-1}\Sigma_{n}(\partial_{k}s_{\xi}^{X})\Sigma_{n}(s_{\xi}^{X})^{-1}  \Sigma_{n}(\partial_{i}s_{\xi}^{X})\Sigma_{n}(s_{\xi}^{X})^{-1} \Sigma_{n}(\partial_{j}s_{\xi}^{X})\Sigma_{n}(s_{\xi}^{X})^{-1} \\
    & + \Sigma_{n}(s_{\xi}^{X})^{-1}\Sigma_{n}(\partial_{i,k}^{2}s_{\xi}^{X})\Sigma_{n}(s_{\xi}^{X})^{-1} \Sigma_{n}(\partial_{j}s_{\xi}^{X})\Sigma_{n}(s_{\xi}^{X})^{-1} \\
    & -\Sigma_{n}(s_{\xi}^{X})^{-1}\Sigma_{n}(\partial_{i}s_{\xi}^{X})\Sigma_{n}(s_{\xi}^{X})^{-1} \Sigma_{n}(\partial_{k}s_{\xi}^{X})\Sigma_{n}(s_{\xi}^{X})^{-1} \Sigma_{n}(\partial_{j}s_{\xi}^{X})\Sigma_{n}(s_{\xi}^{X})^{-1} \\
    & + \Sigma_{n}(s_{\xi}^{X})^{-1}\Sigma_{n}(\partial_{i}s_{\xi}^{X})\Sigma_{n}(s_{\xi}^{X})^{-1} \Sigma_{n}(\partial_{j,k}^{2}s_{\xi}^{X})\Sigma_{n}(s_{\xi}^{X})^{-1} \\
    & -\Sigma_{n}(s_{\xi}^{X})^{-1}\Sigma_{n}(\partial_{i}s_{\xi}^{X})\Sigma_{n}(s_{\xi}^{X})^{-1} \Sigma_{n}(\partial_{j}s_{\xi}^{X})\Sigma_{n}(s_{\xi}^{X})^{-1} \Sigma_{n}(\partial_{k}s_{\xi}^{X})\Sigma_{n}(s_{\xi}^{X})^{-1}.
\end{align*}

In the rest of the proof, we only prove the second assertion because, from the above expressions of the derivatives $\partial_{i}\Sigma_{n}(s_{\xi}^{X})^{-1}$, $\partial_{i,j}^{2}\Sigma_{n}(s_{\xi}^{X})^{-1}$ and $\partial_{i,j,k}^{3}\Sigma_{n}(s_{\xi}^{X})^{-1}$, we can see that the other assertions can be proved similarly. 
Notice that we can show
\begin{align*}
    \left| \mathbf{1}_{n}^{\top}\partial_{i,j}^{2}\Sigma_{n}(s_{\theta}^{X})^{-1}\mathbf{1}_{n} \right|
    \leq& \mathbf{1}_{n}^{\top} \Sigma_{n}(s_{\theta}^{X})^{-1} \Sigma_{n}(|\partial_{i,j}^{2}s_{\theta}^{X}|) \Sigma_{n}(s_{\theta}^{X})^{-1}\mathbf{1}_{n} \\
    &+ 2\sum_{m_{1},m_{2}\in\{+,-\}} \left| \mathbf{1}_{n}^{\top} \Sigma_{n}(s_{\theta}^{X})^{-1} \Sigma_{n}((\partial_{j}s_{\theta}^{X})_{m_{2}}) \Sigma_{n}(s_{\theta}^{X})^{-1}  \Sigma_{n}((\partial_{i}s_{\theta}^{X})_{m_{1}}) \Sigma_{n}(s_{\theta}^{X})^{-1}\mathbf{1}_{n} \right|.
\end{align*}
and
\begin{align*}
    &\sum_{m_{1},m_{2}\in\{+,-\}} \left| \mathbf{1}_{n}^{\top} \Sigma_{n}(s_{\theta}^{X})^{-1} \Sigma_{n}((\partial_{j}s_{\theta}^{X})_{m_{2}}) \Sigma_{n}(s_{\theta}^{X})^{-1}  \Sigma_{n}((\partial_{i}s_{\theta}^{X})_{m_{1}}) \Sigma_{n}(s_{\theta}^{X})^{-1}\mathbf{1}_{n} \right| \\
    &\leq \sum_{m_{1},m_{2}\in\{+,-\}} \left\| \Sigma_{n}(s_{\theta}^{X})^{-\frac{1}{2}} \Sigma_{n}((\partial_{j}s_{\theta}^{X})_{m_{1}}) \Sigma_{n}(s_{\theta}^{X})^{-1}\mathbf{1}_{n} \right\|_{\R^{n}} 
    \left\| \Sigma_{n}(s_{\theta}^{X})^{-\frac{1}{2}} \Sigma_{n}((\partial_{i}s_{\theta}^{X})_{m_{1}}) \Sigma_{n}(s_{\theta}^{X})^{-1}\mathbf{1}_{n} \right\|_{\R^{n}} \\
    &\leq \sum_{m_{1},m_{2}\in\{+,-\}} \left\| \Sigma_{n}(s_{\theta}^{X})^{-\frac{1}{2}} \Sigma_{n}((\partial_{j}s_{\theta}^{X})_{m_{1}})^{\frac{1}{2}} \right\|_{\mathrm{op}}^{2} 
    \left\| \Sigma_{n}(s_{\theta}^{X})^{-\frac{1}{2}} \Sigma_{n}((\partial_{i}s_{\theta}^{X})_{m_{1}})^{\frac{1}{2}} \right\|_{\mathrm{op}}^{2}
    \left\| \Sigma_{n}(s_{\theta}^{X})^{-\frac{1}{2}} \mathbf{1}_{n} \right\|_{\R^{n}}^{2}
\end{align*}
so that we conclude the second assertions using  
Lemma~5.3 in \cite{Dahlhaus-1989} and Lemma~6 in the full version of \cite{Lieberman-2012}. 
Therefore we finish the proof. 
%%%%%%%%%%   
\end{proof}

~\\
\begin{proof}[Proof of Lemma~\ref{Lemma:moment-dv-mu-xi}]
Since we have
\begin{align}\label{expr:mu-xi-error}
    \mu_{n}(\xi)-\mu_{0} =  \frac{ \mathbf{1}_{n}^{\top}\Sigma_{n}(s_{\xi}^{X})^{-1}(\mathbf{X}_{n}-\mu_{0}\mathbf{1}_{n}) }{ \mathbf{1}_{n}^{\top}\Sigma_{n}(s_{\xi}^{X})^{-1}\mathbf{1}_{n} }
    \sim \mathcal{N}\left( 0, \frac{ \mathbf{1}_{n}^{\top}\Sigma_{n}(s_{\theta}^{X})^{-1} \Sigma_{n}(s_{\theta_{0}}^{X}) \Sigma_{n}(s_{\theta}^{X})^{-1}\mathbf{1}_{n} }{ (\mathbf{1}_{n}^{\top}\Sigma_{n}(s_{\theta}^{X})^{-1}\mathbf{1}_{n})^{2} } \right)
\end{align}
under the distribution $\P_{\vartheta_{0}}^{n}$, we can show that 
\begin{align*}
    \E_{\vartheta_{0}}^{n}\left[ \left|\mu_{n}(\xi)-\mu_{0}\right|^{q} \right] = 2^{\frac{q}{2}}\Gamma\left(\frac{q+1}{2}\right)\pi^{-\frac{1}{2}} \left| \frac{ (\mathbf{1}_{n}^{\top}\Sigma_{n}(s_{\theta}^{X})^{-1} \Sigma_{n}(s_{\theta_{0}}^{X}) \Sigma_{n}(s_{\theta}^{X})^{-1}\mathbf{1}_{n})^{\frac{1}{2}} }{ \mathbf{1}_{n}^{\top}\Sigma_{n}(s_{\theta}^{X})^{-1}\mathbf{1}_{n} } \right|^{q} 
\end{align*}
so that the first assertion follows from Lemma~\ref{Lemma:Key-Ineq-mu}.\\ 

Set $\mathbf{Z}_{n}:=\mathbf{X}_{n}-\mu_{0}\mathbf{1}_{n}$. First notice that the derivatives of $\mu_{n}(\xi)$ with respect to $\xi$ can be expressed by 
    \begin{align}
        \partial_{i} \mu_{n}(\xi)
        =& \frac{ \mathbf{1}_{n}^{\top}\partial_{i}\Sigma_{n}(s_{\xi}^{X})^{-1} \mathbf{X}_{n} }{ \mathbf{1}_{n}^{\top}\Sigma_{n}(s_{\xi}^{X})^{-1}\mathbf{1}_{n} }
        - \mu_{n}(\xi)\frac{ \mathbf{1}_{n}^{\top} \partial_{i}\Sigma_{n}(s_{\xi}^{X})^{-1} \mathbf{1}_{n} }{ \mathbf{1}_{n}^{\top}\Sigma_{n}(s_{\xi}^{X})^{-1}\mathbf{1}_{n} } 
        \nonumber \\
        =& \frac{ \mathbf{1}_{n}^{\top}\partial_{i}\Sigma_{n}(s_{\xi}^{X})^{-1} \mathbf{Z}_{n} }{ \mathbf{1}_{n}^{\top}\Sigma_{n}(s_{\xi}^{X})^{-1}\mathbf{1}_{n} }
        -(\mu_{n}(\xi)-\mu_{0}) \frac{ \mathbf{1}_{n}^{\top}\partial_{i}\Sigma_{n}(s_{\xi}^{X})^{-1} \mathbf{1}_{n} }{ \mathbf{1}_{n}^{\top}\Sigma_{n}(s_{\xi}^{X})^{-1}\mathbf{1}_{n} },
        \label{expr:mu-xi-dv1}
    \end{align} 
    \begin{align}
        \partial_{i,j}^{2} \mu_{n}(\xi)
        =& \frac{ \mathbf{1}_{n}^{\top}\partial_{i,j}^{2}\Sigma_{n}(s_{\xi}^{X})^{-1} \mathbf{Z}_{n} }{ \mathbf{1}_{n}^{\top}\Sigma_{n}(s_{\xi}^{X})^{-1}\mathbf{1}_{n} } 
        - \left( \frac{ \mathbf{1}_{n}^{\top}\partial_{i}\Sigma_{n}(s_{\xi}^{X})^{-1} \mathbf{Z}_{n} }{ \mathbf{1}_{n}^{\top}\Sigma_{n}(s_{\xi}^{X})^{-1}\mathbf{1}_{n} } \right) \left( \frac{ \mathbf{1}_{n}^{\top}\partial_{j}\Sigma_{n}(s_{\xi}^{X})^{-1} \mathbf{1}_{n} }{ \mathbf{1}_{n}^{\top}\Sigma_{n}(s_{\xi}^{X})^{-1}\mathbf{1}_{n} } \right)
        - \partial_{j}\mu_{n}(\xi) \frac{ \mathbf{1}_{n}^{\top} \partial_{i}\Sigma_{n}(s_{\xi}^{X})^{-1} \mathbf{1}_{n} }{ \mathbf{1}_{n}^{\top}\Sigma_{n}(s_{\xi}^{X})^{-1}\mathbf{1}_{n} } 
        \nonumber \\
        & - (\mu_{n}(\xi)-\mu_{0})\left[
        \frac{ \mathbf{1}_{n}^{\top}\partial_{i,j}^{2}\Sigma_{n}(s_{\xi}^{X})^{-1} \mathbf{1}_{n} }{ \mathbf{1}_{n}^{\top}\Sigma_{n}(s_{\xi}^{X})^{-1}\mathbf{1}_{n} } - \left( \frac{ \mathbf{1}_{n}^{\top}\partial_{i}\Sigma_{n}(s_{\xi}^{X})^{-1} \mathbf{Z}_{n} }{ \mathbf{1}_{n}^{\top}\Sigma_{n}(s_{\xi}^{X})^{-1}\mathbf{1}_{n} } \right) \left( \frac{ \mathbf{1}_{n}^{\top}\partial_{j}\Sigma_{n}(s_{\xi}^{X})^{-1} \mathbf{1}_{n} }{ \mathbf{1}_{n}^{\top}\Sigma_{n}(s_{\xi}^{X})^{-1}\mathbf{1}_{n} } \right)
        \right]
        \label{exp:dv2-mu-xi}
    \end{align}
    and
    \begin{align}
        \partial_{i,j,k}^{3} \mu_{n}(\xi)
        =& \frac{ \mathbf{1}_{n}^{\top}\partial_{i,j,k}^{3}\Sigma_{n}(s_{\xi}^{X})^{-1} \mathbf{Z}_{n} }{ \mathbf{1}_{n}^{\top}\Sigma_{n}(s_{\xi}^{X})^{-1}\mathbf{1}_{n} } 
        - \left( \frac{ \mathbf{1}_{n}^{\top}\partial_{i,j}^{2}\Sigma_{n}(s_{\xi}^{X})^{-1} \mathbf{Z}_{n} }{ \mathbf{1}_{n}^{\top}\Sigma_{n}(s_{\xi}^{X})^{-1}\mathbf{1}_{n} } \right)\left( \frac{ \mathbf{1}_{n}^{\top}\partial_{k}\Sigma_{n}(s_{\xi}^{X})^{-1} \mathbf{1}_{n} }{ \mathbf{1}_{n}^{\top}\Sigma_{n}(s_{\xi}^{X})^{-1}\mathbf{1}_{n} } \right) \nonumber \\
        &- \left( \frac{ \mathbf{1}_{n}^{\top}\partial_{i,k}^{2}\Sigma_{n}(s_{\xi}^{X})^{-1} \mathbf{Z}_{n} }{ \mathbf{1}_{n}^{\top}\Sigma_{n}(s_{\xi}^{X})^{-1}\mathbf{1}_{n} } \right)\left( \frac{ \mathbf{1}_{n}^{\top}\partial_{j}\Sigma_{n}(s_{\xi}^{X})^{-1} \mathbf{1}_{n} }{ \mathbf{1}_{n}^{\top}\Sigma_{n}(s_{\xi}^{X})^{-1}\mathbf{1}_{n} } \right) 
        -\left( \frac{ \mathbf{1}_{n}^{\top}\partial_{i}\Sigma_{n}(s_{\xi}^{X})^{-1} \mathbf{Z}_{n} }{ \mathbf{1}_{n}^{\top}\Sigma_{n}(s_{\xi}^{X})^{-1}\mathbf{1}_{n} } \right) \left( \frac{ \mathbf{1}_{n}^{\top}\partial_{j,k}^{2}\Sigma_{n}(s_{\xi}^{X})^{-1} \mathbf{1}_{n} }{ \mathbf{1}_{n}^{\top}\Sigma_{n}(s_{\xi}^{X})^{-1}\mathbf{1}_{n} } \right) \nonumber \\
        &+2 \left( \frac{ \mathbf{1}_{n}^{\top}\partial_{i}\Sigma_{n}(s_{\xi}^{X})^{-1} \mathbf{Z}_{n} }{ \mathbf{1}_{n}^{\top}\Sigma_{n}(s_{\xi}^{X})^{-1}\mathbf{1}_{n} } \right) 
        \left( \frac{ \mathbf{1}_{n}^{\top}\partial_{j}\Sigma_{n}(s_{\xi}^{X})^{-1} \mathbf{1}_{n} }{ \mathbf{1}_{n}^{\top}\Sigma_{n}(s_{\xi}^{X})^{-1}\mathbf{1}_{n} } \right) 
        \left( \frac{ \mathbf{1}_{n}^{\top}\partial_{k}\Sigma_{n}(s_{\xi}^{X})^{-1} \mathbf{1}_{n} }{ \mathbf{1}_{n}^{\top}\Sigma_{n}(s_{\xi}^{X})^{-1}\mathbf{1}_{n} } \right) 
        - \partial_{j,k}^{2}\mu_{n}(\xi) \left( \frac{ \mathbf{1}_{n}^{\top} \partial_{i}\Sigma_{n}(s_{\xi}^{X})^{-1} \mathbf{1}_{n} }{ \mathbf{1}_{n}^{\top}\Sigma_{n}(s_{\xi}^{X})^{-1}\mathbf{1}_{n} } \right) \nonumber \\
        &- \partial_{j}\mu_{n}(\xi) \left[
        \frac{ \mathbf{1}_{n}^{\top}\partial_{i,k}^{2}\Sigma_{n}(s_{\xi}^{X})^{-1} \mathbf{1}_{n} }{ \mathbf{1}_{n}^{\top}\Sigma_{n}(s_{\xi}^{X})^{-1}\mathbf{1}_{n} } - \left( \frac{ \mathbf{1}_{n}^{\top}\partial_{i}\Sigma_{n}(s_{\xi}^{X})^{-1} \mathbf{1}_{n} }{ \mathbf{1}_{n}^{\top}\Sigma_{n}(s_{\xi}^{X})^{-1}\mathbf{1}_{n} } \right)
        \left( \frac{ \mathbf{1}_{n}^{\top}\partial_{k}\Sigma_{n}(s_{\xi}^{X})^{-1} \mathbf{1}_{n} }{ \mathbf{1}_{n}^{\top}\Sigma_{n}(s_{\xi}^{X})^{-1}\mathbf{1}_{n} } \right)
        \right] \nonumber \\
        &- \partial_{k}\mu_{n}(\xi)\left[
        \frac{ \mathbf{1}_{n}^{\top}\partial_{i,j}^{2}\Sigma_{n}(s_{\xi}^{X})^{-1} \mathbf{1}_{n} }{ \mathbf{1}_{n}^{\top}\Sigma_{n}(s_{\xi}^{X})^{-1}\mathbf{1}_{n} } - \left( \frac{ \mathbf{1}_{n}^{\top}\partial_{i}\Sigma_{n}(s_{\xi}^{X})^{-1} \mathbf{Z}_{n} }{ \mathbf{1}_{n}^{\top}\Sigma_{n}(s_{\xi}^{X})^{-1}\mathbf{1}_{n} } \right) \left( \frac{ \mathbf{1}_{n}^{\top}\partial_{j}\Sigma_{n}(s_{\xi}^{X})^{-1} \mathbf{1}_{n} }{ \mathbf{1}_{n}^{\top}\Sigma_{n}(s_{\xi}^{X})^{-1}\mathbf{1}_{n} } \right) \right] \nonumber \\
        & - (\mu_{n}(\xi)-\mu_{0})\left[
        \frac{ \mathbf{1}_{n}^{\top}\partial_{i,j,k}^{3}\Sigma_{n}(s_{\xi}^{X})^{-1} \mathbf{1}_{n} }{ \mathbf{1}_{n}^{\top}\Sigma_{n}(s_{\xi}^{X})^{-1}\mathbf{1}_{n} } 
        -\left( \frac{ \mathbf{1}_{n}^{\top}\partial_{i,j}^{2}\Sigma_{n}(s_{\xi}^{X})^{-1} \mathbf{1}_{n} }{ \mathbf{1}_{n}^{\top}\Sigma_{n}(s_{\xi}^{X})^{-1}\mathbf{1}_{n} } \right) 
        \left( \frac{ \mathbf{1}_{n}^{\top}\partial_{k}\Sigma_{n}(s_{\xi}^{X})^{-1} \mathbf{1}_{n} }{ \mathbf{1}_{n}^{\top}\Sigma_{n}(s_{\xi}^{X})^{-1}\mathbf{1}_{n} } \right)
        \right. \nonumber \\
        &\left. \hspace{2.5cm}  - \left( \frac{ \mathbf{1}_{n}^{\top}\partial_{i,k}^{2}\Sigma_{n}(s_{\xi}^{X})^{-1} \mathbf{1}_{n} }{ \mathbf{1}_{n}^{\top}\Sigma_{n}(s_{\xi}^{X})^{-1}\mathbf{1}_{n} } \right) \left( \frac{ \mathbf{1}_{n}^{\top}\partial_{j}\Sigma_{n}(s_{\xi}^{X})^{-1} \mathbf{1}_{n} }{ \mathbf{1}_{n}^{\top}\Sigma_{n}(s_{\xi}^{X})^{-1}\mathbf{1}_{n} } \right) 
        - \left( \frac{ \mathbf{1}_{n}^{\top}\partial_{i}\Sigma_{n}(s_{\xi}^{X})^{-1} \mathbf{1}_{n} }{ \mathbf{1}_{n}^{\top}\Sigma_{n}(s_{\xi}^{X})^{-1}\mathbf{1}_{n} } \right) \left( \frac{ \mathbf{1}_{n}^{\top}\partial_{j,k}^{2}\Sigma_{n}(s_{\xi}^{X})^{-1} \mathbf{1}_{n} }{ \mathbf{1}_{n}^{\top}\Sigma_{n}(s_{\xi}^{X})^{-1}\mathbf{1}_{n} } \right) \right. \nonumber \\ 
        &\left. \hspace{2.5cm} +2 \left( \frac{ \mathbf{1}_{n}^{\top}\partial_{i}\Sigma_{n}(s_{\xi}^{X})^{-1} \mathbf{1}_{n} }{ \mathbf{1}_{n}^{\top}\Sigma_{n}(s_{\xi}^{X})^{-1}\mathbf{1}_{n} } \right)
        \left( \frac{ \mathbf{1}_{n}^{\top}\partial_{j}\Sigma_{n}(s_{\xi}^{X})^{-1} \mathbf{1}_{n} }{ \mathbf{1}_{n}^{\top}\Sigma_{n}(s_{\xi}^{X})^{-1}\mathbf{1}_{n} } \right) \left( \frac{ \mathbf{1}_{n}^{\top}\partial_{k}\Sigma_{n}(s_{\xi}^{X})^{-1} \mathbf{1}_{n} }{ \mathbf{1}_{n}^{\top}\Sigma_{n}(s_{\xi}^{X})^{-1}\mathbf{1}_{n} } \right)
        \right].
        \label{exp:dv3-mu-xi}
    \end{align}
Then we can show that %there exists a positive constant $C_{q}$ such that for any $\theta\in\Theta$ and $\varepsilon>0$, 
\begin{align*}
    \E_{\vartheta_{0}}^{n}\left[ \left|\partial_{i}\mu_{n}(\xi)\right|^{q} \right]
    \leq& 2^{q} \left( \E_{\vartheta_{0}}^{n}\left[ \left| \frac{ \mathbf{1}_{n}^{\top}\partial_{i}\Sigma_{n}(s_{\theta}^{X})^{-1} \mathbf{Z}_{n} }{ \mathbf{1}_{n}^{\top}\Sigma_{n}(s_{\theta}^{X})^{-1}\mathbf{1}_{n} } \right|^{q} \right]
    + \E_{\vartheta_{0}}^{n}\left[ \left|\mu_{n}(\xi)-\mu_{0}\right|^{q} \right] \left| \frac{ \mathbf{1}_{n}^{\top}\partial_{i}\Sigma_{n}(s_{\theta}^{X})^{-1} \mathbf{1}_{n} }{ \mathbf{1}_{n}^{\top}\Sigma_{n}(s_{\theta}^{X})^{-1}\mathbf{1}_{n} } \right|^{q} \right) \\
    =& 2^{q} \left( \left| \frac{ (\mathbf{1}_{n}^{\top}\partial_{i}\Sigma_{n}(s_{\theta}^{X})^{-1} \Sigma_{n}(s_{\theta_{0}}^{X}) \partial_{i}\Sigma_{n}(s_{\theta}^{X})^{-1} \mathbf{1}_{n})^{\frac{1}{2}} }{ \mathbf{1}_{n}^{\top}\Sigma_{n}(s_{\theta}^{X})^{-1}\mathbf{1}_{n} } \right|^{q}
    + \E_{\vartheta_{0}}^{n}\left[ \left|\mu_{n}(\xi)-\mu_{0}\right|^{q} \right] \left| \frac{ \mathbf{1}_{n}^{\top}\partial_{i}\Sigma_{n}(s_{\theta}^{X})^{-1} \mathbf{1}_{n} }{ \mathbf{1}_{n}^{\top}\Sigma_{n}(s_{\theta}^{X})^{-1}\mathbf{1}_{n} } \right|^{q} \right) 
\end{align*}
so that we conclude the second assertion of Lemma~\ref{Lemma:moment-dv-mu-xi} using Lemma~\ref{Lemma:Key-Ineq-mu} and the first assertion of Lemma~\ref{Lemma:moment-dv-mu-xi}. 
Moreover, using Lemma~\ref{Lemma:Key-Ineq-mu}, we can also show that there exists a positive constant $C_{q}$ such that for any $\theta\in\Theta$ and $\varepsilon>0$, 
\begin{align*}
    \E_{\vartheta_{0}}^{n}\left[ \left|\partial_{i,j}^{2}\mu_{n}(\xi)\right|^{q} \right]
    \leq& C_{q} \left( \E_{\vartheta_{0}}^{n}\left[ \left| \frac{ \mathbf{1}_{n}^{\top}\partial_{i,j}^{2} \Sigma_{n}(s_{\xi}^{X})^{-1} \mathbf{Z}_{n} }{ \mathbf{1}_{n}^{\top}\Sigma_{n}(s_{\xi}^{X})^{-1}\mathbf{1}_{n} } \right|^{q} \right]
    + \E_{\vartheta_{0}}^{n}\left[ \left| \frac{ \mathbf{1}_{n}^{\top}\partial_{i}\Sigma_{n}(s_{\xi}^{X})^{-1} \mathbf{Z}_{n} }{ \mathbf{1}_{n}^{\top}\Sigma_{n}(s_{\xi}^{X})^{-1}\mathbf{1}_{n} } \right|^{q} \right] n^{\varepsilon} \right) \\
    & + C_{q}n^{\varepsilon} \left( 
    \E_{\vartheta_{0}}^{n}\left[ \left|\partial_{j}\mu_{n}(\xi)\right|^{q} \right] 
    + \E_{\vartheta_{0}}^{n}\left[ \left|\mu_{n}(\xi)-\mu_{0}\right|^{q} \right] \right)
\end{align*}
and
\begin{align*}
    &\E_{\vartheta_{0}}^{n}\left[ \left|\partial_{i,j,k}^{3}\mu_{n}(\xi)\right|^{q} \right] \\
    &\leq C_{q} \max_{i,j,k\in\{1,\cdots,p\}} \left( 
    \E_{\vartheta_{0}}^{n}\left[ \left| \frac{ \mathbf{1}_{n}^{\top}\partial_{i,j,k}^{3} \Sigma_{n}(s_{\xi}^{X})^{-1} \mathbf{Z}_{n} }{ \mathbf{1}_{n}^{\top}\Sigma_{n}(s_{\xi}^{X})^{-1}\mathbf{1}_{n} } \right|^{q} \right]
    + \E_{\vartheta_{0}}^{n}\left[ \left| \frac{ \mathbf{1}_{n}^{\top}\partial_{i,j}^{2} \Sigma_{n}(s_{\xi}^{X})^{-1} \mathbf{Z}_{n} }{ \mathbf{1}_{n}^{\top}\Sigma_{n}(s_{\xi}^{X})^{-1}\mathbf{1}_{n} } \right|^{q} \right] n^{\varepsilon}
    + \E_{\vartheta_{0}}^{n}\left[ \left| \frac{ \mathbf{1}_{n}^{\top}\partial_{i}\Sigma_{n}(s_{\xi}^{X})^{-1} \mathbf{Z}_{n} }{ \mathbf{1}_{n}^{\top}\Sigma_{n}(s_{\xi}^{X})^{-1}\mathbf{1}_{n} } \right|^{q} \right] n^{\varepsilon} \right) \\
    &\quad + C_{q} n^{\varepsilon} \max_{i,j\in\{1,\cdots,p\}} \left( 
    \E_{\vartheta_{0}}^{n}\left[ \left|\partial_{i,j}^{2}\mu_{n}(\xi)\right|^{q} \right]
    + \E_{\vartheta_{0}}^{n}\left[ \left|\partial_{i}\mu_{n}(\xi)\right|^{q} \right]
    + \E_{\vartheta_{0}}^{n}\left[ \left|\mu_{n}(\xi)-\mu_{0}\right|^{q} \right]  \right).
\end{align*}
Since we can also compute the absolute moments of Gaussian random variables by
\begin{align*}
    \E_{\vartheta_{0}}^{n}\left[ \left| \mathbf{1}_{n}^{\top}\partial_{i,j}^{2} \Sigma_{n}(s_{\xi}^{X})^{-1} \mathbf{Z}_{n} \right|^{q} \right] 
    =& 2^{\frac{q}{2}}\Gamma\left(\frac{q+1}{2}\right)\pi^{-\frac{1}{2}} (\mathbf{1}_{n}^{\top} \partial_{i,j}^{2}\Sigma_{n}(s_{\xi}^{X})^{-1} \Sigma_{n}(s_{\theta_{0}}^{X}) \partial_{i,j}^{2}\Sigma_{n}(s_{\xi}^{X})^{-1} \mathbf{1}_{n})^{\frac{q}{2}}, \\
    \E_{\vartheta_{0}}^{n}\left[ \left| \mathbf{1}_{n}^{\top}\partial_{i,j,k}^{3} \Sigma_{n}(s_{\xi}^{X})^{-1} \mathbf{Z}_{n} \right|^{q} \right] 
    =& 2^{\frac{q}{2}}\Gamma\left(\frac{q+1}{2}\right)\pi^{-\frac{1}{2}} (\mathbf{1}_{n}^{\top} \partial_{i,j,k}^{3}\Sigma_{n}(s_{\xi}^{X})^{-1} \Sigma_{n}(s_{\theta_{0}}^{X}) \partial_{i,j,k}^{3}\Sigma_{n}(s_{\xi}^{X})^{-1} \mathbf{1}_{n})^{\frac{q}{2}},
\end{align*}
we can also conclude the third and the fourth assertions of Lemma~\ref{Lemma:moment-dv-mu-xi} using Lemma~\ref{Lemma:Key-Ineq-mu} and the other assertion of Lemma~\ref{Lemma:moment-dv-mu-xi} similarly. 
Therefore, the proof is complete. 
\end{proof}

\subsection{Proof of Lemma~\ref{Lemma:unifconv-sig2-dv1-dv2}}\label{Sec:proof-Lemma:unifconv-sig2-dv1-dv2}

Recall that 
\begin{align*}
    \bar{\sigma}^{2}_{n}(\xi) 
    = \frac{1}{n}\left(\mathbf{X}_{n}-\mu_{n}(\xi)\mathbf{1}_{n}\right)^{\top}\Sigma_{n}(s_{\xi}^{X})^{-1}\left(\mathbf{X}_{n}-\mu_{n}(\xi)\mathbf{1}_{n}\right),\ \  \mu_{n}(\xi) = \frac{ \mathbf{1}_{n}^{\top}\Sigma_{n}(s_{\xi}^{X})^{-1}\mathbf{X}_{n} }{ \mathbf{1}_{n}^{\top}\Sigma_{n}(s_{\xi}^{X})^{-1}\mathbf{1}_{n} },
\end{align*}
and $e_{n,1}(\xi):=\bar{\sigma}^{2}_{n}(\xi)-\tilde{\sigma}^{2}_{n}(\xi)$, $e_{n,2}(\xi):=\E_{\vartheta_{0}}^{n}[\tilde{\sigma}^{2}_{n}(\xi)]-\sigma^{2}(\xi)$,  $e_{n,3}(\xi):=\tilde{\sigma}^{2}_{n}(\xi)-\E_{\vartheta_{0}}^{n}[\tilde{\sigma}^{2}_{n}(\xi)]$, and
\begin{align*}
    \partial_{\xi}e_{n,1}(\xi)
    = -2n^{-1}(\mu_{n}(\xi)-\mu_{0}) \partial_{\xi}\mu_{n}(\xi) \mathbf{1}_{n}^{\top}\Sigma_{n}(s_{\xi}^{X})^{-1}\mathbf{1}_{n} 
    -n^{-1}(\mu_{n}(\xi)-\mu_{0})^{2} \mathbf{1}_{n}^{\top}\partial_{\xi}\Sigma_{n}(s_{\xi}^{X})^{-1}\mathbf{1}_{n},
\end{align*}
see \eqref{expr:en1-dv1}. 
In the rest of the proof, we will show the error terms $e_{n,i}(\xi)$, $i=1,2,3$ are negligible uniformly in $\xi\in\Theta_{\xi}(\iota)$ as $n\to\infty$.\\

First, we evaluate the first term $e_{n,1}(\xi)$. Notice that we can show 
\begin{align*}
    \partial_{i,j}^{2}e_{n,1}(\xi) 
    =& -2n^{-1}\partial_{j}\mu_{n}(\xi)\partial_{i}\mu_{n}(\xi) (\mathbf{1}_{n}^{\top}\Sigma_{n}(s_{\theta}^{X})^{-1}\mathbf{1}_{n})
     -2n^{-1}(\mu_{n}(\xi)-\mu_{0}) \partial_{i,j}^{2}\mu_{n}(\xi) (\mathbf{1}_{n}^{\top}\Sigma_{n}(s_{\theta}^{X})^{-1}\mathbf{1}_{n}) \\
    & -2n^{-1}(\mu_{n}(\xi)-\mu_{0}) \partial_{i}\mu_{n}(\xi) (\mathbf{1}_{n}^{\top}\partial_{j}\Sigma_{n}(s_{\theta}^{X})^{-1}\mathbf{1}_{n}) \\
    & -2n^{-1}(\mu_{n}(\xi)-\mu_{0})\partial_{j}\mu_{n}(\xi) (\mathbf{1}_{n}^{\top}\partial_{i}\Sigma_{n}(s_{\theta}^{X})^{-1}\mathbf{1}_{n}) 
    -n^{-1}(\mu_{n}(\xi)-\mu_{0})^{2} \mathbf{1}_{n}^{\top}\partial_{i,j}^{2}\Sigma_{n}(s_{\theta}^{X})^{-1}\mathbf{1}_{n}
\end{align*}
and
\begin{align*}
    \partial_{i,j,k}^{3}e_{n,1}(\xi) 
    =& -2n^{-1}\bigl( \partial_{j,k}^{2}\mu_{n}(\xi) \partial_{i}\mu_{n}(\xi) + \partial_{j}\mu_{n}(\xi) \partial_{i,k}^{2}\mu_{n}(\xi) \bigr) \mathbf{1}_{n}^{\top}\Sigma_{n}(s_{\theta}^{X})^{-1}\mathbf{1}_{n} \\
    &-2n^{-1} \partial_{j}\mu_{n}(\xi)\partial_{i}\mu_{n}(\xi) (\mathbf{1}_{n}^{\top}\partial_{k}\Sigma_{n}(s_{\theta}^{X})^{-1}\mathbf{1}_{n}) \\
    &-2n^{-1} \bigl\{ \partial_{k}\mu_{n}(\xi) \partial_{i,j}^{2}\mu_{n}(\xi) + (\mu_{n}(\xi)-\mu_{0}) \partial_{i,j,k}^{3}\mu_{n}(\xi) \bigr\} \mathbf{1}_{n}^{\top}\Sigma_{n}(s_{\theta}^{X})^{-1}\mathbf{1}_{n} \\
    &-2n^{-1} (\mu_{n}(\xi)-\mu_{0}) \partial_{i,j}^{2}\mu_{n}(\xi) (\mathbf{1}_{n}^{\top}\partial_{k}\Sigma_{n}(s_{\theta}^{X})^{-1}\mathbf{1}_{n}) \\
    &-2n^{-1} \left\{ \partial_{k}\mu_{n}(\xi)\partial_{i}\mu_{n}(\xi) + (\mu_{n}(\xi)-\mu_{0}) \partial_{i,k}^{2}\mu_{n}(\xi) \right\} (\mathbf{1}_{n}^{\top}\partial_{j}\Sigma_{n}(s_{\theta}^{X})^{-1}\mathbf{1}_{n}) \\
    & -2n^{-1}(\mu_{n}(\xi)-\mu_{0}) \partial_{i}\mu_{n}(\xi) (\mathbf{1}_{n}^{\top}\partial_{j,k}^{2}\Sigma_{n}(s_{\theta}^{X})^{-1}\mathbf{1}_{n}) \\
    &-2n^{-1} \left\{ \partial_{k}\mu_{n}(\xi)\partial_{j}\mu_{n}(\xi) + (\mu_{n}(\xi)-\mu_{0}) \partial_{j,k}^{2}\mu_{n}(\xi) \right\} (\mathbf{1}_{n}^{\top}\partial_{i}\Sigma_{n}(s_{\theta}^{X})^{-1}\mathbf{1}_{n}) \\
    & -2n^{-1}(\mu_{n}(\xi)-\mu_{0}) \partial_{j}\mu_{n}(\xi) (\mathbf{1}_{n}^{\top}\partial_{i,k}^{2}\Sigma_{n}(s_{\theta}^{X})^{-1}\mathbf{1}_{n}) \\
    &-2n^{-1}(\mu_{n}(\xi)-\mu_{0}) \partial_{k}\mu_{n}(\xi) (\mathbf{1}_{n}^{\top}\partial_{i,j}^{2}\Sigma_{n}(s_{\theta}^{X})^{-1}\mathbf{1}_{n})
    -n^{-1}(\mu_{n}(\xi)-\mu_{0})^{2} \mathbf{1}_{n}^{\top}\partial_{i,j,k}^{3}\Sigma_{n}(s_{\theta}^{X})^{-1}\mathbf{1}_{n}.
\end{align*}
Similar calculations to \eqref{error-mu-estfunc-xi-ineq1} using the Sobolev inequality in Lemma~\ref{Lemma:Sobolev} and the Fubini theorem yield that for each $q>p-1$ and $\iota\in(0,1)$, there exists a positive constant $C_{q}$ such that for any $i,j\in\{1,2,\cdots,p-1\}$,
\begin{align}\label{Sobolev-en1-dv2}
     \E_{\vartheta_{0}}^{n}\left[ \sup_{\xi\in\Theta_{\xi}(\iota)}|\partial_{i,j}^{2}e_{n,1}(\xi)|^{q} \right] 
     \leq& C_{q}%(\Theta) 
     \sup_{\xi\in\Theta_{\xi}(\iota)}\left( \E_{\vartheta_{0}}^{n}\left[|\partial_{i,j}^{2}e_{n,1}(\xi)|^{q}\right] 
     + \sum_{k=1}^{p}\E_{\vartheta_{0}}^{n}\left[|\partial_{i,j,k}^{3}e_{n,1}(\xi)|^{q}\right] \right).
\end{align}
Using the above expressions of $\partial_{i,j}^{2}e_{n,1}(\xi)$ and $\partial_{i,j,k}^{3}e_{n,1}(\xi)$, Lemmas~\ref{Lemma:Key-Ineq-mu} and \ref{Lemma:moment-dv-mu-xi} and the Cauchy-Schwarz inequality, the RHS of \eqref{Sobolev-en1-dv2} goes to zero as $n\to\infty$ so that we conclude $e_{n,1}(\xi)$ is negligible uniformly in $\xi\in\Theta_{\xi}(\iota)$ as $n\to\infty$.\\

Next, we evaluate the second term $e_{n,2}(\xi)$. Since we have the equality \\
$\E_{\vartheta_{0}}^{n}[\tilde{\sigma}^{2}_{n}(\xi)] = \sigma^{2}_{0}n^{-1}\mathrm{Tr}[\Sigma_{n}(s_{\xi_{0}}^{X})\Sigma_{n}(s_{\xi}^{X})^{-1}]$, we can show
\begin{align*}
    \partial_{i}\E_{\vartheta_{0}}^{n}[\tilde{\sigma}^{2}_{n}(\xi)] 
    =& -\frac{\sigma^{2}_{0}}{n} \mathrm{Tr}[\Sigma_{n}(s_{\xi_{0}}^{X})\Sigma_{n}(s_{\xi}^{X})^{-1}\Sigma_{n}(\partial_{i}s_{\xi}^{X})\Sigma_{n}(s_{\xi}^{X})^{-1}], \\
    \partial_{i,j}^{2}\E_{\vartheta_{0}}^{n}[\tilde{\sigma}^{2}_{n}(\xi)] 
    =& \frac{2\sigma^{2}_{0}}{n} \mathrm{Tr}[\Sigma_{n}(s_{\xi_{0}}^{X})\Sigma_{n}(s_{\xi}^{X})^{-1}\Sigma_{n}(\partial_{j}s_{\xi}^{X})\Sigma_{n}(s_{\xi}^{X})^{-1}\Sigma_{n}(\partial_{i}s_{\xi}^{X})\Sigma_{n}(s_{\xi}^{X})^{-1}] \\
    & -\frac{\sigma^{2}_{0}}{n} \mathrm{Tr}[\Sigma_{n}(s_{\xi_{0}}^{X})\Sigma_{n}(s_{\xi}^{X})^{-1}\Sigma_{n}(\partial_{i,j}^{2}s_{\xi}^{X})\Sigma_{n}(s_{\xi}^{X})^{-1}], 
\end{align*}
where we used the facts that $\mathrm{Tr}[A]=\mathrm{Tr}[A^{\top}]$ and $\mathrm{Tr}[AB]=\mathrm{Tr}[BA]$ hold for any square matrices $A$ and $B$ in the last inequality, so that, using the expressions of $\partial_{i}\sigma^{2}(\xi)$ and $\partial_{i,j}^{2}\sigma^{2}(\xi)$ in Lemma~\ref{Lemma:unifconv-sig2-dv1-dv2} and Lemma~\ref{Lemma:trace-app-inv}, we conclude $\partial_{i}e_{n,2}(\xi)$ and $\partial_{i,j}^{2}e_{n,2}(\xi)$ vanish uniformly on $\Theta_{\xi}(\iota)$ as $n\to\infty$.\\

Finally, we consider the third term $e_{n,3}(\xi)$. 
The Sobolev inequality in Lemma~\ref{Lemma:Sobolev} yields that for any $\iota\in(0,1)$ and $2q>p-1$, there exists a positive constant $C_{2q}$ such that for any $i,j,k=1,\ldots,p-1$, 
\begin{align}
    &\E_{\vartheta_{0}}^{n}\left[ \sup_{\xi\in\Theta_{\xi}(\iota)}|\partial_{i}e_{n,3}(\xi)|^{2q} \right] 
    \leq C_{2q}
    \sup_{\xi\in\Theta_{\xi}(\iota)}\left( \E_{\vartheta_{0}}^{n}\left[|\partial_{i}e_{n,3}(\xi)|^{2q}\right] 
    + \sum_{j=1}^{p-1}\E_{\vartheta_{0}}^{n}\bigl[|\partial_{i,j}^{2}e_{n,3}(\xi)|^{2q}\bigr] \right), 
    \label{Sobolev-en3-dv1} \\
    &\E_{\vartheta_{0}}^{n}\left[ \sup_{\xi\in\Theta_{\xi}(\iota)}|\partial_{i,j}^{2}e_{n,3}(\xi)|^{2q} \right] 
    \leq C_{2q}
    \sup_{\xi\in\Theta_{\xi}(\iota)}\left( \E_{\vartheta_{0}}^{n}\left[|\partial_{i,j}^{2}e_{n,3}(\xi)|^{2q}\right] 
    + \sum_{j=1}^{p-1}\E_{\vartheta_{0}}^{n}\bigl[|\partial_{i,j,k}^{3}e_{n,3}(\xi)|^{2q}\bigr] \right). 
    \label{Sobolev-en3-dv2} 
\end{align}
Notice that we can show
\begin{align*}
        &\partial_{i}\tilde{\sigma}^{2}_{n}(\xi) 
        = -\frac{1}{n}\left(\mathbf{X}_{n}-\mu_{0}\mathbf{1}_{n}\right)^{\top} 
        \Sigma_{n}(s_{\xi}^{X})^{-1} \Sigma_{n}(\partial_{i}s_{\xi}^{X})\Sigma_{n}(s_{\xi}^{X})^{-1} 
        \left(\mathbf{X}_{n}-\mu_{0}\mathbf{1}_{n}\right) \\
        &\partial_{i,j}^{2}\tilde{\sigma}^{2}_{n}(\xi) 
        = \frac{2}{n}\left(\mathbf{X}_{n}-\mu_{0}\mathbf{1}_{n}\right)^{\top} 
        \Sigma_{n}(s_{\xi}^{X})^{-1}\Sigma_{n}(\partial_{i}s_{\xi}^{X})\Sigma_{n}(s_{\xi}^{X})^{-1}\Sigma_{n}(\partial_{j}s_{\xi}^{X})\Sigma_{n}(s_{\xi}^{X})^{-1} 
        \left(\mathbf{X}_{n}-\mu_{0}\mathbf{1}_{n}\right) \\
        &\hspace{2cm} -\frac{1}{n}\left(\mathbf{X}_{n}-\mu_{0}\mathbf{1}_{n}\right)^{\top} 
        \Sigma_{n}(s_{\xi}^{X})^{-1} \Sigma_{n}(\partial_{i,j}^{2}s_{\xi}^{X})\Sigma_{n}(s_{\xi}^{X})^{-1} 
        \left(\mathbf{X}_{n}-\mu_{0}\mathbf{1}_{n}\right), \\
        &\partial_{i,j,k}^{3}\tilde{\sigma}^{2}_{n}(\xi) 
        = -\frac{2}{n}\left(\mathbf{X}_{n}-\mu_{0}\mathbf{1}_{n}\right)^{\top} 
        \Sigma_{n}(s_{\xi}^{X})^{-1}\Sigma_{n}(\partial_{k}s_{\xi}^{X})\Sigma_{n}(s_{\xi}^{X})^{-1}\Sigma_{n}(\partial_{i}s_{\xi}^{X})\Sigma_{n}(s_{\xi}^{X})^{-1}\Sigma_{n}(\partial_{j}s_{\xi}^{X})\Sigma_{n}(s_{\xi}^{X})^{-1} 
        \left(\mathbf{X}_{n}-\mu_{0}\mathbf{1}_{n}\right) \\
        &\hspace{2cm} + \frac{2}{n}\left(\mathbf{X}_{n}-\mu_{0}\mathbf{1}_{n}\right)^{\top} 
        \Sigma_{n}(s_{\xi}^{X})^{-1}\Sigma_{n}(\partial_{i,k}^{2}s_{\xi}^{X})\Sigma_{n}(s_{\xi}^{X})^{-1}\Sigma_{n}(\partial_{j}s_{\xi}^{X})\Sigma_{n}(s_{\xi}^{X})^{-1} \left(\mathbf{X}_{n}-\mu_{0}\mathbf{1}_{n}\right) \\
        &\hspace{2cm} -\frac{2}{n}\left(\mathbf{X}_{n}-\mu_{0}\mathbf{1}_{n}\right)^{\top} 
        \Sigma_{n}(s_{\xi}^{X})^{-1}\Sigma_{n}(\partial_{i}s_{\xi}^{X})\Sigma_{n}(s_{\xi}^{X})^{-1}\Sigma_{n}(\partial_{k}s_{\xi}^{X})\Sigma_{n}(s_{\xi}^{X})^{-1}\Sigma_{n}(\partial_{j}s_{\xi}^{X})\Sigma_{n}(s_{\xi}^{X})^{-1} 
        \left(\mathbf{X}_{n}-\mu_{0}\mathbf{1}_{n}\right) \\
        &\hspace{2cm} + \frac{2}{n}\left(\mathbf{X}_{n}-\mu_{0}\mathbf{1}_{n}\right)^{\top} 
        \Sigma_{n}(s_{\xi}^{X})^{-1}\Sigma_{n}(\partial_{i}s_{\xi}^{X})\Sigma_{n}(s_{\xi}^{X})^{-1}\Sigma_{n}(\partial_{j,k}^{2}s_{\xi}^{X})\Sigma_{n}(s_{\xi}^{X})^{-1} \left(\mathbf{X}_{n}-\mu_{0}\mathbf{1}_{n}\right) \\
        &\hspace{2cm} -\frac{2}{n}\left(\mathbf{X}_{n}-\mu_{0}\mathbf{1}_{n}\right)^{\top} 
        \Sigma_{n}(s_{\xi}^{X})^{-1}\Sigma_{n}(\partial_{i}s_{\xi}^{X})\Sigma_{n}(s_{\xi}^{X})^{-1}\Sigma_{n}(\partial_{j}s_{\xi}^{X})\Sigma_{n}(s_{\xi}^{X})^{-1}\Sigma_{n}(\partial_{k}s_{\xi}^{X})\Sigma_{n}(s_{\xi}^{X})^{-1} 
        \left(\mathbf{X}_{n}-\mu_{0}\mathbf{1}_{n}\right) \\
        &\hspace{2cm} +\frac{1}{n}\left(\mathbf{X}_{n}-\mu_{0}\mathbf{1}_{n}\right)^{\top} 
        \Sigma_{n}(s_{\xi}^{X})^{-1} \Sigma_{n}(\partial_{k}s_{\xi}^{X})\Sigma_{n}(s_{\xi}^{X})^{-1} \Sigma_{n}(\partial_{i,j}^{2}s_{\xi}^{X})\Sigma_{n}(s_{\xi}^{X})^{-1} \left(\mathbf{X}_{n}-\mu_{0}\mathbf{1}_{n}\right) \\
        &\hspace{2cm} -\frac{1}{n}\left(\mathbf{X}_{n}-\mu_{0}\mathbf{1}_{n}\right)^{\top} 
        \Sigma_{n}(s_{\xi}^{X})^{-1} \Sigma_{n}(\partial_{i,j,k}^{3}s_{\xi}^{X})\Sigma_{n}(s_{\xi}^{X})^{-1} 
        \left(\mathbf{X}_{n}-\mu_{0}\mathbf{1}_{n}\right) \\
        &\hspace{2cm} +\frac{1}{n}\left(\mathbf{X}_{n}-\mu_{0}\mathbf{1}_{n}\right)^{\top} 
        \Sigma_{n}(s_{\xi}^{X})^{-1} \Sigma_{n}(\partial_{i,j}^{2}s_{\xi}^{X})\Sigma_{n}(s_{\xi}^{X})^{-1} \Sigma_{n}(\partial_{k}s_{\xi}^{X}) \Sigma_{n}(s_{\xi}^{X})^{-1} \left(\mathbf{X}_{n}-\mu_{0}\mathbf{1}_{n}\right),
\end{align*}
which imply
\begin{align*}
    \mbox{$\partial_{i}\E_{\vartheta_{0}}^{n}[\tilde{\sigma}^{2}_{n}(\xi)]=\E_{\vartheta_{0}}^{n}[\partial_{i}\tilde{\sigma}^{2}_{n}(\xi)]$, $\partial_{i,j}^{2}\E_{\vartheta_{0}}^{n}[\tilde{\sigma}^{2}_{n}(\xi)]=\E_{\vartheta_{0}}^{n}[\partial_{i,j}^{2}\tilde{\sigma}^{2}_{n}(\xi)]$ and $\partial_{i,j,k}^{3}\E_{\vartheta_{0}}^{n}[\tilde{\sigma}^{2}_{n}(\xi)]=\E_{\vartheta_{0}}^{n}[\partial_{i,j,k}^{3}\tilde{\sigma}^{2}_{n}(\xi)]$}.
\end{align*}
Then we can show that the quantities in the RHS of the inequalities \eqref{Sobolev-en3-dv1} and \eqref{Sobolev-en3-dv2} vanish as $n\to\infty$ using Lemma~\ref{Lemma:trace-app-inv} similarly to the proof of \eqref{Sobolev-en3} in the proof of \eqref{unifconv-sig2-bar}. 
Therefore, we finish the proof of Lemma~\ref{Lemma:unifconv-sig2-dv1-dv2}.
\subsubsection{Proof of \eqref{unifconv-sig2-bar}}\label{Sec:proof-unifconv-sig2-bar}
Fix $\iota\in(0,1)$. Recall that
    \begin{align*}
     \tilde{\sigma}^{2}_{n}(\xi) := \sigma_{n}^{2}((\xi,\mu_{0})^{\top})
     = \frac{1}{n}\left(\mathbf{X}_{n}-\mu_{0}\mathbf{1}_{n}\right)^{\top}\Sigma_{n}(s_{\xi}^{X})^{-1}\left(\mathbf{X}_{n}-\mu_{0}\mathbf{1}_{n}\right).
    \end{align*}
    We decompose the error $\bar{\sigma}^{2}_{n}(\xi)-\sigma^{2}(\xi)$ into the following three terms:
    \begin{align*}
        e_{n,1}(\xi):=\bar{\sigma}^{2}_{n}(\xi)-\tilde{\sigma}^{2}_{n}(\xi),\ \ 
        e_{n,2}(\xi):=\E_{\vartheta_{0}}^{n}[\tilde{\sigma}^{2}_{n}(\xi)]-\sigma^{2}(\xi),\ \ 
        \mbox{and}\ \ e_{n,3}(\xi):=\tilde{\sigma}^{2}_{n}(\xi)-\E_{\vartheta_{0}}^{n}[\tilde{\sigma}^{2}_{n}(\xi)].
    \end{align*}
    In the rest of the proof, we will show that the error terms $e_{n,i}(\xi)$, $i=1,2,3$, vanish uniformly on $\Theta_{\xi}(\iota)$ as $n\to\infty$.\\
    
    We first consider the first term $e_{n,1}(\xi)$. 
    Notice that the error term $e_{n,1}(\xi)$ is written as
\begin{align}
    e_{n,1}(\xi) =& -2n^{-1}(\mu_{n}(\xi)-\mu_{0}) \mathbf{1}_{n}^{\top}\Sigma_{n}(s_{\xi}^{X})^{-1}(\mathbf{X}_{n}-\mu_{0}\mathbf{1}_{n}) 
    +n^{-1}(\mu_{n}(\xi)-\mu_{0})^{2} \mathbf{1}_{n}^{\top}\Sigma_{n}(s_{\xi}^{X})^{-1}\mathbf{1}_{n} 
    \nonumber \\
    =& -n^{-1}(\mu_{n}(\xi)-\mu_{0})^{2} \mathbf{1}_{n}^{\top}\Sigma_{n}(s_{\xi}^{X})^{-1}\mathbf{1}_{n}
    \label{expr:en1}
\end{align}
so that, using the chain rule, its first-order derivatives with respect to $\xi$ are expressed by
\begin{align}\label{expr:en1-dv1}
    \partial_{\xi}e_{n,1}(\xi)
    = -2n^{-1}(\mu_{n}(\xi)-\mu_{0}) \partial_{\xi}\mu_{n}(\xi) \mathbf{1}_{n}^{\top}\Sigma_{n}(s_{\xi}^{X})^{-1}\mathbf{1}_{n} 
    -n^{-1}(\mu_{n}(\xi)-\mu_{0})^{2} \mathbf{1}_{n}^{\top}\partial_{\xi}\Sigma_{n}(s_{\xi}^{X})^{-1}\mathbf{1}_{n}.
\end{align}
Moreover, the Sobolev inequality in Lemma~\ref{Lemma:Sobolev} yields that for any $\iota\in(0,1)$ and $2q>p-1$, there exists a positive constant $C_{2q}$ such that 
\begin{align}\label{Sobolev-en1}
    \E_{\vartheta_{0}}^{n}\left[ \sup_{\xi\in\Theta_{\xi}(\iota)}|e_{n,1}(\xi)|^{2q} \right] 
    \leq C_{2q}
    \sup_{\xi\in\Theta_{\xi}(\iota)}\left( \E_{\vartheta_{0}}^{n}\left[|e_{n,1}(\xi)|^{2q}\right] 
    + \sum_{j=1}^{p-1}\E_{\vartheta_{0}}^{n}\left[|\partial_{j}e_{n,1}(\xi)|^{2q}\right] \right).
\end{align}
Then, using the above expressions of $e_{n,1}(\xi)$ and $\partial_{\xi}e_{n,1}(\xi)$, Lemmas~\ref{Lemma:Key-Ineq-mu} and \ref{Lemma:moment-dv-mu-xi} and the Cauchy-Schwarz inequality, we can show that the quantity in the RHS of the inequality \eqref{Sobolev-en1} converges to zero as $n\to\infty$ so that we conclude $e_{n,1}(\xi)$ vanishes uniformly on $\Theta_{\xi}(\iota)$ as $n\to\infty$.\\

Next, we consider the second term $e_{n,2}(\xi)$. 
Notice that we have the equality $\E_{\vartheta_{0}}^{n}[\tilde{\sigma}^{2}_{n}(\xi)] = \sigma^{2}_{0}n^{-1}\mathrm{Tr}[\Sigma_{n}(s_{\xi_{0}}^{X})\Sigma_{n}(s_{\xi}^{X})^{-1}]$ so that, using Lemma~\ref{Lemma:trace-app-inv}, we conclude $e_{n,2}(\xi)$ vanishes uniformly on $\Theta_{\xi}(\iota)$ as $n\to\infty$.\\ 

Finally, we consider the third term $e_{n,3}(\xi)$. 
The Sobolev inequality in Lemma~\ref{Lemma:Sobolev} yields that for any $\iota\in(0,1)$ and $2q>p-1$, there exists a positive constant $C_{2q}$ such that 
\begin{align}\label{Sobolev-en3}
    \E_{\vartheta_{0}}^{n}\left[ \sup_{\xi\in\Theta_{\xi}(\iota)}|e_{n,3}(\xi)|^{2q} \right] 
    \leq C_{2q}
    \sup_{\xi\in\Theta_{\xi}(\iota)}\left( \E_{\vartheta_{0}}^{n}\left[|e_{n,3}(\xi)|^{2q}\right] 
    + \sum_{j=1}^{p-1}\E_{\vartheta_{0}}^{n}\bigl[|\partial_{j}e_{n,3}(\xi)|^{2q}\bigr] \right).
\end{align}
Then we can show that the quantity in the RHS of the inequality \eqref{Sobolev-en3} vanishes as $n\to\infty$ using Lemma~\ref{Lemma:trace-app-inv} because we know that 1)~the moments $\E_{\vartheta_{0}}^{n}\left[|e_{n,3}(\xi)|^{2q}\right]$ and $\E_{\vartheta_{0}}^{n}[|\partial_{j}e_{n,3}(\xi)|^{2q}] $ can be expressed by linear combinations of cumulants up to the order $2q$ using the Leonov-Shiryaev formula, 2)~$e_{n,3}(\xi)$ and $\partial_{j}e_{n,3}(\xi)$ are centralized quadratic forms of Gaussian vector so that for each $r\geq 2$, its $r$ th order cumulants can be expressed by $\mathrm{cum}_{r}[e_{n,3}(\xi)]=n^{-r}c_{r}\mathrm{Tr}[(\Sigma_{n}(s_{\xi}^{X})^{-1}\Sigma_{n}(s_{\xi_0}^{X}))^r]$ and $\mathrm{cum}_{r}[\partial_{j}e_{n,3}(\xi)]=n^{-r}c_{r}\mathrm{Tr}[(\Sigma_{n}(s_{\xi}^{X})^{-1}\Sigma_{n}(\partial_\xi s_{\xi}^{X})\Sigma_{n}(s_{\xi}^{X})^{-1}\Sigma_{n}(s_{\xi_0}^{X}))^r]$ for some positive constant $c_r$, and 3)~the first order cumulants of $e_{n,3}(\xi)$ and $\partial_{j}e_{n,3}(\xi)$ are equal to zero. 
Therefore, we conclude $e_{n,3}(\xi)$ also vanishes uniformly on $\Theta_{\xi}(\iota)$ as $n\to\infty$.
This completes the proof of \eqref{unifconv-sig2-bar}. 

\subsection{Proof of \eqref{unifconv-sig2-bar}}\label{Sec:proof-unifconv-sig2-bar}
Fix $\iota\in(0,1)$. Recall that
    \begin{align*}
     \tilde{\sigma}^{2}_{n}(\xi) := \sigma_{n}^{2}((\xi,\mu_{0})^{\top})
     = \frac{1}{n}\left(\mathbf{X}_{n}-\mu_{0}\mathbf{1}_{n}\right)^{\top}\Sigma_{n}(s_{\xi}^{X})^{-1}\left(\mathbf{X}_{n}-\mu_{0}\mathbf{1}_{n}\right).
    \end{align*}
    We decompose the error $\bar{\sigma}^{2}_{n}(\xi)-\sigma^{2}(\xi)$ into the following three terms:
    \begin{align*}
        e_{n,1}(\xi):=\bar{\sigma}^{2}_{n}(\xi)-\tilde{\sigma}^{2}_{n}(\xi),\ \ 
        e_{n,2}(\xi):=\E_{\vartheta_{0}}^{n}[\tilde{\sigma}^{2}_{n}(\xi)]-\sigma^{2}(\xi),\ \ 
        \mbox{and}\ \ e_{n,3}(\xi):=\tilde{\sigma}^{2}_{n}(\xi)-\E_{\vartheta_{0}}^{n}[\tilde{\sigma}^{2}_{n}(\xi)].
    \end{align*}
    In the rest of the proof, we will show that the error terms $e_{n,i}(\xi)$, $i=1,2,3$, vanish uniformly on $\Theta_{\xi}(\iota)$ as $n\to\infty$.\\
    
    We first consider the first term $e_{n,1}(\xi)$. 
    Note that the error term $e_{n,1}(\xi)$ is written as
\begin{align}
    e_{n,1}(\xi) =& -2n^{-1}(\mu_{n}(\xi)-\mu_{0}) \mathbf{1}_{n}^{\top}\Sigma_{n}(s_{\xi}^{X})^{-1}(\mathbf{X}_{n}-\mu_{0}\mathbf{1}_{n}) 
    +n^{-1}(\mu_{n}(\xi)-\mu_{0})^{2} \mathbf{1}_{n}^{\top}\Sigma_{n}(s_{\xi}^{X})^{-1}\mathbf{1}_{n} 
    \nonumber \\
    =& -n^{-1}(\mu_{n}(\xi)-\mu_{0})^{2} \mathbf{1}_{n}^{\top}\Sigma_{n}(s_{\xi}^{X})^{-1}\mathbf{1}_{n},
    \label{expr:en1}
\end{align}
so that, by the chain rule, its first order derivatives with respect to $\xi$ is expressed by
\begin{align}\label{expr:en1-dv1}
    \partial_{\xi}e_{n,1}(\xi)
    = -2n^{-1}(\mu_{n}(\xi)-\mu_{0}) \partial_{\xi}\mu_{n}(\xi) \mathbf{1}_{n}^{\top}\Sigma_{n}(s_{\xi}^{X})^{-1}\mathbf{1}_{n} 
    -n^{-1}(\mu_{n}(\xi)-\mu_{0})^{2} \mathbf{1}_{n}^{\top}\partial_{\xi}\Sigma_{n}(s_{\xi}^{X})^{-1}\mathbf{1}_{n}.
\end{align}
Moreover, the Sobolev inequality in Lemma~\ref{Lemma:Sobolev} yields that for any $\iota\in(0,1)$ and $2q>p-1$, there exists a positive constant $C_{2q}$ such that 
\begin{align}\label{Sobolev-en1}
    \E_{\vartheta_{0}}^{n}\left[ \sup_{\xi\in\Theta_{\xi}(\iota)}|e_{n,1}(\xi)|^{2q} \right] 
    \leq C_{2q}
    \sup_{\xi\in\Theta_{\xi}(\iota)}\left( \E_{\vartheta_{0}}^{n}\left[|e_{n,1}(\xi)|^{2q}\right] 
    + \sum_{j=1}^{p-1}\E_{\vartheta_{0}}^{n}\left[|\partial_{j}e_{n,1}(\xi)|^{2q}\right] \right).
\end{align}
Then, using the above expressions of $e_{n,1}(\xi)$ and $\partial_{\xi}e_{n,1}(\xi)$, Lemmas~\ref{Lemma:Key-Ineq-mu} and \ref{Lemma:moment-dv-mu-xi} and the Cauchy-Schwarz inequality, we can show that the quantity in the RHS of the inequality \eqref{Sobolev-en1} converges to zero as $n\to\infty$. Hence, we conclude $e_{n,1}(\xi)$ vanishes uniformly on $\Theta_{\xi}(\iota)$ as $n\to\infty$.\\

Next, we consider the second term $e_{n,2}(\xi)$. 
Notice that we have the equality $\E_{\vartheta_{0}}^{n}[\tilde{\sigma}^{2}_{n}(\xi)] = \sigma^{2}_{0}n^{-1}\mathrm{Tr}[\Sigma_{n}(s_{\xi_{0}}^{X})\Sigma_{n}(s_{\xi}^{X})^{-1}]$ so that, using Lemma~\ref{Lemma:trace-app-inv}, we conclude $e_{n,2}(\xi)$ vanishes uniformly on $\Theta_{\xi}(\iota)$ as $n\to\infty$.\\ 

Finally, we consider the third term $e_{n,3}(\xi)$. 
The Sobolev inequality in Lemma~\ref{Lemma:Sobolev} yields that for any $\iota\in(0,1)$ and $2q>p-1$, there exists a positive constant $C_{2q}$ such that 
\begin{align}\label{Sobolev-en3}
    \E_{\vartheta_{0}}^{n}\left[ \sup_{\xi\in\Theta_{\xi}(\iota)}|e_{n,3}(\xi)|^{2q} \right] 
    \leq C_{2q}
    \sup_{\xi\in\Theta_{\xi}(\iota)}\left( \E_{\vartheta_{0}}^{n}\left[|e_{n,3}(\xi)|^{2q}\right] 
    + \sum_{j=1}^{p-1}\E_{\vartheta_{0}}^{n}\bigl[|\partial_{j}e_{n,3}(\xi)|^{2q}\bigr] \right).
\end{align}
Then we can show that the quantity in the RHS of the inequality \eqref{Sobolev-en3} vanishes as $n\to\infty$ using Lemma~\ref{Lemma:trace-app-inv} because we know that 1)~the moments $\E_{\vartheta_{0}}^{n}\left[|e_{n,3}(\xi)|^{2q}\right]$ and $\E_{\vartheta_{0}}^{n}[|\partial_{j}e_{n,3}(\xi)|^{2q}] $ can be expressed by linear combinations of cumulants up to the order $2q$ using the Leonov-Shiryaev formula, 2)~$e_{n,3}(\xi)$ and $\partial_{j}e_{n,3}(\xi)$ are centralized quadratic forms of Gaussian vector so that for each $r\geq 2$, its $r$th order cumulants can be expressed by $\mathrm{cum}_{r}[e_{n,3}(\xi)]=n^{-r}c_{r}\mathrm{Tr}[(\Sigma_{n}(s_{\xi}^{X})^{-1}\Sigma_{n}(s_{\xi_0}^{X}))^r]$ and $\mathrm{cum}_{r}[\partial_{j}e_{n,3}(\xi)]=n^{-r}c_{r}\mathrm{Tr}[(\Sigma_{n}(s_{\xi}^{X})^{-1}\Sigma_{n}(\partial_\xi s_{\xi}^{X})\Sigma_{n}(s_{\xi}^{X})^{-1}\Sigma_{n}(s_{\xi_0}^{X}))^r]$ for some positive constant $c_r$, and 3)~the first order cumulants of $e_{n,3}(\xi)$ and $\partial_{j}e_{n,3}(\xi)$ are equal to zero. 
Therefore, we conclude $e_{n,3}(\xi)$ also vanishes uniformly on $\Theta_{\xi}(\iota)$ as $n\to\infty$.
This completes the proof of \eqref{unifconv-sig2-bar}. 

\subsubsection{Proof of \eqref{unifconv-log-sig2-bar}} \label{Sec:proof-unifconv-log-sig2-bar}
Using the Taylor theorem, we can write
\begin{align*}
    \log{\bar{\sigma}^{2}_{n}(\xi)} - \log{\sigma^{2}(\xi)}
    =  ( \bar{\sigma}^{2}_{n}(\xi)-\sigma^{2}(\xi) )
    \int_{0}^{1}\Bigl( \sigma^{2}(\xi) + u(\bar{\sigma}^{2}_{n}(\xi)-\sigma^{2}(\xi)) \Bigr)^{-1}\dd u
\end{align*}
so that we obtain
\begin{align*}
    |\log{\bar{\sigma}^{2}_{n}(\xi)} - \log{\sigma^{2}(\xi)}| \leq |\bar{\sigma}^{2}_{n}(\xi)-\sigma^{2}(\xi)| \Bigl| \sigma^{2}(\xi) - |\bar{\sigma}^{2}_{n}(\xi)-\sigma^{2}(\xi)| \Bigr|^{-1}. 
\end{align*}
Set $\sigma^{2}(\xi_{\ast}(\iota)):=\inf_{\xi\in\Theta_{\xi}(\iota)} \sigma^{2}(\xi)>0$ so that we can take $\epsilon_{1}\in(0,\sigma^{2}(\xi_{\ast}(\iota)))$.  
Then, for any $\epsilon>0$, we can show the inequality
\begin{align*}
    \P_{\vartheta_{0}}^{n}\left[ \sup_{\xi\in\Theta_{\xi}(\iota)}|\log{\bar{\sigma}^{2}_{n}(\xi)} - \log{\sigma^{2}(\xi)}| > \epsilon \right]
    \leq& \P_{\vartheta_{0}}^{n}\left[ \sup_{\xi\in\Theta_{\xi}(\iota)}|\bar{\sigma}^{2}_{n}(\xi)-\sigma^{2}(\xi)| > \epsilon_{1} \right] \\
    &+ \P_{\vartheta_{0}}^{n}\left[ \sup_{\xi\in\Theta_{\xi}(\iota)}|\bar{\sigma}^{2}_{n}(\xi)-\sigma^{2}(\xi)| > ( \sigma^{2}(\xi_{\ast}(\iota)) - \epsilon_{1}) \epsilon \right],
\end{align*}
which concludes \eqref{unifconv-log-sig2-bar} using \eqref{unifconv-sig2-bar}.

\begin{proof}[Proof of Proposition~\ref{Thm:MLE-mu}]
    Under Assumption~\ref{Assump:DSPD2}, Theorems~4.1 and 5.2 in \cite{Adenstedt-1974} yield the convergence
    \begin{align*}
        n^{\frac{1}{2}(1-\alpha_{X}(\xi_{0}))} ( \mu_{n}(\xi_{0}) - \mu_{0}) \to \mathcal{N}\left( 0,\frac{2\pi\sigma^2_0c_{X}(\xi_{0})\Gamma(1-\alpha_{X}(\xi_{0}))}{B(1-\alpha_{X}(\xi_{0})/2,1-\alpha_{X}(\xi_{0})/2)} \right) 
    \end{align*}
    in law under the distribution $\P_{\vartheta_{0}}^{n}$ as $n\to\infty$ for any interior point $\vartheta_{0}$ of $\Theta$. 
    Thus it suffices to prove that
    \begin{align}\label{Thm:MLE-mu_Suff1}
        n^{\frac{1}{2}(1-\alpha_{X}(\xi_{0}))} ( \mu_{n}(\widehat{\xi}_{n}) - \mu_{n}(\xi_{0}) ) = o_{\P_{\vartheta_{0}}^{n}}(1)\ \ \mbox{as $n\to\infty$}.
    \end{align}
    Notice that the sequence $\{\sqrt{n}(\hat{\xi}_{n}-\xi_{0})\}_{n\in\N}$ is stochastically bounded from the assumption.
    Then, using the Cauchy-Schwarz inequality and the Chebyshev inequality, for any $\epsilon>0$ and $M>0$, we obtain
    \begin{align*}
        &\P_{\vartheta_{0}}^{n}\left[ \left| n^{\frac{1}{2}(1-\alpha_{X}(\xi_{0}))}( \mu_{n}(\widehat{\xi}_{n}) - \mu_{n}(\xi_{0}) ) \right| \geq \varepsilon \right] \\
        &\leq \P_{\vartheta_{0}}^{n}[\|\overline{\xi}_{n}\|_{\R^{p-1}}\geq M] 
        + \P_{\vartheta_{0}}^{n}\left[ \sup_{\xi\in B_{n^{-1/2}M}(\xi_{0})} \left| n^{\frac{1}{2}(1-\alpha_{X}(\xi_{0}))}( \mu_{n}(\xi) - \mu_{n}(\xi_{0}) ) \right| \geq \varepsilon \right] \\
        &\leq \P_{\vartheta_{0}}^{n}[\|\overline{\xi}_{n}\|_{\R^{p-1}}\geq M] 
        + \varepsilon^{-q}\E_{\vartheta_{0}}^{n}\left[ \sup_{\xi\in B_{n^{-1/2}M}(\xi_{0})} \left| n^{\frac{1}{2}(1-\alpha_{X}(\xi_{0}))}( \mu_{n}(\xi) - \mu_{n}(\xi_{0}) ) \right|^{q} \right] 
    \end{align*}
    so that \eqref{Thm:MLE-mu_Suff1} follows once we have proved that for any $M>0$ and $q>p-1$, 
    \begin{align}\label{Thm:MLE-mu_Suff2}
        \E_{\vartheta_{0}}^{n}\left[ \sup_{\xi\in B_{n^{-1/2}M}(\xi_{0})} \left| n^{\frac{1}{2}(1-\alpha_{X}(\xi_{0}))}( \mu_{n}(\xi) - \mu_{n}(\xi_{0}) ) \right|^{q} \right] = o(1)\ \ \mbox{as $n\to\infty$}.
    \end{align}

    Notice that Lemma~\ref{Lemma:Sobolev} and the Fubini theorem yield that for any $q>p-1$, it holds
\begin{align}
     &\E_{\vartheta_{0}}^{n}\left[ \sup_{\xi\in B_{n^{-1/2}M}(\xi_{0})} \left| n^{\frac{1}{2}(1-\alpha_{X}(\xi_{0}))}( \mu_{n}(\xi) - \mu_{n}(\xi_{0}) ) \right|^{q} \right] \nonumber \\
     &\leq C_{q,1} (n^{-1/2}M)^{q-1} \E_{\vartheta_{0}}^{n}\left[ \left( \sum_{j=1}^{p-1} \left\| n^{\frac{1}{2}(1-\alpha_{X}(\xi_{0}))} \partial_{j}\mu_{n}(\cdot) \right\|_{L^{q}(B_{n^{-1/2}M}(\xi_{0}))} \right)^{q} \right] \nonumber \\
     &\leq C_{q,2} (n^{-1/2}M)^{q} \sum_{j=1}^{p-1} \sup_{\xi\in B_{n^{-1/2}M}(\xi_{0})} \E_{\vartheta_{0}}^{n}\left[\left| n^{\frac{1}{2}(1-\alpha_{X}(\xi_{0}))} \partial_{j}\mu_{n}(\xi) \right|^{q}\right]
     \label{error-mu-estfunc-xi-ineq1}
\end{align}
and Lemma~\ref{Lemma:moment-dv-mu-xi} gives the inequality
\begin{align}
    &n^{-\frac{q}{2}} \sum_{j=1}^{p-1} \sup_{\xi\in B_{n^{-1/2}M}(\xi_{0})} \E_{\vartheta_{0}}^{n}\left[| n^{\frac{1}{2}(1-\alpha_{X}(\xi_{0}))} \partial_{j}\mu_{n}(\xi) |^{q}\right] \nonumber \\
    &\leq C_{q,3} n^{-\frac{q}{2}} \sum_{j=1}^{p-1} \sup_{\xi\in B_{n^{-1/2}M}(\xi_{0})} 
    n^{\frac{q}{2}(1-\alpha_{X}(\xi_{0}))} (\mathbf{1}_{n}^{\top}\Sigma_{n}(s_{\xi}^{X})^{-1}\mathbf{1}_{n})^{-\frac{q}{2}} n^{\frac{q}{2}(\alpha_{X}(\xi_{0})-\alpha_{X}(\xi))_{+}+\varepsilon}.
    \label{Thm:MLE-mu_ineq1}
\end{align}
Since we can show the inequality 
\begin{align*}
    \mathbf{1}_{n}^{\top}\Sigma_{n}(s_{\xi_{0}}^{X})^{-1}\mathbf{1}_{n}
    =& \left\| \Sigma_{n}(s_{\xi_{0}}^{X})^{-\frac{1}{2}} \Sigma_{n}(s_{\xi}^{X})^{\frac{1}{2}} \Sigma_{n}(s_{\xi}^{X})^{-\frac{1}{2}}\mathbf{1}_{n} \right\|_{\R^{n}}^{2} \\
    \leq& \left\| \Sigma_{n}(s_{\xi_{0}}^{X})^{-\frac{1}{2}} \Sigma_{n}(s_{\xi}^{X})^{\frac{1}{2}} \right\|_{\mathrm{op}}^{2} 
    \left\| \Sigma_{n}(s_{\xi}^{X})^{-\frac{1}{2}}\mathbf{1}_{n} \right\|_{\R^{n}}^{2},
\end{align*}
which follows from the definition of the operator norm, we can further evaluate the last quantity in the inequality \eqref{Thm:MLE-mu_ineq1} up to a constant multiplication by
\begin{align*}
    &n^{-\frac{q}{2}} \sum_{j=1}^{p-1} \sup_{\xi\in B_{n^{-1/2}M}(\xi_{0})} 
    n^{\frac{q}{2}(1-\alpha_{X}(\xi_{0}))} (\mathbf{1}_{n}^{\top}\Sigma_{n}(s_{\xi}^{X})^{-1}\mathbf{1}_{n})^{-\frac{q}{2}} n^{\frac{q}{2}(\alpha_{X}(\xi_{0})-\alpha_{X}(\xi))_{+}+\varepsilon}  \\
    &\leq n^{-\frac{q}{2}} \sum_{j=1}^{p-1} \sup_{\xi\in B_{n^{-1/2}M}(\xi_{0})} 
    n^{\frac{q}{2}(1-\alpha_{X}(\xi_{0}))} (\mathbf{1}_{n}^{\top}\Sigma_{n}(s_{\xi_{0}}^{X})^{-1}\mathbf{1}_{n})^{-\frac{q}{2}} n^{\frac{3}{2}q(\alpha_{X}(\xi_{0})-\alpha_{X}(\xi))_{+}+\varepsilon},
\end{align*}
where we also used Assumption~\ref{Assump:DSPD1}, Lemma~5.3 in \cite{Dahlhaus-1989} and Lemma~6 in the full version of \cite{Lieberman-2012} in the last inequality. 
Then we conclude \eqref{Thm:MLE-mu_Suff2} using Theorems~4.1 and 5.2 in \cite{Adenstedt-1974}. 
Therefore, the proof is complete. 
\end{proof}  

\subsection{Proof of \eqref{key:LAN:ratio-rate-mat}}
Notice that we have $R_{n}(\vartheta,\vartheta^\prime) = \diag(I_{p},n^{-\frac{1}{2}(\alpha_X(\xi+\xi^\prime)-\alpha_X(\xi)})$
for any $\vartheta^\prime=(\xi^{\prime},\sigma^{\prime},\mu^{\prime})^{\top}$, and 
\begin{align*}
    \|\Phi_{n}(\vartheta)\|_{\mathrm{op}}=\max\{ n^{-\frac{1}{2}},n^{-\frac{1}{2}(1-\alpha_X(\xi))} \} = n^{-\frac{1}{2}\min\{1,1-\alpha_{X}(\xi)\}},
\end{align*}
son that we conclude \eqref{key:LAN:ratio-rate-mat} using the continuously differentiability of the function $\alpha_X(\xi)$.

\subsection{Proof of \eqref{key:LAN:conv-Fisher}}
Recall that $\sigma^{2}_{n}(\xi,\mu) = \frac{1}{n}\left(\mathbf{X}_{n}-\mu\mathbf{1}_{n}\right)^{\top}\Sigma_{n}(s_{\xi}^{X})^{-1}\left(\mathbf{X}_{n}-\mu\mathbf{1}_{n}\right)$ and $\tilde{\sigma}^{2}_{n}(\xi)=\sigma^{2}_{n}(\xi,\mu_{0})$. 
Using the expressions of the score function in \eqref{expr:score}, the second order derivatives of the log-likelihood function $\ell_{n}(\vartheta)$ are given by
\begin{align}\label{expr:2nd-order-dv-log-likelihood}
    \begin{cases}
    &\partial_{\xi}^{2}\ell_{n}(\vartheta)= \frac{n}{2\sigma^{2}}\left( \partial_{\xi}^{2}\sigma^{2}_{n}(\xi,\mu)
    -\frac{\sigma^{2}}{n}\mathrm{Tr}\left[\Sigma_{n}(s_{\xi})^{-1}\Sigma_{n}(\partial_{\xi}^{2}s_{\xi})\right] \right)  
    +\frac{1}{2}\mathrm{Tr}\left[\Sigma_{n}(s_{\xi})^{-1}\Sigma_{n}(\partial_{\xi}s_{\xi})\Sigma_{n}(s_{\xi})^{-1}\Sigma_{n}(\partial_{\xi}s_{\xi})\right], \\
    &\partial_{\xi}\partial_{\sigma}\ell_{n}(\vartheta) = \partial_{\sigma}\partial_{\xi}\ell_{n}(\vartheta)
    = \frac{n}{\sigma^3}\partial_{\xi}\sigma^{2}_{n}(\xi,\mu),\ \ 
    \partial_{\sigma}^{2}\ell_{n}(\vartheta) = -\frac{n}{\sigma^{4}}\left( 3\sigma^{2}_{n}(\xi,\mu) - \sigma^{2} \right), \\ 
    &\partial_{\xi}\partial_{\mu}\ell_{n}(\vartheta) = \partial_{\mu}\partial_{\xi}\ell_{n}(\vartheta)
    = -\frac{1}{\sigma^{2}}\mathbf{1}_{n}^{\top}\partial_{\xi}\Sigma_{n}(s_{\xi}^{X})^{-1}\left(\mathbf{X}_{n}-\mu\mathbf{1}_{n}\right),\ \ 
    \partial_{\mu}^{2}\ell_{n}(\vartheta) = -\frac{1}{\sigma^{2}}(\mathbf{1}_{n}^{\top}\Sigma_{n}(s_{\xi}^{X})^{-1}\mathbf{1}_{n}), \\
    &\partial_{\sigma}\partial_{\mu}\ell_{n}(\vartheta) = \partial_{\mu}\partial_{\sigma}\ell_{n}(\vartheta)
    = -\frac{2}{\sigma^{3}}\mathbf{1}_{n}^{\top}\Sigma_{n}(s_{\xi}^{X})^{-1}\left(\mathbf{X}_{n}-\mu\mathbf{1}_{n}\right).
    \end{cases}
\end{align}
Then we can show \eqref{key:LAN:conv-Fisher} in the similar ways to the proof of Lemma~2.6 in \cite{Cohen-Gamboa-Lacaux-Loubes-2013} and Lemma~3.4 in \cite{Kawai-2013} using Lemma~\ref{Lemma:trace-app-inv} so we omit their detailed proofs. 

\subsection{Proof of \eqref{key:LAN:reminder}}
Using the Taylor theorem, it suffices to prove that
\begin{align}\label{key:LAN:reminder-suff1}
    n^{-\frac{1}{2}(r_1+r_2)-\frac{1}{2}(1-\alpha_{X}(\xi))r_3}\sup_{u\in\mathbb{U}_{n,c}(\vartheta)} 
    \biggl|\partial_{\xi_i}^{r_1}\partial_{\sigma}^{r_2}\partial_{\mu}^{r_3}\ell_{n}(\vartheta+\Phi_{n}(\vartheta)u)\biggr| = o_{\P_\vartheta^n}(1)\ \ \mbox{as $n\to\infty$}
\end{align}
for any $c>0$, $i=1,\ldots,p-1$ and $r_1,r_2,r_3\in\{1,2,3\}$ with $r_1+r_2+r_3=3$. 
Using the expressions of the second order derivatives of $\ell_n(\vartheta)$ in \eqref{expr:2nd-order-dv-log-likelihood}, we get the expressions of the third order derivatives of $\ell_n(\vartheta)$ by
\begin{align*}
    \partial_{\xi}^{3}\ell_{n}(\vartheta) =& \frac{n}{2\sigma^{2}}\partial_{\xi}^{3}\sigma^{2}_{n}(\xi,\mu)
    -\frac{1}{2}\mathrm{Tr}\left[\Sigma_{n}(s_{\xi})^{-1}\Sigma_{n}(\partial_{\xi}^{3}s_{\xi})\right]  \\
    &+ \frac{1}{2}\mathrm{Tr}\left[\Sigma_{n}(s_{\xi})^{-1}\Sigma_{n}(\partial_{\xi}s_{\xi})\Sigma_{n}(s_{\xi})^{-1}\Sigma_{n}(\partial_{\xi}^{2}s_{\xi})\right] +\frac{1}{2}\partial_{\xi}\mathrm{Tr}\left[\Sigma_{n}(s_{\xi})^{-1}\Sigma_{n}(\partial_{\xi}s_{\xi})\Sigma_{n}(s_{\xi})^{-1}\Sigma_{n}(\partial_{\xi}s_{\xi})\right]
\end{align*}
and
\begin{align*}
    \partial_{\xi}^{2}\partial_{\sigma}\ell_{n}(\vartheta) =& \frac{n}{\sigma^3}\partial_{\xi}^{2}\sigma^{2}_{n}(\xi,\mu),\ \ 
    \partial_{\xi}^{2}\partial_{\mu}\ell_{n}(\vartheta) = -\frac{1}{\sigma^{2}}\mathbf{1}_{n}^{\top}\partial_{\xi}^{2}\Sigma_{n}(s_{\xi}^{X})^{-1}\left(\mathbf{X}_{n}-\mu\mathbf{1}_{n}\right),\\ 
    \partial_{\sigma}^{2}\partial_{\xi}\ell_{n}(\vartheta) =& -\frac{3n}{\sigma^{4}}\partial_{\xi}\sigma^{2}_{n}(\xi,\mu),\ \ 
    \partial_{\sigma}^{2}\partial_{\mu}\ell_{n}(\vartheta) = \frac{2}{\sigma^{3}}(\mathbf{1}_{n}^{\top}\Sigma_{n}(s_{\xi}^{X})^{-1}\mathbf{1}_{n}), \\ 
    \partial_{\mu}^{2}\partial_{\xi}\ell_{n}(\vartheta) =& -\frac{1}{\sigma^{2}}(\mathbf{1}_{n}^{\top}\partial_{\xi}\Sigma_{n}(s_{\xi}^{X})^{-1}\mathbf{1}_{n}),\ \ \partial_{\mu}^{2}\partial_{\sigma}\ell_{n}(\vartheta) = \frac{2}{\sigma^{3}}(\mathbf{1}_{n}^{\top}\partial_{\xi}\Sigma_{n}(s_{\xi}^{X})^{-1}\mathbf{1}_{n}), \\
    \partial_{\xi}\partial_{\sigma}\partial_{\mu}\ell_{n}(\vartheta) =& -\frac{2}{\sigma^{3}}\mathbf{1}_{n}^{\top}\partial_{\xi}\Sigma_{n}(s_{\xi}^{X})^{-1}\left(\mathbf{X}_{n}-\mu\mathbf{1}_{n}\right),\ \ \partial_{\sigma}^{3}\ell_{n}(\vartheta)=\frac{2n}{\sigma^{5}}\left( 6\sigma^{2}_{n}(\xi,\mu) - \sigma^{3} \right),\ \ \partial_{\mu}^{3}\ell_{n}(\vartheta)=0
\end{align*}
so that \eqref{key:LAN:reminder-suff1} follows from the similar arguments to the proofs of Lemma~2.7 in \cite{Cohen-Gamboa-Lacaux-Loubes-2013} using Lemmas~\ref{Lemma:Key-Ineq-mu} and \ref{Lemma:trace-app-inv} as well as Theorems~4.1 and 5.2 in \cite{Adenstedt-1974}. This completes the proof of Theorem~\ref{Thm:LAN} as well as that of \eqref{key:LAN:reminder}.

\subsection{Alternative expression of MLE}
\label{Sec:mle_implement}
In this section, we derive an easily tractable alternative expression of the log-likelihood function $\ell_{n}(\vartheta)$, which is useful to quickly compute the exact MLE. 
Denote by $p_{\vartheta}(x_{1},\cdots,x_{n})$ the Gaussian likelihood function of the distribution $\P_{\vartheta}^{n}$ and by $p_{\vartheta}(x_{j}|x_{1},\cdots,x_{j-1})$ the conditional likelihood function of the distribution of $X_{j}^{\vartheta}$ conditional on the $j$-dimensional vector $(X_{1}^{\vartheta},X_{2}^{\vartheta},\cdots,X_{j-1}^{\vartheta})$. 
By expressing the likelihood function $p_{\vartheta}(x_{1},\cdots,x_{n})$ of the joint distribution as a product of the conditional likelihood functions $p_{\vartheta}(x_{j}|x_{1},\cdots,x_{j-1})$ and using the closed-form expression of their conditional Gaussian likelihood functions $p_{\vartheta}(x_{j}|x_{1},\cdots,x_{j-1})$, 
the log-likelihood function $\ell_{n}(\vartheta)$ can be expressed by
\begin{align*}
    \ell_{n}(\vartheta)=&\log p_{\vartheta}(X_{1},...,X_{n}),\notag \\
    =& \log p_{\vartheta}(X_{1}) \prod_{j=2}^n \log p_{\vartheta}(X_{j}|X_1,...,X_{j-1})\notag \\
    =&	-\frac{1}{2}\sum_{j=1}^{n}\log{v_{j}(\theta)}
        -\frac{1}{2}\sum_{j=1}^{n} \frac{(X_{j}-\eta_{j}(\vartheta))^{2}}{v_{j}(\theta)},
\end{align*}
	\bigskip where $\eta_{1}(\vartheta):=\mu$, $v_{1}(\theta):=\Var[X_{1}^{\vartheta}]=\sigma^{2}$, $\eta_{j}(\vartheta):=\mathbb{E}[X_{j}^{\vartheta}|\mathbf{X}_{j-1}]$ and  $%
	v_{j}(\theta) := \Var[X_{j}^{\vartheta}|\mathbf{X}_{j-1}]$ for $j\in\{2,3,\cdots,n\}$, which can be written as
	\begin{equation*}
		\eta_{j}(\vartheta) = \sum_{i=1}^{j-1}\phi _{j,i}(\xi)X_{j-i} + w_j(\xi) \mu
        \ \ \mbox{and}\ \ 
        v_{j}(\theta) = \gamma_{\theta}^{X}(0) \Pi _{i=1}^{j-1}\left( 1-\phi_{i,i}(\xi)^{2}\right) ,
	\end{equation*} 
    where $w_{1}(\xi):=1$, $w_{j}(\xi):=\left(1-\sum_{i=1}^{j-1} \phi_{j,i}(\xi)\right)$, and
	\begin{eqnarray*}
		\phi_{1,1}(\xi) &:=& \gamma_{\theta}^{X}(1)/\gamma_{\theta}^{X}(0),  \\
		\phi_{j,j}(\xi) &:=&\left[ \gamma_{\theta}^{X}\left( j\right) -\sum_{i=1}^{j-1}\phi_{j-1,i}(\xi)\gamma_{\theta}^{X}\left( j-i\right) \right] v_{j-1}(\theta)^{-1}, \\
		\phi_{j,i}(\xi) &:=&\phi_{j-1,i}(\xi)-\phi_{j,j}(\xi)\phi_{j-1,j-i}(\xi)
	\end{eqnarray*}
	for $i\in\{1,2,\ldots ,j-1\}$ and $j\in\{2,3,\cdots,n\}$. 
 Here, $\phi_{j,i}(\xi)$ are called the partial linear regression coefficients. 
 Notice that the above expression of $\ell_{n}(\vartheta)$ can be rewritten as
 \begin{align*}
     \ell_{n}(\vartheta) = -\frac{n}{2}\log{\sigma^{2}} -\frac{1}{2}\sum_{j=1}^{n}\log{\bar{v}_{j}(\xi)}
        -\frac{1}{2\sigma^{2}}\sum_{j=1}^{n} \frac{(Z_{j}(\xi)-w_{j}(\xi)\mu)^{2}}{\bar{v}_{j}(\xi)}
 \end{align*}
 using the notation $\bar{v}_{j}(\xi):=\sigma^{-2}v_{j}(\theta)$ and $Z_{1}(\xi):=X_{1}$, $Z_{j}(\xi):=X_{j}- \sum_{i=1}^{j-1}\phi _{j,i}(\xi)X_{j-i}$ for $j\in\{2,3,\cdots,n\}$. 
 Then we can see that the MLE also satisfies the estimating equations
 \begin{align*}
    \mu = \left( \sum_{j=1}^{n} \frac{ w_{j}(\xi)^{2} }{ \bar{v}_{j}(\xi) } \right)^{-1} 
    \sum_{j=1}^{n} \frac{ w_{j}(\xi) }{ \bar{v}_{j}(\xi) } Z_{j}(\xi)
    \ \ \mbox{and}\ \ 
    \sigma^{2} = \frac{1}{n}\sum_{j=1}^{n} \frac{ ( Z_{j}(\xi) - w_{j}(\xi) \mu ) ^{2} }{ \bar{v}_{j}(\xi) }
\end{align*}
for any $(\xi,\sigma,\mu)\in\Theta_{\xi}\times(0,\infty)\times\R$. 
Therefore, from the uniqueness of the maximum values of $\sigma$ and $\mu$ on $(0,\infty)$ and $\R$ respectively, we obtain the equalities
\begin{align*}
    \mu_{n}(\xi) = \left( \sum_{j=1}^{n} \frac{ w_{j}(\xi)^{2} }{ \bar{v}_{j}(\xi) } \right)^{-1} 
    \sum_{j=1}^{n} \frac{ w_{j}(\xi) }{ \bar{v}_{j}(\xi) } Z_{j}(\xi)
    \ \ \mbox{and}\ \ 
    \sigma_{n}^{2}(\xi,\mu) = \frac{1}{n}\sum_{j=1}^{n} \frac{ ( Z_{j}(\xi) - w_{j}(\xi) \mu ) ^{2} }{ \bar{v}_{j}(\xi) }.
\end{align*}
Notice that the coefficients $\phi_{j,i}(\xi)$ and the corresponding conditional variance $\bar{v}_{j}(\xi)$ can be calculated by the Durbin-Lenvinson recursive algorithm \citep[Chapter 5]{brockwell1987}.

\subsection{Robustness check: $\mu \neq 0$}
\label{Sec:robust_mu}
\begin{table}[htbp]
\caption{Bias and Std of alternative MLEs for $ARFIMA(0,d,0)$: $\mu=1$ and $\sigma=1$.}
\centering
\scalebox{0.82}{
    \begin{tabular}{cc|ccc|ccc|ccc}
        \hline \hline
        & & MLE1 & MLE2 & MLE3  & MLE1 & MLE2 & MLE3  & MLE1 & MLE2 & MLE3 \\
         \hline
         \multicolumn{11}{c}{$n=250$}\\
         \hline
        & & \multicolumn{3}{c|}{$d=$ -0.40} & \multicolumn{3}{c|}{$d=$ -0.30}  & \multicolumn{3}{c}{$d=$ -0.20}  \\
        \hline
        $\mu$ & Bias & 0.0000 & 0.0002 & 0.0000 & 0.0000 & 0.0007 & 0.0004 & 0.0000 & 0.0003 & 0.0000 \\
            & Std  & 0.0000 & 0.0096 & 0.0115 & 0.0000 & 0.0153 & 0.0165 & 0.0000 & 0.0235 & 0.0241 \\
   %         & RMSE & 0.0000 & 0.0097 & 0.0115 & 0.0000 & 0.0153 & 0.0165 & 0.0000 & 0.0235 & 0.0242 \\
            \multicolumn{11}{c}{}\\
        $d$ & Bias & -0.0021 & -0.0135 & -0.0074 & -0.0045 & -0.0175 & -0.0140 & -0.0056 & -0.0185 & -0.0170 \\
            & Std  & 0.0488 & 0.0502 & 0.0494 & 0.0507 & 0.0534 & 0.0520 & 0.0519 & 0.0545 & 0.0538 \\
   %         & RMSE & 0.0489 & 0.0521 & 0.0500 & 0.0510 & 0.0563 & 0.0539 & 0.0522 & 0.0576 & 0.0565 \\
            \multicolumn{11}{c}{}\\
      $\sigma$ & Bias & -0.0029 & -0.0051 & -0.0042 & -0.0024 & -0.0047 & -0.0043 & -0.0022 & -0.0044 & -0.0042 \\
         & Std  & 0.0445  & 0.0444  & 0.0444  & 0.0418  & 0.0420  & 0.0420  & 0.0454  & 0.0453  & 0.0453  \\
    %     & RMSE & 0.0446  & 0.0447  & 0.0446  & 0.0419  & 0.0422  & 0.0422  & 0.0454  & 0.0455  & 0.0455  \\
            \hline
        & & \multicolumn{3}{c|}{$d=$ -0.10} & \multicolumn{3}{c|}{$d=$ 0.00}  & \multicolumn{3}{c}{$d=$ 0.10}  \\
        \hline
        $\mu$ & Bias & 0.0000 & 0.0004 & 0.0003 & 0.0000 & -0.0006 & -0.0006 & 0.0000 & -0.0028 & -0.0027 \\
            & Std  & 0.0000 & 0.0380 & 0.0384 & 0.0000 & 0.0617 & 0.0616 & 0.0000 & 0.1082 & 0.1086 \\
    %        & RMSE & 0.0000 & 0.0380 & 0.0384 & 0.0000 & 0.0618 & 0.0617 & 0.0000 & 0.1084 & 0.1088 \\
            \multicolumn{11}{c}{}\\
        $d$ & Bias & -0.0044 & -0.0174 & -0.0168 & -0.0056 & -0.0194 & -0.0193 & -0.0045 & -0.0193 & -0.0193 \\
            & Std  & 0.0525 & 0.0548 & 0.0544 & 0.0498 & 0.0531 & 0.0529 & 0.0495 & 0.0528 & 0.0528 \\
    %        & RMSE & 0.0527 & 0.0576 & 0.0570 & 0.0502 & 0.0566 & 0.0564 & 0.0497 & 0.0563 & 0.0563 \\
            \multicolumn{11}{c}{}\\
       $\sigma$ & Bias & -0.0044 & -0.0065 & -0.0065 & -0.0013 & -0.0035 & -0.0035 & -0.0010 & -0.0033 & -0.0033 \\
         & Std  & 0.0440  & 0.0441  & 0.0441  & 0.0449  & 0.0450  & 0.0450  & 0.0453  & 0.0451  & 0.0451  \\
    %     & RMSE & 0.0442  & 0.0446  & 0.0445  & 0.0449  & 0.0451  & 0.0451  & 0.0453  & 0.0452  & 0.0452  \\
         \hline
        & & \multicolumn{3}{c|}{$d=$ 0.20} & \multicolumn{3}{c|}{$d=$ 0.30}  & \multicolumn{3}{c}{$d=$ 0.40}  \\
        \hline
        $\mu$ & Bias & 0.0000 & 0.0029 & 0.0035 & 0.0000 & -0.0248 & -0.0269 & 0.0000 & 0.0334 & 0.0359 \\
            & Std  & 0.0000 & 0.1948 & 0.1959 & 0.0000 & 0.3532 & 0.3551 & 0.0000 & 0.6428 & 0.6440 \\
    %        & RMSE & 0.0000 & 0.1951 & 0.1963 & 0.0000 & 0.3546 & 0.3566 & 0.0000 & 0.6446 & 0.6459 \\
            \multicolumn{11}{c}{}\\
        $d$ & Bias & -0.0031 & -0.0193 & -0.0193 & -0.0083 & -0.0255 & -0.0254 & -0.0135 & -0.0293 & -0.0291 \\
            & Std  & 0.0504 & 0.0536 & 0.0537 & 0.0492 & 0.0532 & 0.0532 & 0.0425 & 0.0474 & 0.0474 \\
         %   & RMSE & 0.0506 & 0.0571 & 0.0571 & 0.0500 & 0.0590 & 0.0591 & 0.0447 & 0.0557 & 0.0557 \\
            \multicolumn{11}{c}{}\\
     $\sigma$ & Bias & -0.0045 & -0.0067 & -0.0067 & -0.0050 & -0.0071 & -0.0071 & -0.0032 & -0.0048 & -0.0048 \\
         & Std  & 0.0443  & 0.0445  & 0.0445  & 0.0437  & 0.0438  & 0.0438  & 0.0432  & 0.0433  & 0.0433  \\
    %     & RMSE & 0.0445  & 0.0450  & 0.0450  & 0.0440  & 0.0444  & 0.0443  & 0.0433  & 0.0435  & 0.0435  \\
          \hline
         \multicolumn{11}{c}{$n=1000$}\\
         \hline
        & & \multicolumn{3}{c|}{$d=$ -0.40} & \multicolumn{3}{c|}{$d=$ -0.30}  & \multicolumn{3}{c}{$d=$ -0.20}  \\
        \hline
        $\mu$ & Bias & 0.0000 & -0.0000 & -0.0001 & 0.0000 & 0.0001 & 0.0003 & 0.0000 & 0.0004 & 0.0004 \\
            & Std  & 0.0000 & 0.0029 & 0.0035 & 0.0000 & 0.0051 & 0.0054 & 0.0000 & 0.0091 & 0.0093 \\
      %      & RMSE & 0.0000 & 0.0029 & 0.0035 & 0.0000 & 0.0051 & 0.0054 & 0.0000 & 0.0091 & 0.0093 \\
            \multicolumn{11}{c}{}\\
        $d$ & Bias & -0.0014 & -0.0052 & -0.0029 & -0.0001 & -0.0040 & -0.0030 & -0.0010 & -0.0048 & -0.0045 \\
            & Std  & 0.0249 & 0.0253 & 0.0252 & 0.0249 & 0.0254 & 0.0252 & 0.0246 & 0.0249 & 0.0248 \\
      %        & RMSE & 0.0250 & 0.0258 & 0.0254 & 0.0249 & 0.0257 & 0.0254 & 0.0246 & 0.0254 & 0.0252 \\
            \multicolumn{11}{c}{}\\
 $\sigma$ & Bias & -0.0008 & -0.0013 & -0.0011 & -0.0016 & -0.0021 & -0.0020 & 0.0000 & -0.0005 & -0.0005 \\
         & Std  & 0.0222  & 0.0222  & 0.0223  & 0.0222  & 0.0222  & 0.0222  & 0.0220 & 0.0221  & 0.0221  \\
     %      & RMSE & 0.0223  & 0.0223  & 0.0223  & 0.0223  & 0.0223  & 0.0223  & 0.0220 & 0.0221  & 0.0221  \\
             \hline
        & & \multicolumn{3}{c|}{$d=$ -0.10} & \multicolumn{3}{c|}{$d=$ 0.00}  & \multicolumn{3}{c}{$d=$ 0.10}  \\
        \hline
        $\mu$ & Bias & 0.0000 & 0.0001 & 0.0001 & 0.0000 & -0.0009 & -0.0010 & 0.0000 & -0.0026 & -0.0025 \\
            & Std  & 0.0000 & 0.0164 & 0.0164 & 0.0000 & 0.0311 & 0.0311 & 0.0000 & 0.0620 & 0.0622 \\
     %         & RMSE & 0.0000 & 0.0164 & 0.0164 & 0.0000 & 0.0311 & 0.0311 & 0.0000 & 0.0621 & 0.0623 \\
            \multicolumn{11}{c}{}\\
        $d$ & Bias & -0.0017 & -0.0055 & -0.0054 & -0.0027 & -0.0068 & -0.0068 & -0.0011 & -0.0055 & -0.0055 \\
            & Std  & 0.0240 & 0.0246 & 0.0245 & 0.0241 & 0.0248 & 0.0248 & 0.0236 & 0.0243 & 0.0243 \\
      %        & RMSE & 0.0241 & 0.0252 & 0.0251 & 0.0243 & 0.0257 & 0.0257 & 0.0237 & 0.0249 & 0.0249 \\
            \multicolumn{11}{c}{}\\
    $\sigma$ & Bias & -0.0006 & -0.0011 & -0.0011 & -0.0006 & -0.0011 & -0.0011 & -0.0009 & -0.0014 & -0.0014 \\
         & Std  & 0.0226  & 0.0226  & 0.0226  & 0.0229  & 0.0229  & 0.0229  & 0.0227  & 0.0227  & 0.0227  \\
     %      & RMSE & 0.0226  & 0.0226  & 0.0226  & 0.0229  & 0.0229  & 0.0229  & 0.0227  & 0.0227  & 0.0227  \\
            \hline
        & & \multicolumn{3}{c|}{$d=$ 0.20} & \multicolumn{3}{c|}{$d=$ 0.30}  & \multicolumn{3}{c}{$d=$ 0.40}  \\
        \hline
        $\mu$ & Bias & 0.0000 & -0.0092 & -0.0098 & 0.0000 & 0.0058 & 0.0047 & 0.0000 & -0.0328 & -0.0328 \\
            & Std  & 0.0000 & 0.1196 & 0.1204 & 0.0000 & 0.2468 & 0.2505 & 0.0000 & 0.5437 & 0.5466 \\
      %        & RMSE & 0.0000 & 0.1200 & 0.1208 & 0.0000 & 0.2468 & 0.2506 & 0.0000 & 0.5447 & 0.5476 \\
            \multicolumn{11}{c}{}\\
        $d$ & Bias & -0.0023 & -0.0067 & -0.0067 & -0.0014 & -0.0056 & -0.0056 & -0.0044 & -0.0087 & -0.0087 \\
            & Std  & 0.0256 & 0.0264 & 0.0264 & 0.0244 & 0.0251 & 0.0251 & 0.0226 & 0.0236 & 0.0236 \\
      %        & RMSE & 0.0257 & 0.0273 & 0.0273 & 0.0244 & 0.0258 & 0.0258 & 0.0230 & 0.0252 & 0.0252 \\
            \multicolumn{11}{c}{}\\
  $\sigma$ & Bias & -0.0001 & -0.0006 & -0.0006 & -0.0013 & -0.0017 & -0.0017 & -0.0011 & -0.0014 & -0.0014 \\
         & Std  & 0.0234  & 0.0234  & 0.0234  & 0.0215  & 0.0215  & 0.0215  & 0.0224  & 0.0224  & 0.0224  \\
     %      & RMSE & 0.0234  & 0.0234  & 0.0234  & 0.0216  & 0.0216  & 0.0216  & 0.0224  & 0.0225  & 0.0225  \\
            \hline
        \hline
    \end{tabular}
}
\end{table}

\begin{table}[htbp]
\caption{Bias, Std and RMSE of alternative MLEs for $ARFIMA(0,d,0)$: $\mu=-1$ and $\sigma=1$.}
\centering
\scalebox{0.82}{
    \begin{tabular}{cc|ccc|ccc|ccc}
        \hline \hline
        & & MLE1 & MLE2 & MLE3  & MLE1 & MLE2 & MLE3  & MLE1 & MLE2 & MLE3 \\
        \hline
          \multicolumn{11}{c}{$n=250$}\\
          \hline
        & & \multicolumn{3}{c|}{$d=$ -0.40} & \multicolumn{3}{c|}{$d=$ -0.30}  & \multicolumn{3}{c}{$d=$ -0.20}  \\
        \hline
        $\mu$ & Bias & 0.0000 & -0.0005 & -0.0006 & 0.0000 & 0.0007 & 0.0007 & 0.0000 & 0.0003 & 0.0003 \\
            & Std  & 0.0000 & 0.0097 & 0.0118 & 0.0000 & 0.0150 & 0.0159 & 0.0000 & 0.0241 & 0.0249 \\
   %         & RMSE & 0.0000 & 0.0097 & 0.0118 & 0.0000 & 0.0150 & 0.0159 & 0.0000 & 0.0242 & 0.0249 \\
            \multicolumn{11}{c}{}\\
        $d$ & Bias & -0.0055 & -0.0170 & -0.0109 & -0.0044 & -0.0166 & -0.0134 & -0.0046 & -0.0178 & -0.0163 \\
            & Std  & 0.0500 & 0.0518 & 0.0512 & 0.0538 & 0.0558 & 0.0546 & 0.0529 & 0.0555 & 0.0548 \\
    %        & RMSE & 0.0503 & 0.0546 & 0.0524 & 0.0540 & 0.0583 & 0.0563 & 0.0532 & 0.0583 & 0.0573 \\
            \multicolumn{11}{c}{}\\
    $\sigma$ & Bias & -0.0031 & -0.0054 & -0.0045 & -0.0032 & -0.0054 & -0.0050 & -0.0008 & -0.0031 & -0.0029 \\
         & Std  & 0.0452  & 0.0451  & 0.0452  & 0.0462  & 0.0460  & 0.0460  & 0.0439  & 0.0440  & 0.0440  \\
  %       & RMSE & 0.0453  & 0.0454  & 0.0454  & 0.0463  & 0.0463  & 0.0463  & 0.0439  & 0.0441  & 0.0441  \\
            \hline
        & & \multicolumn{3}{c|}{$d=$ -0.10} & \multicolumn{3}{c|}{$d=$ 0.00}  & \multicolumn{3}{c}{$d=$ 0.10}  \\
        \hline
        $\mu$ & Bias & 0.0000 & 0.0013 & 0.0012 & 0.0000 & -0.0017 & -0.0016 & 0.0000 & 0.0030 & 0.0028 \\
            & Std  & 0.0000 & 0.0395 & 0.0397 & 0.0000 & 0.0623 & 0.0622 & 0.0000 & 0.1116 & 0.1120 \\
  %          & RMSE & 0.0000 & 0.0396 & 0.0398 & 0.0000 & 0.0624 & 0.0623 & 0.0000 & 0.1118 & 0.1122 \\
            \multicolumn{11}{c}{}\\
        $d$ & Bias & -0.0006 & -0.0148 & -0.0143 & -0.0052 & -0.0192 & -0.0191 & -0.0039 & -0.0197 & -0.0198 \\
            & Std  & 0.0518 & 0.0557 & 0.0553 & 0.0485 & 0.0525 & 0.0523 & 0.0501 & 0.0538 & 0.0538 \\
    %        & RMSE & 0.0519 & 0.0578 & 0.0572 & 0.0488 & 0.0559 & 0.0558 & 0.0503 & 0.0574 & 0.0574 \\
            \multicolumn{11}{c}{}\\
       $\sigma$ & Bias & -0.0014 & -0.0038 & -0.0037 & -0.0006 & -0.0029 & -0.0029 & -0.0027 & -0.0051 & -0.0051 \\
         & Std  & 0.0436  & 0.0436  & 0.0436  & 0.0446  & 0.0447  & 0.0447  & 0.0454  & 0.0453  & 0.0453  \\
    %     & RMSE & 0.0436  & 0.0438  & 0.0437  & 0.0447  & 0.0448  & 0.0448  & 0.0455  & 0.0456  & 0.0456  \\
            \hline
        & & \multicolumn{3}{c|}{$d=$ 0.20} & \multicolumn{3}{c|}{$d=$ 0.30}  & \multicolumn{3}{c}{$d=$ 0.40}  \\
        \hline
        $\mu$ & Bias & 0.0000 & 0.0066 & 0.0063 & 0.0000 & -0.0070 & -0.0048 & 0.0000 & 0.0471 & 0.0471 \\
            & Std  & 0.0000 & 0.1900 & 0.1906 & 0.0000 & 0.3520 & 0.3565 & 0.0000 & 0.6664 & 0.6697 \\
    %        & RMSE & 0.0000 & 0.1904 & 0.1909 & 0.0000 & 0.3526 & 0.3571 & 0.0000 & 0.6691 & 0.6723 \\
            \multicolumn{11}{c}{}\\
        $d$ & Bias & -0.0047 & -0.0213 & -0.0213 & -0.0062 & -0.0235 & -0.0234 & -0.0129 & -0.0297 & -0.0295 \\
            & Std  & 0.0484 & 0.0535 & 0.0536 & 0.0453 & 0.0498 & 0.0498 & 0.0415 & 0.0464 & 0.0464 \\
   %         & RMSE & 0.0487 & 0.0577 & 0.0577 & 0.0458 & 0.0551 & 0.0551 & 0.0436 & 0.0551 & 0.0551 \\
            \multicolumn{11}{c}{}\\
    $\sigma$ & Bias & -0.0028 & -0.0051 & -0.0051 & -0.0044 & -0.0065 & -0.0065 & -0.0021 & -0.0038 & -0.0037 \\
         & Std  & 0.0444  & 0.0446  & 0.0446  & 0.0446  & 0.0446  & 0.0446  & 0.0446  & 0.0446  & 0.0446  \\
    %     & RMSE & 0.0445  & 0.0449  & 0.0449  & 0.0448  & 0.0451  & 0.0451  & 0.0446  & 0.0448  & 0.0448  \\
            \hline
  \multicolumn{11}{c}{$n=1000$}\\
  \hline
        & & \multicolumn{3}{c|}{$d=$ -0.40} & \multicolumn{3}{c|}{$d=$ -0.30}  & \multicolumn{3}{c}{$d=$ -0.20}  \\
        \hline
        $\mu$ & Bias & 0.0000 & -0.0000 & 0.0000 & 0.0000 & -0.0002 & -0.0002 & 0.0000 & -0.0003 & -0.0003 \\
            & Std  & 0.0000 & 0.0028 & 0.0034 & 0.0000 & 0.0051 & 0.0055 & 0.0000 & 0.0090 & 0.0093 \\
  %          & RMSE & 0.0000 & 0.0028 & 0.0034 & 0.0000 & 0.0051 & 0.0055 & 0.0000 & 0.0090 & 0.0093 \\
            \multicolumn{11}{c}{}\\
        $d$ & Bias & -0.0036 & -0.0074 & -0.0052 & -0.0019 & -0.0058 & -0.0048 & -0.0011 & -0.0048 & -0.0045 \\
            & Std  & 0.0263 & 0.0270 & 0.0266 & 0.0251 & 0.0256 & 0.0254 & 0.0249 & 0.0252 & 0.0251 \\
   %         & RMSE & 0.0265 & 0.0279 & 0.0271 & 0.0252 & 0.0262 & 0.0259 & 0.0249 & 0.0256 & 0.0255 \\
            \multicolumn{11}{c}{}\\
      $\sigma$ & Bias & 0.0001 & -0.0005 & -0.0002 & -0.0009 & -0.0015 & -0.0014 & -0.0007 & -0.0012 & -0.0011 \\
         & Std  & 0.0232 & 0.0232  & 0.0232  & 0.0224  & 0.0224  & 0.0224  & 0.0223  & 0.0223  & 0.0223  \\
   %      & RMSE & 0.0232 & 0.0232  & 0.0232  & 0.0224  & 0.0224  & 0.0224  & 0.0223  & 0.0224  & 0.0224  \\
           \hline
        & & \multicolumn{3}{c|}{$d=$ -0.10} & \multicolumn{3}{c|}{$d=$ 0.00}  & \multicolumn{3}{c}{$d=$ 0.10}  \\
        \hline
        $\mu$ & Bias & 0.0000 & 0.0007 & 0.0007 & 0.0000 & -0.0023 & -0.0023 & 0.0000 & -0.0021 & -0.0020 \\
            & Std  & 0.0000 & 0.0169 & 0.0169 & 0.0000 & 0.0314 & 0.0314 & 0.0000 & 0.0608 & 0.0608 \\
   %         & RMSE & 0.0000 & 0.0169 & 0.0169 & 0.0000 & 0.0315 & 0.0315 & 0.0000 & 0.0609 & 0.0609 \\
            \multicolumn{11}{c}{}\\
        $d$ & Bias & -0.0010 & -0.0051 & -0.0050 & -0.0030 & -0.0071 & -0.0071 & -0.0016 & -0.0058 & -0.0058 \\
            & Std  & 0.0252 & 0.0256 & 0.0256 & 0.0237 & 0.0243 & 0.0243 & 0.0252 & 0.0260 & 0.0260 \\
    %        & RMSE & 0.0252 & 0.0261 & 0.0261 & 0.0239 & 0.0253 & 0.0253 & 0.0253 & 0.0266 & 0.0267 \\
            \multicolumn{11}{c}{}\\
     $\sigma$ & Bias & -0.0005 & -0.0011 & -0.0011 & -0.0002 & -0.0007 & -0.0007 & -0.0019 & -0.0024 & -0.0024 \\
         & Std  & 0.0219  & 0.0219  & 0.0219  & 0.0221  & 0.0220  & 0.0220  & 0.0230  & 0.0229  & 0.0229  \\
   %      & RMSE & 0.0219  & 0.0220  & 0.0220  & 0.0221  & 0.0221  & 0.0221  & 0.0230  & 0.0231  & 0.0231  \\
             \hline
        & & \multicolumn{3}{c|}{$d=$ 0.20} & \multicolumn{3}{c|}{$d=$ 0.30}  & \multicolumn{3}{c}{$d=$ 0.40}  \\
        \hline
        $\mu$ & Bias & 0.0000 & -0.0075 & -0.0081 & 0.0000 & -0.0035 & -0.0022 & 0.0000 & 0.0050 & 0.0069 \\
            & Std  & 0.0000 & 0.1234 & 0.1242 & 0.0000 & 0.2510 & 0.2534 & 0.0000 & 0.5516 & 0.5593 \\
    %        & RMSE & 0.0000 & 0.1236 & 0.1245 & 0.0000 & 0.2510 & 0.2534 & 0.0000 & 0.5516 & 0.5593 \\
            \multicolumn{11}{c}{}\\
        $d$ & Bias & -0.0005 & -0.0050 & -0.0050 & -0.0027 & -0.0073 & -0.0072 & -0.0034 & -0.0077 & -0.0077 \\
            & Std  & 0.0244 & 0.0251 & 0.0251 & 0.0243 & 0.0253 & 0.0253 & 0.0232 & 0.0243 & 0.0243 \\
    %        & RMSE & 0.0244 & 0.0256 & 0.0256 & 0.0244 & 0.0264 & 0.0264 & 0.0235 & 0.0255 & 0.0255 \\
            \multicolumn{11}{c}{}\\
       $\sigma$ & Bias & -0.0012 & -0.0017 & -0.0017 & -0.0004 & -0.0008 & -0.0008 & 0.0001 & -0.0002 & -0.0002 \\
         & Std  & 0.0217  & 0.0217  & 0.0217  & 0.0227  & 0.0227  & 0.0227  & 0.0216 & 0.0216  & 0.0216  \\
   %      & RMSE & 0.0217  & 0.0218  & 0.0218  & 0.0227  & 0.0227  & 0.0227  & 0.0216 & 0.0216  & 0.0216  \\
            \hline
        \hline
    \end{tabular}
}
\end{table} 
\newpage
\subsection{An Financial Application}
\label{Sec:fore_fou}

Fractional models have been employed to model realized volatility (RV) \citep{Andersen-Bollerslev-Diebold-Labys-2003,  Bennedsen-Lunde-Pakkanen-2022, Bennedsen-2024} and trading volume \citep{Shi-Yu-Zhang-2024-b, Wang-Xiao-Yu-Zhang-2023}. One recently introduced model to the volatility literature is the fOU process \citep{Gatheral-Jaisson-Rosenbaum-2018, Wang-Xiao-Yu-2023-JoE}. For instance, \cite{Wang-Xiao-Yu-2023-JoE} demonstrate that fOU outperforms traditional discrete-time fractional models such as ARFIMA$(1,d,0)$ and fBm. In this section, we demonstrate that our exact MLE further enhances the forecasting accuracy of fOU. Compared to the existing literature, we also incorporate the CoF method from \cite{Wang-Xiao-Yu-2023-JoE}. However, since the CoF method is less efficient than the AWML approach proposed by \cite{Shi-Yu-Zhang-2024} and \cite{Wang-Xiao-Yu-Zhang-2023}, we anticipate that forecasting accuracy will rank as follows: MLE2 (exact MLE), followed by MLE3 (plug-in MLE), and then CoF. 

We apply our method to  Dow Jones 30 (DJ30) stocks. The daily realized volatilities of these DJ30 stocks, spanning September 15, 2012, to August 28, 2021, are sourced from Risk Lab. We assume that the log RV follows a fOU process and set $\Delta = 1/250$ to reflect 250 trading days per year. A four-year rolling window is employed to fit the model, estimate the parameters, and generate $h$-day-ahead forecasts of RV. The results, presented in Tables~\ref{fou_fore_com1}-\ref{fou_fore_com3}, align with our expectations. Both MLE approaches outperform the CoF method, with improvements ranging from approximately 2\% to 15\%. The exact MLE slightly enhances the performance of the plug-in MLE. It is not surprising that we find the plug-in MLE exhibits good finite sample performance compared to our exact MLE.

\begin{table}[H]
\centering
\caption{RMSE of the alternative estimation methods for h-day-ahead forecasts for realized volatility (RV) with four-year rolling window between September 15, 2012 and August 28, 2021.}
\begin{tabular}{ccccccc}
   \hline \hline
     &            & $h=1$    & $h=2$    & $h=3$    & $h=4$   & $h=5$    \\
        \hline
\multirow{3}{*}{AAPL}  & MLE2  & 0.0611 & 0.0715 & 0.0782 & 0.0836 & 0.0877 \\
     & MLE3 & 0.0612 & 0.0716 & 0.0783 & 0.0836 & 0.0878 \\
     & CoF        & 0.0639 & 0.0779 & 0.0874 & 0.0950 & 0.1007 \\
     &            &        &        &        &        &        \\
\multirow{3}{*}{ALD}  &   MLE2          & 0.0543 & 0.0635 & 0.0712 & 0.0770 & 0.0816 \\
     &   MLE3         & 0.0545 & 0.0638 & 0.0716 & 0.0774 & 0.0820 \\
     &  CoF           & 0.0555 & 0.0661 & 0.0749 & 0.0813 & 0.0863\\
          &            &        &        &        &        &        \\
\multirow{3}{*}{AMGN}                  & MLE2  & 0.0691 & 0.0770 & 0.0819 & 0.0871 & 0.0909 \\
& MLE3 & 0.0692 & 0.0771 & 0.0820 & 0.0872 & 0.0909 \\
& CoF        & 0.0703 & 0.0796 & 0.0853 & 0.0911 & 0.0954 \\
          &            &        &        &        &        &        \\
\multirow{3}{*}{AXP}                & MLE2  & 0.0577 & 0.0693 & 0.0774 & 0.0842 & 0.0895 \\
 & MLE3 & 0.0580 & 0.0697 & 0.0778 & 0.0847 & 0.0900 \\
 & CoF        & 0.0605 & 0.0757 & 0.0861 & 0.0944 & 0.1006 \\
         &            &        &        &        &        &        \\
 \multirow{3}{*}{BA}                & MLE2  & 0.0922 & 0.1081 & 0.1190 & 0.1287 & 0.1371 \\
   & MLE3 & 0.0925 & 0.1084 & 0.1194 & 0.1292 & 0.1376 \\
 & CoF        & 0.0965 & 0.1192 & 0.1352 & 0.1483 & 0.1587 \\
       &            &        &        &        &        &        \\
\multirow{3}{*}{BEL} & MLE2 & 0.0485 & 0.0560 & 0.0616 & 0.0665 & 0.0706 \\
                     & MLE3 & 0.0486 & 0.0561 & 0.0618 & 0.0667 & 0.0707 \\
                     & CoF  & 0.0497 & 0.0587 & 0.0654 & 0.0708 & 0.0750\\
                         &            &        &        &        &        &        \\
\multirow{3}{*}{CAT} & MLE2 & 0.0595 & 0.0675 & 0.0733 & 0.0783 & 0.0822 \\
                     & MLE3 & 0.0596 & 0.0677 & 0.0735 & 0.0785 & 0.0824 \\
                     & CoF  & 0.0594 & 0.0687 & 0.0754 & 0.0813 & 0.0858\\
                         &            &        &        &        &        &        \\
\multirow{3}{*}{CHV} & MLE2 & 0.0525 & 0.0638 & 0.0725 & 0.0789 & 0.0847 \\
                     & MLE3 & 0.0527 & 0.0642 & 0.0730 & 0.0795 & 0.0853 \\
                     & CoF  & 0.0532 & 0.0650 & 0.0742 & 0.0810 & 0.0868\\
                         &            &        &        &        &        &        \\
\multirow{3}{*}{CRM} & MLE2 & 0.0630 & 0.0739 & 0.0807 & 0.0861 & 0.0905 \\
                     & MLE3 & 0.0631 & 0.0739 & 0.0808 & 0.0861 & 0.0905 \\
                     & CoF  & 0.0663 & 0.0817 & 0.0927 & 0.1015 & 0.1086\\
                         &            &        &        &        &        &        \\
\multirow{3}{*}{CSCO} & MLE2 & 0.0537 & 0.0635 & 0.0714 & 0.0774 & 0.0821 \\
                      & MLE3 & 0.0538 & 0.0637 & 0.0715 & 0.0775 & 0.0822 \\
                      & CoF  & 0.0554 & 0.0675 & 0.0771 & 0.0842 & 0.0896\\
     \hline \hline
\end{tabular}
\label{fou_fore_com1}
\end{table}

      \begin{table}[H]
\centering
\caption{RMSE of the alternative estimation methods for h-day-ahead forecasts for realized volatility (RV) with four-year rolling window between September 15, 2012 and August 28, 2021.}
\begin{tabular}{ccccccc}
   \hline \hline
     &            & $h=1$    & $h=2$    & $h=3$    & $h=4$   & $h=5$    \\
        \hline                
\multirow{3}{*}{DIS} & MLE2 & 0.0556 & 0.0657 & 0.0739 & 0.0801 & 0.0853 \\
                     & MLE3 & 0.0558 & 0.0660 & 0.0743 & 0.0805 & 0.0857 \\
                     & CoF  & 0.0579 & 0.0705 & 0.0804 & 0.0876 & 0.0934\\
                         &            &        &        &        &        &        \\
\multirow{3}{*}{GS} & MLE2 & 0.0513 & 0.0630 & 0.0709 & 0.0772 & 0.0821 \\
                    & MLE3 & 0.0513 & 0.0631 & 0.0710 & 0.0773 & 0.0822 \\
                    & CoF  & 0.0526 & 0.0660 & 0.0751 & 0.0821 & 0.0873\\
                        &            &        &        &        &        &        \\
\multirow{3}{*}{HD} & MLE2 & 0.0538 & 0.0628 & 0.0695 & 0.0754 & 0.0804 \\
                    & MLE3 & 0.0540 & 0.0630 & 0.0697 & 0.0757 & 0.0807 \\
                    & CoF  & 0.0549 & 0.0656 & 0.0736 & 0.0804 & 0.0859\\
                        &            &        &        &        &        &        \\
\multirow{3}{*}{IBM} & MLE2 & 0.0499 & 0.0583 & 0.0642 & 0.0691 & 0.0732 \\
                     & MLE3 & 0.0500 & 0.0585 & 0.0644 & 0.0694 & 0.0734 \\
                     & CoF  & 0.0519 & 0.0632 & 0.0713 & 0.0777 & 0.0826\\
                         &            &        &        &        &        &        \\
\multirow{3}{*}{INTC} & MLE2 & 0.0634 & 0.0742 & 0.0823 & 0.0884 & 0.0930 \\
                      & MLE3 & 0.0635 & 0.0744 & 0.0824 & 0.0885 & 0.0931 \\
                      & CoF  & 0.0648 & 0.0777 & 0.0872 & 0.0943 & 0.0994\\
                          &            &        &        &        &        &        \\
\multirow{3}{*}{JNJ} & MLE2 & 0.0533 & 0.0605 & 0.0665 & 0.0714 & 0.0751 \\
                     & MLE3 & 0.0534 & 0.0606 & 0.0667 & 0.0716 & 0.0753 \\
                     & CoF  & 0.0546 & 0.0620 & 0.0680 & 0.0729 & 0.0765\\
                         &            &        &        &        &        &        \\
\multirow{3}{*}{JPM} & MLE2 & 0.0519 & 0.0639 & 0.0730 & 0.0803 & 0.0858 \\
                     & MLE3 & 0.0521 & 0.0641 & 0.0732 & 0.0806 & 0.0861 \\
                     & CoF  & 0.0532 & 0.0668 & 0.0768 & 0.0846 & 0.0903\\
                         &            &        &        &        &        &        \\
\multirow{3}{*}{KO} & MLE2 & 0.0444 & 0.0529 & 0.0597 & 0.0648 & 0.0689 \\
                    & MLE3 & 0.0446 & 0.0531 & 0.0600 & 0.0650 & 0.0691 \\
                    & CoF  & 0.0458 & 0.0558 & 0.0637 & 0.0693 & 0.0738\\
                        &            &        &        &        &        &        \\
\multirow{3}{*}{MCD} & MLE2 & 0.0485 & 0.0583 & 0.0662 & 0.0725 & 0.0776 \\
                     & MLE3 & 0.0488 & 0.0586 & 0.0666 & 0.0728 & 0.0779 \\
                     & CoF  & 0.0509 & 0.0629 & 0.0722 & 0.0791 & 0.0845\\
                         &            &        &        &        &        &        \\
\multirow{3}{*}{MMM} & MLE2 & 0.0492 & 0.0571 & 0.0628 & 0.0675 & 0.0709 \\
                     & MLE3 & 0.0493 & 0.0572 & 0.0629 & 0.0677 & 0.0711 \\
                     & CoF  & 0.0496 & 0.0592 & 0.0664 & 0.0723 & 0.0766\\
                          \hline \hline
\end{tabular}
\label{fou_fore_com2}
\end{table}

      \begin{table}[H]
\centering
\caption{RMSE of the alternative estimation methods for h-day-ahead forecasts for realized volatility (RV) with four-year rolling window between September 15, 2012 and August 28, 2021.}
\begin{tabular}{ccccccc}
   \hline \hline
        &            & $h=1$    & $h=2$    & $h=3$    & $h=4$   & $h=5$    \\
        \hline    
\multirow{3}{*}{MRK} & MLE2 & 0.0565 & 0.0648 & 0.0714 & 0.0760 & 0.0797 \\
                     & MLE3 & 0.0566 & 0.0650 & 0.0716 & 0.0762 & 0.0799 \\
                     & CoF  & 0.0584 & 0.0694 & 0.0776 & 0.0830 & 0.0871\\
                         &            &        &        &        &        &        \\
 \multirow{3}{*}{MSFT} & MLE2 & 0.0527 & 0.0631 & 0.0711 & 0.0770 & 0.0819 \\
                      & MLE3 & 0.0528 & 0.0633 & 0.0713 & 0.0771 & 0.0821 \\
                      & CoF  & 0.0531 & 0.0646 & 0.0737 & 0.0802 & 0.0857\\
                          &            &        &        &        &        &        \\
\multirow{3}{*}{NIKE} & MLE2 & 0.0565 & 0.0671 & 0.0746 & 0.0808 & 0.0854 \\
                      & MLE3 & 0.0566 & 0.0673 & 0.0749 & 0.0811 & 0.0857 \\
                      & CoF  & 0.0582 & 0.0717 & 0.0810 & 0.0882 & 0.0932\\
                          &            &        &        &        &        &        \\
\multirow{3}{*}{PG} & MLE2 & 0.0547 & 0.0649 & 0.0737 & 0.0806 & 0.0861 \\
                    & MLE3 & 0.0549 & 0.0651 & 0.0740 & 0.0809 & 0.0864 \\
                    & CoF  & 0.0560 & 0.0673 & 0.0769 & 0.0841 & 0.0897\\
                        &            &        &        &        &        &        \\
\multirow{3}{*}{SPC} & MLE2 & 0.0546 & 0.0643 & 0.0715 & 0.0770 & 0.0818 \\
                     & MLE3 & 0.0548 & 0.0645 & 0.0718 & 0.0774 & 0.0822 \\
                     & CoF  & 0.0559 & 0.0665 & 0.0743 & 0.0803 & 0.0853\\
                         &            &        &        &        &        &        \\
\multirow{3}{*}{UNH} & MLE2 & 0.0570 & 0.0665 & 0.0728 & 0.0786 & 0.0834 \\
                     & MLE3 & 0.0571 & 0.0667 & 0.0731 & 0.0790 & 0.0838 \\
                     & CoF  & 0.0577 & 0.0682 & 0.0752 & 0.0816 & 0.0869\\
                         &            &        &        &        &        &        \\
\multirow{3}{*}{V} & MLE2 & 0.0468 & 0.0570 & 0.0648 & 0.0709 & 0.0760 \\
                   & MLE3 & 0.0470 & 0.0572 & 0.0651 & 0.0712 & 0.0763 \\
                   & CoF  & 0.0476 & 0.0586 & 0.0669 & 0.0733 & 0.0786\\
                       &            &        &        &        &        &        \\
  \multirow{3}{*}{WAG} & MLE2 & 0.0727 & 0.0826 & 0.0886 & 0.0937 & 0.0983 \\
                     & MLE3 & 0.0727 & 0.0827 & 0.0887 & 0.0938 & 0.0984 \\
                     & CoF  & 0.0757 & 0.0893 & 0.0980 & 0.1047 & 0.1104\\
                         &            &        &        &        &        &        \\
 \multirow{3}{*}{WMT} & MLE2 & 0.0500 & 0.0580 & 0.0635 & 0.0677 & 0.0716 \\
                     & MLE3 & 0.0502 & 0.0582 & 0.0637 & 0.0680 & 0.0718 \\
                     & CoF  & 0.0518 & 0.0623 & 0.0692 & 0.0742 & 0.0783\\
                         &            &        &        &        &        &        \\
\multirow{3}{*}{XOM} & MLE2 & 0.0520 & 0.0619 & 0.0698 & 0.0763 & 0.0810 \\
                     & MLE3 & 0.0522 & 0.0622 & 0.0702 & 0.0768 & 0.0816 \\
                     & CoF  & 0.0534 & 0.0655 & 0.0752 & 0.0832 & 0.0891\\            
     \hline \hline
\end{tabular}
\label{fou_fore_com3}
\end{table}
\end{document}